\documentclass[reqno,a4paper,11pt]{amsart}

\usepackage[english]{babel}

\usepackage{amsfonts}
\usepackage{amsmath}
\usepackage{amsthm}
\usepackage{amssymb}
\usepackage{indentfirst}
\usepackage{xcolor}
\usepackage{tabularx}
\usepackage{graphicx}
\usepackage{esint}
\usepackage{dsfont}
\usepackage{bm}

\allowdisplaybreaks

\usepackage{soul}

\usepackage[a4paper,top=3cm,bottom=3cm,left=2.5cm,right=2.5cm]{geometry}

\usepackage{subcaption}

\usepackage{mathrsfs}
\usepackage{mathtools}
\usepackage{enumitem}
\usepackage[normalem]{ulem}

\usepackage{empheq}
\usepackage[most]{tcolorbox}

\usepackage{framed}

\usepackage[utf8]{inputenc}
\usepackage[T1]{fontenc}
\usepackage[english]{babel}

\numberwithin{equation}{section}

\theoremstyle{plain}
\begingroup
\theoremstyle{plain}
\newtheorem{theorem}{Theorem}[section]
\newtheorem{corollary}[theorem]{Corollary}

\newtheorem{lemma}[theorem]{Lemma}
\theoremstyle{definition}
\newtheorem{definition}[theorem]{Definition}

\theoremstyle{remark}
\newtheorem{remark}[theorem]{Remark}
\endgroup

\theoremstyle{definition}
\theoremstyle{remark}

\mathchardef\emptyset="001F

\newcommand{\R}{\mathbb{R}}

\newcommand{\spt}{\mathrm{spt}}

\newcommand{\di}{\mathrm{d}}

\newcommand{\Om}{\Omega}

\newcommand{\F}{\mathcal{F}}

\newcommand{\Lip}{\mathrm{Lip}}

\newcommand{\spann}{\mathrm{span}}

\definecolor{dred}{rgb}{.8,0,0}
\definecolor{ddmagenta}{rgb}{0.7,0,0.9}
\definecolor{ddcyan}{rgb}{0,0.2,1.0}
\definecolor{Orchid}{rgb}{0.7,0.4,0}
\definecolor{blue_links}{RGB}{13,0,180} 
\definecolor{lightblue}{RGB}{0.9,0.9,1}

\usepackage{hyperref}
\hypersetup{
    colorlinks=true, 
    linktoc=all,     
    linkcolor=blue_links,  
    citecolor=blue_links,
    urlcolor=blue_links,
}

\newcommand{\eeta}{\boldsymbol{\eta}}
\newcommand{\Pp}{\mathcal{P}}

\newcommand{\Vv}{\mathcal{V}}

\makeatletter
\@namedef{subjclassname@2020}{%
  \textup{2020} Mathematics Subject Classification}
\makeatother

\title[Scaling limits for Moran processes on metric strategy spaces]{Scaling limits for Moran processes on metric strategy spaces: an Eulerian derivation of pure replicator and Fleming--Viot measure-valued PDEs}

\author[S. Almi]{Stefano Almi}
\address[Stefano Almi]{Department of Mathematics and Applications ``R.~Caccioppoli'', University of Naples Federico II, Via Cintia, Monte S. Angelo, 80126 Napoli, Italy.}
\email{stefano.almi@unina.it}

\author[R. Durastanti]{Riccardo Durastanti}
\address[Riccardo Durastanti]{Department of Mathematics and Applications ``R.~Caccioppoli'', University of Naples Federico II, Via Cintia, Monte S. Angelo, 80126 Napoli, Italy.}
\email{riccardo.durastanti@unina.it}

\author[M.~Morandotti]{Marco Morandotti}
\address[Marco Morandotti]{Dipartimento di Scienze Matematiche ``G.~L.~Lagrange'', Politecnico di Torino, Corso Duca degli Abruzzi 24, 10129 Torino, Italy.}
\email{marco.morandotti@polito.it}

\author[G.~Orlando]{Gianluca Orlando}
\address[Gianluca Orlando]{Dipartimento di Meccanica, Matematica e Management, Politecnico di Bari, via E.~Orabona 4, 70125 Bari, Italy}
\email{gianluca.orlando@poliba.it}

\author[F.~Solombrino]{Francesco Solombrino}
\address[Francesco Solombrino]{Dipartimento di Scienze e Tecnologie Biologiche ed Ambientali Centro Ecotekne, via Lecce-Monteroni, 73047 Lecce, Italy}
\email{francesco.solombrino@unisalento.it}


\begin{document}

\begin{abstract}
We study the large-population limit of a discrete-time Moran process featuring multiple strategies, drawn from a possibly infinite strategy space~$\Vv$ (a metric space), under both weak and strong selection. In the Eulerian density formulation, the limiting dynamics depend critically on the relative scaling of population size, mutation rate, and selection intensity. Depending on these scalings, the limit behavior is governed either by a purely deterministic replicator-type continuity equation or by a diffusion-enhanced PDE. In the latter case, the diffusion operator recovers the classical Fleming–Viot operator and the Kimura equation as special instances. Methodologically, we derive the limit by reinterpreting the discrete process in Eulerian coordinates, constructing interpolating curves that satisfy an approximate PDE, and establishing convergence via a compactness argument in a suitable topology on the space of probability measures over probabilities over $\Vv$.
\end{abstract}

\maketitle

\bigskip

\section{Introduction}
\label{s:intro}

Predicting the macroscopic behavior of large systems of interacting agents is crucial in a number of scientific areas and applications, including biology~\cite{SmPr, TaJo}, population dynamics and opinion formation~\cite{DMPW, Toscani}, optimization and robotics~\cite{ABMS, Bullo2009, CCCBO, robots, perea, Pinnau}, artificial intelligence~\cite{Holland}, and economics~\cite{DMPW, JR}. In such systems, a set of strategies (for instance, genes within a population or portfolios in stock markets) is often provided that agents may choose and change over time, in order to optimize the individual outcome. Evolutionary games offer a solid mathematical framework to describe this kind of competing dynamics, going beyond the static notion of Nash equilibria~\cite{Nash}: strategies are selected by each agent based on {\em (i)} interactions among themselves and with the environment and {\em (ii)} certain instantaneous optimization principles.

From the biology point of view, evolutionary games based on fitness embody the Darwinian principle of natural selection~\cite{darwin1859origin}, which drives the evolution at the macroscopic scale. At the microscopic scale, instead, Kimura's neutral theory of molecular evolution~\cite{Kimura2} postulates that most evolutionary changes at the molecular level are selectively neutral: evolution is largely driven by random chance (genetic drift diffusion). A relevant analysis goal is to bridge these two scales: the interplay of selection and diffusion effects will emerge in a prototypical discrete evolutionary process through the identification of a suitable underlying time scale. To clarify our goal, we start by reviewing the two involved evolutionary mechanisms, namely the replicator dynamics (a time-continuous dynamics) and the Moran process (a time-discrete process).

\medskip

\paragraph{\bf Replicator dynamics and the Moran process.}
The replicator dynamics is a paramount example of an evolutionary game, where strategies replicate aiming at instantaneous maximization of a given payoff function. In its simplest form, the replicator equation writes as
\begin{align}
    \label{intro:replicator}
    \dot{\lambda} = b(\lambda) \,,\qquad b_{i} (\lambda) = ( ( A\lambda)_{i} - A\lambda \cdot \lambda) \lambda_{i} \quad \text{for $i=1, \ldots, M$.}
\end{align}
The ODE system~\eqref{intro:replicator} describes the evolution of the frequency $\lambda = (\lambda_{1}, \ldots, \lambda_{M})$ of pure strategies $U= \{ u_{1}, \ldots, u_{M}\}$ in a population of $N$ individuals, namely $\lambda_i=\frac1N \#\{\text{agents with strategy $u_i$}\}$. In classical applications to evolutionary biology, the selection of pure (or mixed) strategies is mediated by how their payoff relates to the average within the population. The entry $a_{ij}$ of the matrix $A \in \R^{M \times M}$ denotes the payoff of the strategy $u_{i}$ against $u_{j}$, and each component $b_{i}(\lambda)$ compares the payoff of a single strategy with the average payoff, represented by $A\lambda \cdot \lambda$. If the payoff of the individual strategy $u_i$ is larger than the average one, its frequency in the population rises (hence the replication); if the payoff of the individual strategy $u_i$ is smaller than the average one, the strategy is progressively suppressed. It is important to notice that the replicator dynamics described by~\eqref{intro:replicator} is deterministic. 
Diffusion effects are thus neglected, as they are typically slower than the transport term in~\eqref{intro:replicator} (see, e.g.,~\cite{AFMS, ChSo, HoSi, MoOr, Wei}).

In~\cite{ChMoRi, ChSo, MoOr} the replicator equation~\eqref{intro:replicator} has been linked to the Moran process~\cite{Moran} for a finite number of pure strategies. The Moran process is a time-discrete stochastic process which models the evolution of strategies within a population of $N$ individuals. 
For time steps equally spaced by $\tau>0$, at a given time step $t_{h} = h\tau$ an agent is selected to replicate their strategy, at the expenses of a randomly selected agent that must abandon their own. The probability of being chosen for reproduction is expressed in terms of {\em fitness} of the strategy, which in turn is a function of the payoff. To better understand the involved parameters and scalings, let us present the simplest scenario  $M=2$ (in the original, biological interpretation, two alleles $u_{1}$ and $u_{2}$ competing for dominance). The probability of reproduction is of the form
\begin{equation}\label{eq:prob}
\begin{split}
  \mathbb{P}\Big(\text{``a $u_i$-individual is chosen for reproduction''}\Big) & = \frac{\lambda_{i}f_{u_i}}{\lambda_{1} f_{u_1} + \lambda_{2} f_{u_2}} \, ,
\end{split}
\end{equation}
$f_{u_{i}}$ denoting the fitness of strategy~$u_{i}$. A typical choice~\cite{Nowak} is to write $f_{u_{i}}$ as a convex combination of the constant~$1$ and the payoff~$\pi_{u_{i}}$ of the strategy $u_{i}$, i.e.,
\begin{align}
\label{intro:fitness}
    f_{u_{i}} = (1 - w) + w \pi_{u_{i}}\,, \qquad w \in [0, 1]\,,
\end{align}
where, recalling the payoff matrix~$A$ from~\eqref{intro:replicator}, $\pi_{u_{i}}$ may be expressed as 
\begin{align}
\label{intro:payoff}
    \pi_{u_{i}} =  a_{i1} \lambda_{1} + a_{i2}\lambda_{2}\,, \qquad i=1, 2.
\end{align}
The extreme case $w=0$ is called \emph{neutral selection} and corresponds to uniform probability in~\eqref{eq:prob}; on the contrary, if $w=1$, then the fitness is solely determined by the payoff. 
Small values of~$w$ model a \emph{weak selection} regime, whereas large ones indicate \emph{strong selection}.

The discrete process is fully described by the stochastic evolution of the  vector $\lambda^{N,\tau}(t_h)=(\lambda^{N,\tau}_1(t_h),\lambda^{N,\tau}_2(t_h))$ collecting the proportions of strategies at discrete times $t_h=h\tau$. Since its components add up to one, $\lambda^{N,\tau}(t_h)$ is a random point in the simplex $\Delta^{1} \subset \R^2$. It has been shown in~\cite{MoOr} that the law $\Lambda^{N,\tau} (t) \in \Pp(\Delta^{1})$ of the piecewise affine interpolant $\lambda^{N,\tau} (t)$ converges to the (unique) solution to an Eulerian version of the replicator equation~\eqref{intro:replicator} in the form
\begin{align}
\label{intro:eulerian-replicator}
\partial_{t} \Lambda + {\rm div} (b\Lambda) = 0\,, \qquad \Lambda \in C([0, T]; \Pp(\Delta^{1}))\,,
\end{align}
with the vector field $b$ as in~\eqref{intro:replicator}. The result has been obtained under suitable scaling assumptions of the three involved parameters: the vanishing weak selection parameter $w$, the vanishing time step $\tau$, and the diverging number of agents $N$. The analysis hinges on two main simplifications. On the one hand, agents only rely on a finite number of pure strategies. On the other hand, the chosen scaling allows one to neglect higher order terms (specifically, diffusion effects), leading to a pure replicator dynamics. In order to quantify such scalings, let us mention (see \cite[Proposition~4.4]{MoOr} and Corollary~\ref{c:almost-equation} below) that the curve $t \mapsto \Lambda^{N,\tau} (t)$ solves in the sense of distribution the (approximate) equation
\begin{align}
    \label{intro:approx-replicator}
    \partial_t \Lambda^{N,\tau} + \frac{w}{\tau N} {\rm div} \big(  b  \Lambda^{N,\tau} \big) \approx 0 \,,
\end{align}
with an error of order $O(1/\tau N^{2})$, involving second-order derivatives of the test function. Thus, assuming $1/\tau N^{2} \to 0$ eliminates the diffusion terms in the limit.

\medskip

\paragraph{\bf Our goals: including diffusion and extending to infinitely many strategies.}
We develop an analytical framework to achieve a twofold goal. On the one hand, we include diffusion effects in the model. Indeed, the natural guess is that in the weak selection regime, under the scaling suggested by the previous discussion (which includes a vanishing selection parameter $w$), competition between the driving principles of natural selection and genetic drift diffusion may arise, generating a richer dynamics than a pure replicator equation. 

We also show, instead, that if $w$ remains positive (as in the strong selection regime) and the payoff function is non-negative (see Remark \ref{strong-sel} and Theorem \ref{t:compactness-3-strong}), selection prevails and the limit dynamics results in a pure replicator-type evolution with a  modified field $b$.

On the other hand, it is highly desirable to encompass an infinite number of strategies varying in a compact metric space $(\Vv, \di)$ in our analytical framework. This corresponds to a well-established point of view in many relevant applications. For instance, Nash's approach~\cite{Nash} to equilibria in non-cooperative games builds upon von Neumann’s notion of mixed strategy. In this setting, for the set of pure strategies $U = \{ u_{1}, \ldots, u_{M}\}$ as above, mixed strategies are represented by probability distributions in $\Vv= \Pp(U)$. Elements of $\Vv$ represent, for instance, the probability that players have of choosing a pure strategy or, in a financial context, a portfolio of investment options~\cite{Friedman}, or also the degree of influence exerted on other players~\cite{AAMS}. In evolutionary biology, a model of particular interest takes $(\Vv, \di)$ to be the continuous spectrum of acoustic frequencies available to birds. Here, fitness measures the reproductive returns of each frequency (or song), capturing its behavioral effectiveness~\cite{Podos and Warren}. We further refer to \cite[Introduction]{BNP} and the references therein for a discussion on dynamic models for continuous strategy spaces.

Summarizing with the present work we aim at developing the analytical framework to encompass infinitely many strategies in the Moran process and to treat different mutual scalings of $w$, $\tau$, and $N$ that allow for second-order effects. To the best of our knowledge, the analysis in the proposed setting is uncovered. Dealing with an infinite number of strategies will result in working with infinite dimensional spaces lacking linearity, with the typical structure of {\em Wasserstein over Wasserstein}~\cite{AAMS, ADS, ADS2, AFMS, MS2020}, that we are going to describe below.

\medskip

\paragraph{\bf Novel contribution.}
In order to be more precise on this aspect, let us briefly introduce the Moran process with an infinite pool of strategies, denoted by a compact metric space $(\Vv, \di)$. At time $t_{h} = h\tau$ the $N$ agents have strategies $\sigma_{n} \in \Vv$ for $n = 1, \ldots, N$. The distribution of strategies is represented by the empirical measure $\lambda^{N,\tau} (t_{h}) = \frac{1}{N} \sum_{n=1}^{N} \delta_{\sigma_{n}} \in \Pp(\Vv)$. Following~\eqref{intro:fitness}--\eqref{intro:payoff}, the fitness of the $n$-th agent to replicate its strategy~$\sigma_{n}$ is given by the convex combination
\begin{align}
\label{intro:fitness-infinite}
    F_{\lambda^{N,\tau} (t_{h}) } (\sigma_{n}) = (1 - w) + w \, \pi_{\lambda^{N,\tau} (t_{h})} (\sigma_{n})\,,
\end{align}
where the payoff function $\pi \colon \Vv \times \Pp(\Vv) \to \R$ is only assumed to be continuous with respect to $(\sigma, \lambda) \in \Vv \times \Pp(\Vv)$. Hence, the strategy $\sigma_{n}$ will replicate with probability
\begin{align*}
    \mathbb{P} \Big(\text{``strategy $\sigma_n$ is chosen for reproduction''}\Big) & =\frac{F_{\lambda^{N,\tau} (t_{h})} (\sigma_{n})}{\sum_{n'=1}^{N} F_{\lambda^{N,\tau} (t_{h})} (\sigma_{n'})}\,.
\end{align*}
Passing to the law $\Lambda^{N,\tau} (t)$ of the piecewise affine interpolant $\lambda^{N,\tau} \colon [0, T] \to \Pp (\Vv)$ of $\lambda^{N,\tau} (t_{h})$, we are left to work with continuous curves of probability measures $\Lambda^{N,\tau} \colon [0, T] \to \Pp(\Pp(\Vv))$, which reveal the {\em Wasserstein over Wassertein structure}. We refer to Section~\ref{S:pre} for further discussions.
The natural domain of our model is the mixed-strategy set $\Vv = \Pp(U)$. Crucially, the underlying analysis imposes no further restrictions beyond compactness and metrizability. We put this latitude to use in Section~\ref{s:spin}, where the strategy space is the continuous circle $\Vv = \mathbb{S}^{1}$ for an alignment problem in spin systems. This example, coupled with the ones preceding it, underscores that our framework cohesively unifies fitness-based interactions, regardless of whether strategies are discrete, continuous, or randomized.

To reveal the significant scalings of $w$, $\tau$, and $N$, we prove in Theorem~\ref{t:almost-equation} and Corollary~\ref{c:almost-equation} that, in the case of vanishing selection parameter, $\Lambda^{N,\tau}$ solves in distributional sense the approximate equation
\begin{align}
\label{intro:approx-eq-nostra}   
 \partial_{t} \Lambda^{N,\tau} + \frac{w}{\tau N}  {\rm div} \big(b(\lambda) \Lambda^{N,\tau} \big) - \frac{1}{\tau N^{2}} {\rm div}_{\lambda} \big( {\rm div}_{\lambda}\big( \mathcal{B}(\lambda) \Lambda^{N,\tau}  \big) \big) \approx 0\,,
\end{align}
 where
    \begin{align*}
        b(\lambda) = (\pi_\lambda - \overline{\pi}_\lambda) \lambda, \qquad \overline{\pi}_{\lambda} = \int_{\Vv} \pi_{\lambda} (\sigma) \, \di \lambda(\sigma) \quad \text{for $\lambda \in \Pp(\Vv)$}
    \end{align*}
    and, at least formally,
    \begin{align*}
      \big\langle {\rm div}_{\lambda} \big( {\rm div}_{\lambda}\big( \mathcal{B}(\lambda) \Lambda  \big) \big), \psi (t, \lambda) \big\rangle & = \frac{1}{2}\langle \mathcal{B}(\lambda) \Lambda , D^{2}_{\lambda} \psi (t, \lambda) \rangle  
    \\
    &
    = \frac{1}{2}\int_{\Pp(\Vv)} \int_{\Vv} \int_{\Vv} D^{2}_{\lambda} \psi (t, \lambda) [\delta_{\sigma} - \delta_{\overline{\sigma}} \ , \ \delta_{\sigma} - \delta_{\overline{\sigma}} ]  \, \di \lambda(\sigma) \, \di \lambda(\overline{\sigma}) \, \di \Lambda (\lambda) .
    \end{align*}
    Once again, equation~\eqref{intro:approx-eq-nostra} is solved up to an error of order $O(1/N^{2})$. We remark that one needs to be very careful in giving a meaning to equation~\eqref{intro:approx-eq-nostra}, as it is formulated on the state space~$\Pp(\Vv)$. Here we make use of the infinite dimensional calculus rules on convex sets developed in~\cite{AFMS}, which can be iteratively employed to define $C^{k}$-functions over $\Pp(\Vv)$ endowed with the Bounded-Lipschitz norm. We refer to Section~\ref{S:pre} for more details. 
    
    The approximate equation~\eqref{intro:approx-eq-nostra} reveals the relevant scalings of the involved parameters. Under the assumption
    \begin{align}
    \label{intro:scales}
        \frac{w}{\tau N} \to \gamma \geq0 \qquad \frac{1}{\tau N^{2}} \to \nu \geq 0
    \end{align}
    with $\gamma, \nu$ non-simultaneously $0$, we will show in Theorems~\ref{t:compactness-1}, ~\ref{t:compactness-3}, and \ref{t:compactness-3-strong} that $\Lambda^{N,\tau}$ converges to a continuous curve $\Lambda \colon [0, T] \to \Pp(\Pp(\Vv))$ which solves the replicator equation with genetic drift diffusion
    \begin{align}
    \label{intro:final-eq}
        \partial_{t} \Lambda + \gamma \, {\rm div}_{\lambda} ( b(\lambda) \Lambda) - \nu \, {\rm div}_{\lambda} ( {\rm div}_{\lambda} ( \mathcal{B} (\lambda) \Lambda)) = 0\,.
    \end{align}
    On the technical side, we remark that the crucial step in proving Theorems~\ref{t:compactness-1}, ~\ref{t:compactness-3}, and \ref{t:compactness-3-strong} is the compactness of the curve~$\Lambda^{N,\tau}$. As pointed out in~\cite[Remark~5.2]{MoOr} for the pure replicator dynamics, we cannot expect $\Lambda^{N,\tau}$ to be equi-continuous, calling for a more refined analysis. Moreover, differently from~\cite{MoOr}, the presence of the second order term in~\eqref{intro:approx-eq-nostra} and~\eqref{intro:final-eq} forces us to endow $\Pp (\Pp(\Vv))$ with a distance $\di_{\mathcal{Z}}$ constructed, in the spirit of the $1$-Wasserstein distance~$W_{1}$, by duality with test functions in
    \[
    \mathcal{Z}\coloneqq \bigl\{\psi \in C^{2}(\Pp(\Vv)) : {\rm Lip}(\psi) + {\rm Lip} (D_{\lambda} \psi) \leq 1 \bigr\}.
    \]
    We refer to~\eqref{e:distanceZ} and Lemma~\ref{l:dZ} below for the definition of~$\di_{\mathcal{Z}}$ and for some basic compactness properties of the metric space $(\Pp(\Pp(\Vv)), \di_{\mathcal{Z}})$.

    We conclude this introduction with a brief comparison of our result with the existing literature concerning the Moran process and replicator dynamics. First, we notice that the analysis of~\cite{MoOr} is obtained as a byproduct of Theorem~\ref{t:compactness-3} when $\Vv = U$ is finite and
    \begin{align*}
         \frac{w}{\tau N} \to \gamma >0 \qquad\text{and} \qquad \frac{1}{\tau N^{2}} \to  0\,.
    \end{align*}
    Moreover, the scalings detected in~\cite{ChSo} are in agreement with~\eqref{intro:scales} and the convergence results in Theorems ~\ref{t:compactness-1}, ~\ref{t:compactness-3}, and \ref{t:compactness-3-strong}. Indeed, setting
    \begin{align}
    \label{e:introscales}
        N \sim \tau^{-\alpha} \qquad \text{and}\qquad w\sim \tau^{\beta}\qquad \text{for $\alpha>0$ and $\beta \geq 0$},
    \end{align}
     we deduce that~$\Lambda^{N,\tau}$ converges to solutions to a pure replicator equation if $\alpha + \beta = 1$ and $\alpha >\frac{1}{2}$, while we obtain a replicator equation with genetic drift for $\alpha = \beta = \frac{1}{2}$. In this setting, our result may be seen as the rigorous counterpart of~\cite{ChSo} for a finite number of pure strategies. Such a scenario is discussed in detail in Section~\ref{s:pure}, where we show how our abstract setting reduces to the classical Fleming--Viot operator with $M$ competing pure strategies~\cite{DaMa, DoKu, Fle-Vio} and to the Kimura equation if $M=2$~\cite{Carrilloetal-2022, Casteras-Monsaingeon, Kimura1, Kimura2}. We summarize our findings on the reciprocal scales on~$w$, $\tau$, and~$N$ in~\eqref{e:introscales} in the following table.

\vspace{4mm}
\begin{center}
     \begin{tabular}{|c|c|p{6cm}|}
            \hline
            \textbf{Scaling Parameters} & \textbf{Asymptotic Rates} & \textbf{Macroscopic Model} \\
            \hline
            \hline
            $\alpha + \beta = 1$ & $\gamma > 0$ \footnotesize{(Selection)} & \textbf{Pure replicator dynamics} \\ 
            $\alpha > 1/2$       & $\nu = 0$ \footnotesize{(No Drift)}     & \footnotesize{Deterministic selection dominates. Diffusion vanishes.} \\ 
            \hline
            $\alpha = 1/2$       & $\gamma > 0$ \footnotesize{(Selection)} & \textbf{Replicator with drift diffusion} \\ 
            $\beta = 1/2$        & $\nu > 0$ \footnotesize{(Drift)}        & \footnotesize{Darwinian selection and Kimura's drift are perfectly balanced.} \\ 
            \hline
            $\alpha = 1/2$       & $\gamma = 0$ \footnotesize{(No Selection)} & \textbf{Pure genetic drift diffusion} \\ 
            $\beta > 1/2$        & $\nu > 0$ \footnotesize{(Drift)}        & \footnotesize{Selection is too weak; pure random fluctuations drive the population.} \\ 
            \hline
        \end{tabular}
        \end{center}
\vspace{3mm}

\medskip

\paragraph{\bf Outline of the paper.}
The paper is organized as follows: in Section~\ref{S:pre}, we collect the mathematical preliminaries concerning function spaces, differential calculus in $\Pp(\Vv)$, and introduce the distance $\di_{\mathcal{Z}}$ between elements of $\Pp(\Pp(\Vv))$.
In Section~\ref{s:discrete-process}, we describe the time-discrete Moran process and state the main convergence Theorems~\ref{t:compactness-1}, ~\ref{t:compactness-3}, and \ref{t:compactness-3-strong}. We devote Section~\ref{S:examples} to explicit examples that are of relevance in some applied contexts. In particular, we show that our framework fits with the theory of Fleming--Viot operators~\cite{Fle-Vio} for a finite or a continuum set of pure strategies~$\Vv$. In the case of 2 sole competing strategies, we recover the well-known Kimura equation~\cite{Kimura1, Kimura2}. Furthermore, we present in Section~\ref{s:spin} an application to spin systems. The proofs of Theorems~\ref{t:compactness-1},~\ref{t:compactness-3}, and \ref{t:compactness-3-strong} are contained in Sections~\ref{s:approximate equation} and~\ref{s:compactness}. Precisely, in Theorem~\ref{t:almost-equation} and Corollary~\ref{c:almost-equation} we rigorously prove~\eqref{intro:approx-eq-nostra}. In Section~\ref{s:compactness} we show compactness of the time-discrete curves and recover the replicator equation with genetic drift diffusion~\eqref{intro:final-eq}.

\section{Preliminaries}
\label{S:pre}

We fix $(\Vv, \di)$ a compact metric space. We denote by $\Pp(\Vv)$ the space of probability measures over $\Vv$. We will always endow $\Pp(\Vv)$ with the metric induced by the Bounded Lipschitz norm
\begin{align*}
\| \mu \|_{\mathrm{BL}} : = \sup\, \bigg\{\langle \mu, \varphi \rangle \,  : \, \varphi \in {\rm Lip} (\Vv), \, \| \varphi\|_{\rm Lip} \leq 1\bigg\} \qquad \text{for every $\mu \in ({\rm Lip} (\Vv) )^{*}$,}
\end{align*}
where $\| \varphi\|_{\rm Lip} : = \| \varphi\|_{\infty} + {\rm Lip} (\varphi)$ and ${\rm Lip} (\varphi)$ denotes the Lipschitz constant of~$\varphi$, namely, 
\begin{align*}
{\rm Lip} (\varphi) : = \sup_{x,y\in \Vv \atop x\neq y}\frac{| \varphi ( x ) - \varphi ( y )|}{\di (x,y)}\,.
\end{align*}
We further denote by $\F(\Vv)$ the so-called Arens--Eells space 
\begin{align*}
\F( \Vv) : =  \overline{ \spann ( \Pp (\Vv) ) }^{\|\cdot\|_{\mathrm{BL}}} \subseteq (\mathrm{Lip}(\Vv))^{*}.
\end{align*}
In view of~\cite{AFMS, Ambrosio-Puglisi}, $( \F(\Vv), \| \cdot\|_{\mathrm{BL}})$ is a separable Banach space. Since $\Vv$ is compact, $(\Pp(\Vv), \| \cdot\|_{\rm BL})$ is a convex compact subset of~$\F(\Vv)$. Moreover, denoting with $\mathrm{d}_{\mathrm{BL}}$ the distance on $\mathcal{P}(\mathcal{V})$ induced by the Bounded Lipschitz norm, thanks to the compactness of $\mathcal{V}$, we have that $\mathrm{d}_{\mathrm{BL}}$ and the 1-Wasserstein distance are equivalent. Thus, $(\mathcal{P}(\mathcal{V}), \| \cdot\|_{\rm BL})$ and $(\mathcal{P}(\mathcal{V}), W_1)$ are the same topological space. We further denote by $E_{\Pp(\Vv)}$ the closed vector space $E_{\Pp(\Vv)} := \overline{\R (\Pp(\Vv) - \Pp(\Vv))}^{\| \cdot\|_{\rm BL}}$, i.e., the closure with respect to the ${\rm BL}$-norm of the set $\{\alpha ( \lambda_{1} - \lambda_{2}) : \, \alpha \in \R, \, \lambda_{1},\lambda_{2} \in \Pp(\Vv)\}$, which is a vector subspace thanks the convexity of $\mathcal{P}(\Vv)$. Notice that $(E_{\Pp(\Vv)}, \| \cdot\|_{\rm BL})$ is a closed subspace of $(\mathcal{F} (\Vv), \| \cdot\|_{\rm BL})$, and therefore is Banach and separable. We denote by $(E_{\Pp(\Vv)})^{*}$ its dual space.

Following~\cite{AFMS}, given $(Y, \| \cdot\|_{Y})$ a normed space, we say that a function $\psi \colon \Pp(\Vv) \to Y$ is $C$-differentiable if for every $\lambda \in \Pp(\Vv)$ there exists $D \psi (\lambda) \in \mathcal{L} ( E_{\Pp(\Vv)} ; Y)$ such that
\begin{align*}
    \lim_{\lambda' \to \lambda } \, \frac{\psi(\lambda') - \psi(\lambda) - D\psi(\lambda) [\lambda' - \lambda]}{ \| \lambda' - \lambda\|_{\rm BL}}  = 0\,.
\end{align*}
We say that $\psi \in C^{1} (\Pp(\Vv); Y)$ if $\psi$ is $C$-differentiable and the map $D\psi \colon \Pp(\Vv) \to  \mathcal{L} ( E_{\Pp(\Vv)} ; Y)$ is continuous, where $\mathcal{L} ( E_{\Pp(\Vv)} ; Y)$ is endowed with the standard norm for linear operators between normed spaces. Arguing recursively, we are able to define $C^{k}$-functions from $\Pp(\Vv)$ to $Y$ for every $k \in \mathbb{N}$. In the sequel, we will often deal with $\psi \in C^{2} (\Pp(\Vv))$ (real-valued functions). In this case, $D\psi \colon \Pp(\Vv) \to (E_{\Pp(\Vv)})^{*}$ and $D^{2} \psi \colon  \Pp(\Vv) \to \mathcal{L} (E_{\Pp(\Vv)}; (E_{\Pp(\Vv)})^{*})$ are well-defined and continuous. As shown in~\cite[Lemma~4.4]{ADS} it holds that
\begin{align*}
    \| D\psi \|_{L^{\infty} (\Pp(\Vv))} \leq {\rm Lip} (\psi) \qquad \text{and} \qquad  \| D^{2}\psi \|_{L^{\infty} (\Pp(\Vv))}  \leq {\rm Lip} (D\psi)
\end{align*}
for every $\psi \in C^{2} (\Pp(\Vv))$. As we will often consider test function $\psi \in C^{2}([0, T] \times \Pp(\Vv))$ also depending on time, we will denote by $D_{\lambda} \psi (t, \lambda)$ and $D^{2}_{\lambda} \psi (t, \lambda)$ the $C$-differentials of the map $\psi (t, \cdot)$ for fixed $t \in [0, T]$.

Later in this work we will further consider the set $\Pp(\Pp(\Vv))$. Notice that $\Pp(\Pp(\Vv))$ can be endowed with the $1$-Wasserstein distance
\begin{align*}
    W_{1} (\Lambda_{1}, \Lambda_{2}) := \sup\, \bigg\{ \int_{\Pp(\Vv)} \varphi(\lambda) \, \di (\Lambda_{1} - \Lambda_{2})(\lambda) : \, \varphi \in {\rm Lip}(\Pp(\Vv)) \text{ with ${\rm Lip} (\varphi) \leq 1$} \bigg\}\,
\end{align*}
where the Lipschitz constant is computed with respect to the ${\rm BL}$-norm. Notice that $(\Pp(\Pp(\Vv)), W_{1})$ is a compact metric space.

We denote by $\mathcal{Z}$ the set
\begin{align*}
        \mathcal{Z} := \big\{ \psi \in C^{2}(\Pp(\Vv)): \, {\rm Lip} (\psi) + {\rm Lip} (D\psi) \leq 1\big\}\,.
\end{align*}
We introduce in $\Pp(\Pp(\Vv))$ the following function $\di _{\mathcal{Z}} \colon \Pp(\Pp(\Vv)) \times \Pp(\Pp(\Vv)) \to [0, +\infty)$:
\begin{align}
\label{e:distanceZ}
    \di_{\mathcal{Z}} ( \Lambda_{1}, \Lambda_{2}) := \sup\left\{\int_{\Pp(\Vv)} \psi(\lambda) \, \di (\Lambda_{1} - \Lambda_{2}) (\lambda) : \, \psi \in \mathcal{Z}\right\}\qquad \text{for $\Lambda_{1}, \Lambda_{2} \in \Pp(\Pp(\Vv))$}.
\end{align}
In the next two lemmas we prove that $\di_{\mathcal{Z}}$ is indeed a distance on~$\Pp(\Pp(\Vv))$.

\begin{lemma}
    \label{l:density}
    The space $C^{\infty} (\Pp(\Vv))$ is dense in ${\rm Lip} (\Pp(\Vv))$ with respect to the $L^{\infty}(\Pp(\Vv))$-norm.
\end{lemma}

\begin{proof}
    We may consider the algebra $\mathcal{A}$ of functions of the form
    \begin{align*}
        \psi ( \varphi_{1} (\lambda), \ldots, \varphi_{n} (\lambda)) \qquad \lambda \in \Pp(\Vv)\,,
    \end{align*}
    where $\varphi_{1}, \ldots, \varphi_{n} \in (\mathcal{F}(\Vv))^{*}$ and $\psi$ is a polynomial of $n$ variables. Notice that such functions belong to $C^{\infty} (\Pp(\Vv))$ (even to $C^{\infty} (\mathcal{F} (\Vv))$). Clearly, $\mathcal{A}$ contains constant functions. Moreover, $\mathcal{A}$ separates the points of~$\Pp(\Vv)$. Hence, a direct application of Stone-Weierstrass implies that $\mathcal{A}$ is dense in the set of continuous functions $C(\Pp(\Vv))$ in the $L^{\infty} (\Pp(\Vv))$-norm. This concludes the proof of the lemma.
\end{proof}

\begin{lemma}
\label{l:dZ}
    The space $(\Pp(\Pp(\Vv)), \di_{\mathcal{Z}})$ is a compact metric space. Moreover, 
    \begin{equation}
    \label{e:111}
        \di_{\mathcal{Z}} (\Lambda_1, \Lambda_2) \leq W_{1}( \Lambda_1, \Lambda_2) \qquad \text{for every $\Lambda_1, \Lambda_2 \in \Pp(\Pp(\Vv))$}\,.
    \end{equation}
\end{lemma}

\begin{proof}
    Inequality \eqref{e:111} is a direct consequence of the definition of~$\di_{\mathcal{Z}}$ and of $W_{1}$. The map $\di_{\mathcal{Z}}$ is clearly symmetric and satisfies the triangle inequality. If $\Lambda_{1}, \Lambda_{2} \in \Pp(\Pp(\Vv))$ are such that $\di_{\mathcal{Z}} (\Lambda_{1}, \Lambda_{2}) = 0$, then
    \begin{align}
    \label{e:zero-distance}
        \int_{\Pp(\Vv)} \psi(\lambda) \, \di (\Lambda_{1} - \Lambda_{2}) (\lambda) = 0 \qquad \text{for every $\psi \in \mathcal{Z}$.}
    \end{align}
    Hence, \eqref{e:zero-distance} is satisfied for every $\psi \in {\rm Span}(\mathcal{Z})$. Notice that, since $\Pp(\Vv)$ is compact, ${\rm Span}(\mathcal{Z}) = C^2(\Pp(\Vv))$. Therefore, \eqref{e:zero-distance} holds for every $\psi \in C^2(\Pp(\Vv))$.
    Thanks to Lemma~\ref{l:density}, we may extend~\eqref{e:zero-distance} to test functions $\psi \in {\rm Lip} (\Pp(\Vv))$. This in turn implies that also $W_{1} (\Lambda_{1} , \Lambda_{2}) = 0$ and $\Lambda_{1} = \Lambda_{2}$. Hence, $\di_{\mathcal{Z}}$ is a metric on~$\Pp(\Pp(\Vv))$. The compactness of~$(\Pp(\Pp(\Vv)) , W_{1})$ yields the compactness of $(\Pp(\Pp(\Vv)), \di_{\mathcal{Z}})$.
\end{proof}

\section{Description of the time-discrete process and main results}
\label{s:discrete-process}
This section describes the abstract Moran process for a finite number of agents $N$ with a possibly infinite pool of strategies~$\Vv$, as well as the main results of the paper concerning its asymptotic behavior in terms of the reciprocal scalings among the diverging number of agents and the vanishing time-step~$\tau$ and selection parameter~$w$. 

\subsection{Time-discrete process}
In this first subsection we describe the time-discrete Moran process. We start by introducing the initial conditions, as well as the fitness and payoff functions. In the last paragraph we describe the death-birth process.

\paragraph{\bf Initial condition.} 
Let us fix $\Lambda_{0} \in \Pp(\Pp(\Vv))$ as initial datum. By Skorokhod representation theorem (see, e.g., \cite[Section~6]{Billingsley}), there exists a probability space $(\Omega, \mathbb{P}, \mathcal{F})$ and a measurable map $\lambda_0 \colon \Omega \to \Pp(\Vv)$ such that $\Lambda_{0} = (\lambda_{0})_{\#} \mathbb{P}$. We consider a finite time horizon $T>0$. For every $k \in \mathbb{N} \setminus \{0\}$ we define $\tau_{k} = \frac{T}{k}$ and let $t^{k}_{m} = m \tau_{k}$ for every $m \in\{ 1, \ldots, k\}$. We approximate the initial condition thanks to the following lemma.

\begin{lemma}
    \label{l:selection}
    Let $(\Omega, \mathbb{P}, \mathcal{F})$, $\lambda_{0} \colon \Omega \to \Pp(\Vv)$ measurable, and $\Lambda_{0} = (\lambda_{0})_{\#} \mathbb{P}$. Then, there exists a sequence of measurable maps $\sigma_{0, n}^{N} \colon \Omega \to \Vv$, $n\in \{1, \ldots, N\}$ and $N \in \mathbb{N}$, such that, defining $\lambda_{0}^{N}:= \frac{1}{N} \sum_{n=1}^{N}\delta_{\sigma^{N}_{0, n}} \in \Pp(\Vv)$ and $\Lambda_{0}^{N} := (\lambda_{0}^{N})_{\#} \mathbb{P} \in \Pp(\Pp(\Vv))$ it holds $W_{1} (\Lambda_{0}^{N}, \Lambda_{0}) \to 0$ as $N \to \infty$. 
\end{lemma}

\begin{proof}
    Let us fix $D$ a dense and at most countable subset of $\Vv$. For $N \in \mathbb{N}$ we assume that $\Vv^{N}$ is endowed with the distance function $\di_{N} \colon \Vv^{N} \times \Vv^{N} \to [0, +\infty)$
    \begin{align*}
        \di_{N} (\theta, \sigma) := \left(\sum_{n=1}^{N} (\di (\theta_{n}, \sigma_{n}))^{2} \right)^{1/2} \qquad \text{for every $\theta, \sigma \in \Vv^{N}$.}
    \end{align*}
    Thus, the set $D^{N}$ is dense in $\Vv^{N}$. We consider the map $G_{N}\colon \Omega \to [0, +\infty)$ defined by
    \begin{align*}
        G_{N}(\omega) :=  \inf_{\theta \in D^{N}} \, W_{1} \left( \lambda_{0} (\omega) , \frac{1}{N} \sum_{n=1}^{N} \delta_{\theta_{n}} \right) = \min_{\theta \in \Vv^{N}}  W_{1} \left( \lambda_{0} (\omega) , \frac{1}{N} \sum_{n=1}^{N} \delta_{\theta_{n}} \right)\,.
    \end{align*}
    Notice that the second inequality follows by continuity of $W_{1}$ and density of~$D$. The function~$G_{N}$ is measurable over $\Omega$, since it is the infimum of a sequence of measurable functions. 

    We define the multi-function $T_{N} \colon \Omega \to \Vv^{N}$ as
    \begin{align*}
        T_{N} (\omega) := \left\{ \theta \in \Vv^{N}: W_{1} \left(\lambda_{0} (\omega) , \frac{1}{N} \sum_{n=1}^{N} \delta_{\theta_{n}} \right) = G_{N}(\omega) \right\} \,. 
    \end{align*}
    For every open set $A \subseteq \Vv^{N}$ we have that
    \begin{align}
        \label{e:a condition}
        \left\{ \omega \in \Omega: \, T_{N} (\omega) \cap A \neq \emptyset \right\} = \left\{ \omega \in \Omega: \, \inf_{\theta \in D^{N} \cap A} \, W_{1} \left(\lambda_{0}(\omega) , \frac{1}{N} \sum_{n=1}^{N} \delta_{\theta_{n}} \right) = G_{N} (\omega)\right\}.
    \end{align}
    Hence, the set $\{\omega \in \Omega : \, T_{N} (\omega) \cap A \neq \emptyset\}$ is a measurable subset of~$\Omega$. By~\cite[Theorem III.6]{Castaign-Valadier}, there exists a measurable selection $\sigma^{N}_{0} (\omega)  \in T_{N} (\omega)$.

    Let us set $\lambda_{0}^{N} (\omega) := \frac{1}{N} \sum_{n=1}^{N} \delta_{\sigma^{N}_{0, n} (\omega)}$ and $\Lambda^{N}_{0} := (\lambda^{N}_{0})_{\#} \mathbb{P}$. To conclude, it is enough to notice that (see, e.g.,~\cite[Step~2, Proof of Theorem~3.5]{MS2020})
    \begin{align*}
        \lim_{N \to \infty} \inf_{\theta \in \Vv^{N}} \, W_{1} \left(\lambda_{0} (\omega) , \frac{1}{N} \sum_{n=1}^{N} \delta_{\theta_{n}} \right) = 0 \qquad \text{for $\omega \in \Omega$}
    \end{align*}
    and that, by compactness of~$\Vv$, $W_{1} (\lambda_{0} (\omega) , \lambda^{N}_{0} (\omega) ) \leq 2 \, {\rm diam} (\Vv)$. Hence,
    \begin{align*}
        W_{1} (\Lambda^{N}_{0} , \Lambda_{0}) & = \sup \bigg\{ \int_{\Pp(\Vv)} \varphi (\lambda) \, \di (\Lambda^{N}_{0} - \Lambda_{0}) (\lambda) : \, \varphi \in {\rm Lip} (\Pp(\Vv)) , \, {\rm Lip} (\varphi) \leq 1\bigg\} 
        \\
        &
        \vphantom{\int_{\Pp(\Vv)}} = \sup \Big\{ \mathbb{E}  \big( \varphi (\lambda_{0}) - \varphi (\lambda^{N}_{0})  \big)  : \varphi \in {\rm Lip} (\Pp(\Vv)) , \, {\rm Lip} (\varphi) \leq 1 \Big\}
        \\
        &
        \leq \vphantom{\int_{\Pp(\Vv)}}  \mathbb{E}  \big( \|  \lambda_{0} - \lambda^{N}_{0}\|_{\rm BL}  \big) \leq  \mathbb{E}  \big( W_{1} ( \lambda_{0} , \lambda^{N}_{0} )  \big)\,.
    \end{align*}
    By the Dominated Convergence Theorem, we conclude that $W_{1} (\Lambda^{N}_{0} , \Lambda_{0}) \to 0$ as $N \to \infty$.
\end{proof}

We fix from now on a sequence of initial data $\Lambda^{N}_{0} \in \Pp(\Pp(\Vv))$ as in the statement of Lemma~\ref{l:selection}.

\medskip

\paragraph{\bf Fitness and payoff functions.}
In order to describe the Moran time-discrete process, we introduce the following {\em fitness} function. For $\lambda \in \Pp(\Vv)$, we define $F_{\lambda} \colon \Vv \to \R$ as
\begin{align}
\label{e:Flambda}
F_{\lambda} (\sigma) & = 1 + w ( \pi_{\lambda} (\sigma) -1) \qquad \text{for $\sigma \in \Vv$ and $\lambda \in \Pp(\Vv)$,}
\end{align}
where $w\in [0,1]$ is the selection parameter and $\pi_{\lambda} \colon \Vv \to \R$ represents the {\em payoff} function. We will consider both the fitness and the payoff function as functions of the pair $(\sigma, \lambda) \in \Vv \times \Pp(\Vv)$ and denote $F(\sigma, \lambda) = F_{\lambda} (\sigma)$ and $\pi(\sigma, \lambda) = \pi_{\lambda} (\sigma)$. For $\lambda \in \Pp(\Vv)$ we further set
\begin{align}
    \label{e:averagepi}
    \overline{\pi}_{\lambda } := \int_{\Vv} \pi_{\lambda} (\sigma) \, \di \lambda(\sigma) \qquad \text{and} \qquad \overline{F}_{\lambda} := \int_{\Vv} F_{\lambda} (\sigma) \, \di \lambda (\sigma) \,.
\end{align}

The only assumption we need in view of our scaling limits is the following:
\begin{itemize}
    \item [$(\pi.1)$] The map $\pi \colon \Vv \times \Pp(\Vv) \to \R$ is continuous.
\end{itemize}
Notice that this implies, together with the compactness of $\Vv \times \Pp(\Vv)$, that there exists a constant $C_{\pi}>0$ such that
\begin{align}
    \label{e:pi-bdd}
    | \pi_{\lambda} (\sigma)| \leq C_{\pi} \qquad \text{for every $(\sigma, \lambda) \in \Vv \times \Pp(\Vv)$.}
\end{align}

\begin{remark}[On the strict positivity of the {\em fitness} function]
\label{positiveFitness}
For the discrete Moran process to be mathematically well-defined and biologically meaningful, it is crucial that the fitness function satisfies $F_\lambda(\sigma) > 0$ for any strategy $\sigma \in \mathcal{V}$ and any population distribution $\lambda \in \mathcal{P}(\mathcal{V})$. Indeed, in the Moran process, the probability of an individual being selected for reproduction is proportional to its relative fitness, namely $\frac{\frac{1}{N}F_\lambda(\sigma)}{\overline{F}_\lambda}$ (see below). If the fitness were allowed to vanish, there could exist specific population states $\lambda$ where all agents simultaneously achieve zero fitness. In such a scenario, the average population fitness $\overline{F}_\lambda$ would drop to zero, leading to an undefined $\frac{0}{0}$ transition probability. Bounding the fitness strictly away from zero guarantees that the denominator is always strictly positive, ensuring that the transition matrix of the Markov chain is well-defined at every discrete time step. It follows from \eqref{e:Flambda} and \eqref{e:pi-bdd} that, if $w\in \left [0, \frac{1}{C_{\pi}+1}\right)$, then $F_\lambda(\sigma)> 0$ for every $\sigma \in\Vv$ and $\lambda \in \Pp(\Vv)$. However, since we are interested in a vanishing selection parameter $w$, this assumption on $w$ is naturally satisfied in our main results (Theorems \ref{t:compactness-1} and \ref{t:compactness-3}).
\end{remark}

\begin{remark}[The non-vanishing selection case]
\label{strong-sel}
Notice that in the case of non-vanishing selection parameter (Theorem \ref{t:compactness-3-strong}), i.e. $w\in (0,1]$, then the positivity of the {\em fitness} function is strictly related to the non-negativity of the {\em payoff} function. More precisely, rewriting the fitness as $F_\lambda(\sigma) = (1 - w) + w\pi_\lambda(\sigma)$, it follows that, if
\begin{equation}
\label{strong-sel-ass}
\text{either } w\in (0,1) \text{ and } \pi \geq 0 \qquad \text{or } w=1 \text{ and } \pi>0
\end{equation}
then the fitness $F$ is strictly positive in $\Vv \times \Pp(\Vv)$.
\end{remark}

\medskip

\paragraph{\bf Time-discrete process.}
Assume that either $w \in \left[0, \frac{1}{C_\pi + 1}\right)$ or \eqref{strong-sel-ass} holds.
We consider the following time-discrete process. For $m \in  \{ 1, \ldots, k\}$ and $n \in \{1, \ldots, N\}$ we let $\sigma^{N}_{n} (t^{k}_{m-1}) \colon \Om \to \Vv$ be the strategy of agent $n$ at time $t^{k}_{m-1}$. We let $\bm{\sigma}^{N}(t^{k}_{m-1}) = (\sigma^{N}_1(t^{k}_{m-1}),\dots,\sigma_N^N(t^{k}_{m-1}))$ be the state of the system at time $t^{k}_{m-1}$ and we consider the corresponding empirical distribution of strategies $\lambda^{N}_{t^{k}_{m-1}} \colon \Omega \to \Pp(\Vv)$ defined by $\lambda^{N}_{t^{k}_{m-1}} := \frac{1}{N} \sum_{n=1}^{N} \delta_{\sigma^{N}_{n} (t^{k}_{m-1})}$, and $\Lambda^{N}_{t^{k}_{m-1}} = (\lambda^{N}_{t^{k}_{m-1}})_{\#} \mathbb{P}$. The distribution of strategies at time step $t^{k}_{m-1}$ is conditional on the random state $\bm{\sigma}^{N}(t^{k}_{m-1})$ as described in the following. The probability that the agent $n \in \{ 1, \ldots, N\}$ is chosen to abandon their strategy $\sigma^{N}_{n} (t^{k}_{m-1})$ is $\frac{1}{N}$. The probability that the agent $n' \in \{1, \dots, N\}$ is chosen to replicate their strategy $\sigma^{N}_{n} (t^{k}_{m-1})$ is instead expressed in terms of the relative fitness given by
\begin{align*}
\frac{\frac{1}{N} F_{ \lambda^{N}_{ t^{k}_{m - 1 } } } ( \sigma^{N}_{n'} ( t^{k}_{ m - 1 } ) ) }{ \frac{1}{N} \displaystyle \sum_{\ell =1}^{N} F_{ \lambda^{N}_{ t^{k}_{ m - 1 } }  }  (\sigma^{N}_{\ell} (t^{k}_{m-1}) ) } = \frac{\frac{1}{N} F_{ \lambda^{N}_{ t^{k}_{m - 1 } } } ( \sigma^{N}_{n'} ( t^{k}_{ m - 1 } ) ) }{  \overline{F}_{ \lambda^{N}_{ t^{k}_{ m - 1 } }  }} \,.
\end{align*}
Thanks to Remarks \ref{positiveFitness} and \ref{strong-sel}, this transition is well-defined.
Therefore, at time $t^{k}_{m}$ the strategies $\bm{\sigma}^{N} (t^{k}_{m-1})$ are updated as follows: 
\begin{equation*}
\begin{split}
\sigma^{N}_{\ell} (t^{k}_{m}) & = \sigma^{N}_{\ell} (t^{k}_{m-1} ) \qquad \text{for $\ell \neq n$,}
\\
\sigma^{N}_{n} (t^{k}_{m}) & = \sigma^{N}_{n'} (t^{k}_{m-1} )\,.
\end{split}
\end{equation*}
Such an update occurs with conditional probability
\begin{align}
\label{def-tran-prob}
    \mathbb{P} \Big(\sigma^{N}_{n} (t^{k}_{m}) = \sigma^{N}_{n'} (t^{k}_{m-1}) \ \Big| \ \bm{\sigma}^{N}(t^{k}_{m-1}) \Big) = \frac{\frac{1}{N^{2}} F_{ \lambda^{N}_{ t^{k}_{m - 1 } } } ( \sigma^{N}_{n'} ( t^{k}_{ m - 1 } ) ) }{  \overline{F}_{ \lambda^{N}_{ t^{k}_{ m - 1 } }  }}\,.
\end{align}
Then, we set $\lambda^{N}_{t^{k}_{m}} = \frac{1}{N} \sum_{n = 1}^{N} \delta_{ \sigma^{N}_{n} ( t^{k}_{m} ) }$ and $\Lambda^{N}_{t^{k}_{m}} = (\lambda^{N}_{t^{k}_{m}})_{\#} \mathbb{P} \in \Pp (\Pp( \Vv))$. 

For $t \in [t^{k}_{m-1}, t^{k}_{m}]$ we define the affine interpolation 
\begin{align}
\label{e:def-interpolata-2}
\lambda^{N}_{t} & = \lambda^{N}_{t^{k}_{m-1}} + \frac{t - t^{k}_{m-1}}{\tau_{k}} \big( \lambda^{N}_{t^{k}_{m}} - \lambda^{N}_{t^{k}_{m-1}} \big)  \in \Pp(\Vv)
\\
\label{e:def-interpolata}\Lambda^{N}_{t} & = ( \lambda^{N}_{t})_{\#}\mathbb{P} \in \Pp (\Pp (\Vv))\,.
\end{align}
We further define
\begin{align*}
\overline{\lambda}^{N}_{t} := \lambda^{N}_{t^{k}_{m-1}} \,,\qquad \overline{\Lambda}^{N}_{t} := (\overline{\lambda}^{N}_{t})_{\#} \mathbb{P} = \Lambda^{N}_{t^{k}_{m-1}}
\end{align*}
for $t \in [t^{k}_{m-1}, t^{k}_{m})$.

\begin{lemma}
    \label{l:distance}
    For every $t \in [0, T]$ and $N \in \mathbb{N}\setminus\{0\}$ it holds
    \begin{align*}
        W_{1} (\Lambda^{N}_{t} , \overline{\Lambda}^{N}_{t}) \leq \frac{2}{N}\,.
    \end{align*}
\end{lemma}

\begin{proof}
    The thesis is clearly true for $t = t^{k}_{m-1}$ for some $m \in \{0, \ldots, k\}$. Let us assume that $t \in (t^{k}_{m-1}, t^{k}_{m})$. Then,
    \begin{align*}
        W_{1}(\Lambda^{N}_{t}, \overline{\Lambda}^{N}_{t}) &  = \sup \bigg\{ \int_{\Pp(\Vv)} \psi(\lambda) \, \di (\Lambda^{N}_{t} - \overline{\Lambda}^{N}_{t}) (\lambda) : \, \psi \in {\rm Lip}(\Pp(\Vv)) \text{ with } {\rm Lip}(\psi)\leq 1\bigg\}
        \\
        &
        \leq \frac{t - t^{k}_{m-1}}{\tau_{k}} \, \mathbb{E} \left( \big\| \lambda^{N}_{t^{k}_{m}} - \lambda^{N}_{t^{k}_{m-1}} \big\|_{\rm BL} \right) \,.
    \end{align*}
        By construction, for $\bm{\sigma}^{N}(t^{k}_{m-1})_\#\mathbb{P}$-a.e.\ $\bm{\sigma}$, conditional on $\bm{\sigma}^{N}(t^{k}_{m-1}) = \bm{\sigma}$, we have that 
    \begin{align*}
        \big\| \lambda^{N}_{t^{k}_{m}} - \lambda^{N}_{t^{k}_{m-1}} \big\|_{\rm BL} = \frac{1}{N} \big\| \delta_{\sigma_{n'}} - \delta_{\sigma_{n}} \big\|_{\rm BL} \leq \frac{2}{N} \, ,
    \end{align*}
    with probability $\frac{1}{N^{2}} F_{\lambda } ( \sigma_{n'} ) / \overline{F}_{ \lambda }$ (where $\lambda = \frac{1}{N} \sum_{n=1}^N \delta_{\sigma_n})$.
    Therefore, summing over $n$ and $n'$, we deduce that 
    \begin{align*}
        \big\| \lambda^{N}_{t^{k}_{m}} - \lambda^{N}_{t^{k}_{m-1}} \big\|_{\rm BL} \leq \frac{2}{N} \, ,
    \end{align*}
    with probability 1, whence 
    \begin{align*}
        \mathbb{E}\Big( \big\| \lambda^{N}_{t^{k}_{m}} - \lambda^{N}_{t^{k}_{m-1}} \big\|_{\rm BL} \ \Big| \ \bm{\sigma}^{N}(t^{k}_{m-1}) \Big) \leq 1 \quad \mathbb{P}-\text{a.s.},
    \end{align*}
    which implies
    \begin{align*}
        \mathbb{E} \Big( \big\| \lambda^{N}_{t^{k}_{m}} - \lambda^{N}_{t^{k}_{m-1}} \big\|_{\rm BL} \Big) = \mathbb{E}\Big( \mathbb{E}\Big( \big\| \lambda^{N}_{t^{k}_{m}} - \lambda^{N}_{t^{k}_{m-1}} \big\|_{\rm BL} \ \Big| \ \bm{\sigma}^{N}(t^{k}_{m-1}) \Big) \Big) \leq \frac{2}{N} \,.
    \end{align*}
\end{proof}

\subsection{Notions of solutions and main results}
\label{sub:convergence-thm}

In this subsection we give the notion of solutions to the replicator equation, the replicator equation with genetic drift and the pure genetic drift diffusion equation. We then state the convergence results of the Moran process to the three equations, depending on the scalings of the parameters $N_{k}$, $\tau_{k}$, and $w_{k}$. \\
In what follows, $b:\Pp(\Vv) \to E_{\Pp(\Vv)}$.

\begin{definition}[Weak solution to the  replicator equation]
\label{d:replicator-eq}
We say that a curve $\Lambda\colon [0, T] \to (\Pp(\Pp(\Vv)); W_{1})$ is a weak solution to the replicator equation (with parameter $\gamma>0$) with initial condition $\Lambda(0) = \Lambda_{0} \in \Pp(\Pp(\Vv))$ if for every $\psi \in C^{1}([0, T]\times \Pp(\Vv))$ it holds
\begin{align}
\label{e:weak-replicator}
    \int_{\Pp(\Vv)} \psi (T, \lambda) \, \di \Lambda_{T} (\lambda) - \int_{\Pp(\Vv)} \psi(0, \lambda) \, \di \Lambda_{0} (\lambda) = & \int_{0}^{T} \int_{\Pp(\Vv)} \partial_{t} \psi(t, \lambda) \, \di \Lambda_{t} (\lambda) \, \di t 
    \\
    &
    \nonumber + \gamma\int_{0}^{T} \int_{\Pp(\Vv)} D_{\lambda} \psi(t, \lambda) [b(\lambda)]\, \di \Lambda_{t} (\lambda) \, \di t\,.
\end{align}
\end{definition}

\begin{definition}[Very weak solution to the replicator equation with genetic drift diffusion]
\label{d:weak-replicator-mutator}
We say that a curve $\Lambda\colon [0, T] \to (\Pp(\Pp(\Vv)); W_{1})$ is a very weak solution to the replicator equation with genetic drift diffusion (with parameters $\gamma, \nu >0$) with initial condition $\Lambda(0) = \Lambda_{0} \in \Pp(\Pp(\Vv))$ if for every $\psi \in C^{2}([0, T]\times \Pp(\Vv))$ it holds
\begin{align}
\label{e:fleming-viot}
    \int_{\Pp(\Vv)}  \psi (T, \lambda) \, \di \Lambda_{T} (\lambda) & - \int_{\Pp(\Vv)} \psi(0, \lambda) \, \di \Lambda_{0} (\lambda) = \int_{\Pp(\Vv)} \partial_{t} \psi(t, \lambda) \, \di \Lambda_{t} (\lambda) \, \di t 
    \\
    &
    + \nonumber\gamma\int_{0}^{T} \int_{\Pp(\Vv)} D_{\lambda} \psi(t, \lambda) [b(\lambda)]\, \di \Lambda_{t} (\lambda) \, \di t
    \\
    &
   + \frac{\nu}{2} \int_{0}^{T} \int_{\Pp(\Vv)} \int_{\Vv} \int_{\Vv} D^{2}_{\lambda} \psi(t, \lambda)[\delta_{\sigma} - \delta_{\overline{\sigma}} \ , \ \delta_{\sigma} - \delta_{\overline{\sigma}}] \, \di \lambda(\sigma) \, \di \lambda(\overline{\sigma}) \, \di \Lambda_{t} (\lambda) \, \di t \,. \nonumber
\end{align}
\end{definition}

\begin{definition}[Very weak solution to the pure genetic drift diffusion equation]
\label{d:weak-pure-genetic-drift}
We say that a curve $\Lambda\colon [0, T] \to (\Pp(\Pp(\Vv)); W_{1})$ is a very weak solution to the pure genetic drift diffusion equation (with parameter $\nu >0$) with initial condition $\Lambda(0) = \Lambda_{0} \in \Pp(\Pp(\Vv))$ if for every $\psi \in C^{2}([0, T]\times \Pp(\Vv))$ it holds
\begin{align}
\label{e:fleming-viot-puro}
    \int_{\Pp(\Vv)}  \psi (T, \lambda) \, \di \Lambda_{T} (\lambda) & - \int_{\Pp(\Vv)} \psi(0, \lambda) \, \di \Lambda_{0} (\lambda) = \int_{\Pp(\Vv)} \partial_{t} \psi(t, \lambda) \, \di \Lambda_{t} (\lambda) \, \di t 
    \\
    &
   + \frac{\nu}{2} \int_{0}^{T} \int_{\Pp(\Vv)} \int_{\Vv} \int_{\Vv} D^{2}_{\lambda} \psi(t, \lambda)[\delta_{\sigma} - \delta_{\overline{\sigma}} \ , \ \delta_{\sigma} - \delta_{\overline{\sigma}}] \, \di \lambda(\sigma) \, \di \lambda(\overline{\sigma}) \, \di \Lambda_{t} (\lambda) \, \di t \,. \nonumber
\end{align}
\end{definition}

\begin{remark}
\label{r:kimura-fleming-viot}
     In Section~\ref{S:examples} we will consider the case $\Vv = \Pp(U)$ for a finite set $U = \{u_{1}, \ldots, u_{M}\}$ of pure strategies, where we identify $\Pp(\Vv)$ with the simplex $\Delta^{M-1}$ of $\R^{M}$. In such a framework, we will see that the right-hand side of~\eqref{e:fleming-viot} involves a Fleming--Viot operator (cf.~\eqref{e:blambda} and~\cite{Fle-Vio}) and the strong version of~\eqref{e:fleming-viot} writes as in~\eqref{e:strong-equation} using the tangential divergence in the simplex~$\Delta^{M-1}$. For $M=2$, equation~\eqref{e:strong-equation} reduces to the Kimura equation~\cite{Kimura1, Kimura2} (see \eqref{e:kimura-example}), and Definition~\ref{d:weak-replicator-mutator} can be compared with the corresponding notion given in~\cite[Definition~3.1]{Casteras-Monsaingeon}. As also discussed in~\cite{Casteras-Monsaingeon}, the Definition~\ref{d:weak-replicator-mutator} of very weak solution for the replicator equation with genetic drift does not carry any extra boundary conditions. In contrast to~\cite{Carrilloetal-2022}, where a no-flux condition is imposed, this can lead to extinction and fixation effects, i.e., concentration effects on the boundary of~$\Delta^{M-1}$ in finite time.
\end{remark}

We are now in a position to state the convergence results concerning the replicator equation with or without genetic drift and the pure genetic drift diffusion equation. The proofs of the next two theorems are postponed to Sections~\ref{sub:thm1} and~\ref{sub:thm2}, respectively. We recall that $\tau_k = \frac{T}{k}$, so $\tau_k \to 0$ as $k\to +\infty$. \\

\begin{theorem}
    \label{t:compactness-1}
Let $N_k\to \infty$ and $w_k\to 0$ as $k\to \infty$. Let $\Lambda^{N_{k}} \in C([0, T]; \Pp(\Pp(\Vv)))$ be the curves defined in~\eqref{e:def-interpolata-2}--\eqref{e:def-interpolata}. Assume that $(\pi.1)$ holds and that 
   \begin{align}
       \label{e:scalings} & \lim_{k \to \infty} \frac{w_{k}}{\tau_{k} N_{k}} = \gamma \geq 0 \,, \qquad \lim_{k \to \infty} \frac{1}{\tau_{k} N^{2}_{k}} = \nu >0 \,.
   \end{align}
   Let $b(\lambda)=(\pi_\lambda-\overline{\pi}_\lambda)\lambda$. Then, there exists $\Lambda\colon [0, T] \to \Pp(\Pp(\Vv))$ such that, up to a not relabeled subsequence, $W_{1} (\Lambda^{N_{k}}_{t} , \Lambda_{t}) \to 0$ as $k \to \infty$ for every $t \in [0, T]$, and
   \begin{itemize}
       \item if $\gamma>0$, $\Lambda$ is a very weak solution to the replicator equation with genetic drift in the sense of Definition \ref{d:weak-replicator-mutator},
       \item if $\gamma=0$, $\Lambda$ is a very weak solution to the pure genetic drift diffusion equation in the sense of Definition \ref{d:weak-pure-genetic-drift}.
   \end{itemize}
    In particular, $t \mapsto \Lambda_{t}$ is Lipschitz continuous in the $\di_{\mathcal{Z}}$-metric and continuous in the $W_{1}$-metric of $\Pp(\Pp(\Vv))$. 
\end{theorem}

\begin{remark}
Let us notice that Theorem \ref{t:compactness-1} with $\gamma=0$ also includes the case of {\em neutral selection}, in which $w=w_k=0$ for every $k\in \mathbb{N}$.
\end{remark}

\begin{theorem}
    \label{t:compactness-3}
Let $N_k\to \infty$ and $w_k\to 0$ as $k\to \infty$. Let $\Lambda^{N_{k}} \in C([0, T]; \Pp(\Pp(\Vv))$ be the curves defined in~\eqref{e:def-interpolata-2}--\eqref{e:def-interpolata}. Assume that $(\pi.1)$ holds and that 
   \begin{align}
       \label{e:scalings-2} & \lim_{k \to \infty} \frac{w_{k}}{\tau_{k} N_{k}} = \gamma >0 \,, \qquad \lim_{k \to \infty} \frac{1}{\tau_{k} N^{2}_{k}} = 0 \,.
   \end{align}
Let $b(\lambda)=(\pi_\lambda-\overline{\pi}_\lambda)\lambda$. Then, there exists a weak solution to the replicator equation $\Lambda\colon [0, T] \to (\Pp(\Pp(\Vv)), W_{1})$ such that, up to a (not relabeled) subsequence, $W_{1} (\Lambda^{N_{k}}_{t} , \Lambda_{t}) \to 0$ as $k \to \infty$ for every $t \in [0, T]$. Moreover,~$\Lambda$ is Lipschitz continuous with respect to the $W_{1}$-metric of $\Pp (\Pp(\Vv))$.
\end{theorem}

Finally, we also state the result in case of non-vanishing selection parameter $w$ (as in the case of the {\em strong selection}), in which $w\in (0,1]$; the proof is postponed to Section \ref{sub:thm3-strong}.
\begin{theorem}
\label{t:compactness-3-strong}
Let $N_k\to \infty$ as $k\to \infty$. Let $\Lambda^{N_{k}} \in C([0, T]; \Pp(\Pp(\Vv))$ be the curves defined in~\eqref{e:def-interpolata-2}--\eqref{e:def-interpolata}. Assume that $(\pi.1)$ and \eqref{strong-sel-ass} hold, and that 
\begin{align}
\label{e:scaling-strong} \lim_{k \to \infty} \frac{1}{\tau_{k} N_{k}} = \gamma >0.
\end{align}
Let $\displaystyle b(\lambda)= \frac{w(\pi_\lambda-\overline{\pi}_\lambda)\lambda}{1+w(\overline{\pi}_\lambda-1)}  $. Then, there exists a weak solution to the replicator equation $\Lambda\colon [0, T] \to (\Pp(\Pp(\Vv)), W_{1})$ such that, up to a (not relabeled) subsequence, $W_{1} (\Lambda^{N_{k}}_{t} , \Lambda_{t}) \to 0$ as $k \to \infty$ for every $t \in [0, T]$. Moreover,~$\Lambda$ is Lipschitz continuous with respect to the $W_{1}$-metric of $\Pp (\Pp(\Vv))$.
\end{theorem}

\subsection{Remarks on the notions of solution}
\label{sub:comments}

We conclude the section with a few comments concerning the Definitions~\ref{d:replicator-eq} and~\ref{d:weak-replicator-mutator}, focusing on uniqueness and/or comparison with previous literature.

\medskip

\paragraph{\bf Structure of solutions to the replicator equation.} We start with the structure of solutions to the replicator equation~\eqref{e:weak-replicator}. We remark that a weak solution $\Lambda \colon [0, T] \to   \Pp (\Pp(\Vv))$ to the replicator equation solves the following continuity equation in a distributional sense
\begin{align}
\label{e:cont-eq}
\partial_{t} \Lambda + \gamma \, {\rm div}_{\lambda} (b (\lambda) \Lambda) = 0 \,.
\end{align}
Under the sole continuity assumption~$(\pi.1)$, we deduce from~\cite[Section 8.1]{AGS} that $\Lambda$ has a continuous representative. Furthermore, an application of the superposition principle~\cite[Theorem~5.2]{AFMS} guarantees that there exists a probability measure $\boldsymbol{\eta} \in \Pp (C([0, T]; \Pp (\Vv)))$ such that $\Lambda_{t} = (e_{t})_{\#} \eeta$, where for $t \in [0, T]$ the evaluation map $e_{t}\colon C([0, T]; \Pp (\Vv)) \to \Pp(\Vv)$ is defined by $e_{t} (\lambda) = \lambda_{t}$ for every $\lambda \in C([0, T]; \Pp (\Vv))$. Moreover, $\eeta$ is supported on curves $\lambda \in C([0, T]; \Pp (\Vv))$ solutions to the ODE
\begin{align}
\label{e:cauchy-pb}
    \dot{\lambda}_{t} = \gamma \,  b (\lambda_{t})\,, \qquad \lambda_{0} \in {\rm spt} (\Lambda_{0})\,.
\end{align}
Notice that~\cite[Section I.3, Theorem 1.4, Corollary 1.1]{Brezis} yields the existence of solutions to the Cauchy problem~\eqref{e:cauchy-pb} for any initial condition $\lambda_{0} \in {\spt} (\Lambda_{0})$. This provides the relation between the notion of weak solution to the replicator equation and the classical replicator dynamics~\eqref{intro:replicator} and~\eqref{e:cauchy-pb}. Indeed, Definition~\ref{d:replicator-eq} and the continuity equation~\eqref{e:cont-eq} are the Eulerian counterpart to the (Lagrangian) replicator dynamics~\eqref{e:cauchy-pb}. 

Whenever $\pi \colon \Vv \times \Pp(\Vv) \to \R$ is assumed to be Lipschitz continuous in $\Vv \times \Pp(\Vv)$ endowed with the product metric, we further infer that the solution to~\eqref{e:cauchy-pb} is unique for any initial condition $\lambda_{0} \in \Pp(\Vv)$. In that case, letting $\Psi\colon [0, T] \times \Pp(\Vv) \to \Pp(\Vv)$ be the corresponding flow map, we have that $\Lambda_{t} = (\Psi (t, \cdot) )_{\#} \Lambda_{0}$.

\medskip

\paragraph{\bf Replicator equation with genetic drift and cylindrical tests.}
We briefly comment on the notion of very weak solution to the replicator equation with genetic drift~\eqref{e:fleming-viot}, as the second order term has no explicit form in general (we refer to Section~\ref{s:pure} for some explicit examples for a finite number of pure strategies). We consider a special class of test function made of cylindrical tests (see for instance~\cite[Definition 5.1.11]{AGS}) of the form
\begin{align}
\label{e:special-test}
    \psi (\lambda) = \varphi \big( \langle g_{1}  , \lambda\rangle , \ldots, \langle g_{m}  , \lambda\rangle \big)\,,  
\end{align}
for $\lambda \in \Pp(\Vv)$, $m \in \mathbb{N} \setminus\{0\}$, $\varphi \in C^{2} (\R^{m})$, and $g_{1}, \ldots, g_{m} \in C (\Vv)$. In~\eqref{e:special-test} we have introduced the short-hand notation $\langle g_{i}, \lambda\rangle := \int_{\Vv} g_{i} (\sigma) \di \lambda(\sigma)$. Then, $\psi \in C^{2} (\Pp(\Vv))$ with
\begin{align}
\label{e:cyl-gradient}
    & D_{\lambda} \psi (\lambda) [\mu] = \sum_{i=1}^{m} \partial_{i} \varphi \big( \langle g_{1}  , \lambda\rangle , \ldots, \langle g_{m}  , \lambda\rangle \big) \langle g_{i}  , \mu \rangle  \,,
    \\
    &
    \label{e:cyl-hessian}
    D^{2}_{\lambda} \psi (\lambda) [\mu, \theta] = \sum_{i, j=1}^{m} \partial^{2}_{ij} \varphi \big(\langle g_{1}  , \lambda\rangle , \ldots, \langle g_{m}  , \lambda\rangle \big)  \langle g_{i} , \mu \rangle \, \langle g_{j} ,  \theta \rangle
\end{align}
for every $\mu, \theta \in \R ( \Pp(\Vv) - \Pp(\Vv))$. 

If $\mu = b(\lambda) = (\pi_{\lambda} (\cdot) - \overline{\pi}_{\lambda}) \lambda \in  \R ( \Pp(\Vv) - \Pp(\Vv))$,~\eqref{e:cyl-gradient} writes as
\begin{align*}
  D_{\lambda} \psi (\lambda) [b(\lambda)] = \sum_{i=1}^{m} & \partial_{i} \varphi \big( \langle g_{1}  , \lambda\rangle , \ldots, \langle g_{m}  , \lambda\rangle \big) \int_{\Vv} g_{i} (\sigma) ( \pi_{\lambda} (\sigma) - \overline{\pi}_{\lambda})\, \di \lambda (\sigma) \,,  
\end{align*}
If $\mu = \theta = \delta_{\sigma} - \delta_{\overline{\sigma}}$ for $\sigma, \overline{\sigma} \in \Vv$,~\eqref{e:cyl-hessian} becomes
\begin{align*}
D^{2}_{\lambda} \psi (\lambda) [\delta_{\sigma} - \delta_{\overline{\sigma}}, \delta_{\sigma} - \delta_{\overline{\sigma}}] = \sum_{i, j=1}^{m}  & \partial^{2}_{ij} \varphi \big( \langle g_{1}  , \lambda\rangle , \ldots, \langle g_{m}  , \lambda\rangle \big) \big( g_{i} (\sigma) - g_{i} (\overline{\sigma}) \big) \big( g_{j} (\sigma) - g_{j} (\overline{\sigma}) \big)\,.
\end{align*}
Therefore, equation~\eqref{e:fleming-viot} for the tests~\eqref{e:special-test} reads
\begin{align}
\label{e:special-eq}
    & \int_{\Pp(\Vv)}  \psi (\lambda) \, \di (\Lambda_{T} - \Lambda_{0})(\lambda) =
    \\
    &
    \quad \gamma\int_{0}^{T} \int_{\Pp(\Vv)} \sum_{i=1}^{m}  \partial_{i} \varphi \big( \langle g_{1}  , \lambda\rangle , \ldots, \langle g_{m}  , \lambda\rangle \big) \int_{\Vv} g_{i} (\sigma) ( \pi_{\lambda} (\sigma) - \overline{\pi}_{\lambda})\, \di \lambda (\sigma)\, \di\Lambda_{t} (\lambda) \di t \nonumber
    \\
    &
    \quad + \nu \int_{0}^{T} \int_{\Pp(\Vv)} \sum_{i, j=1}^{m} \partial^{2}_{ij} \varphi \big(\langle g_{1}  , \lambda\rangle , \ldots, \langle g_{m}  , \lambda\rangle \big) \big(\langle g_{i} g_{j}, \lambda \rangle - \langle g_{i}, \lambda\rangle \, \langle g_{j}, \lambda\rangle \big) \, \di \Lambda_{t} (\lambda) \di t \,. \nonumber
\end{align}
In particular, the resulting equation~\eqref{e:special-eq} coincides with the Fleming--Viot equation considered in~\cite[Section~5]{Fle-Vio} for a general compact metric space $\Vv$ of strategies. We refer to Section~\ref{s:pure} for simplified expressions of~\eqref{e:fleming-viot} for a finite set of strategies~$\Vv$.

\section{Relevant examples and applications}
\label{S:examples}

In this section, we briefly discuss a few model cases fitting the presented theoretical framework. The general setting we consider is that of multi-agent multi-label systems, introduced in~\cite{AFMS, MS2020}. In our case the limit dynamics can be decomposed into two fundamental components: the replicator term, which accounts for natural selection (where strategies with higher-than-average fitness increase their prevalence in the population), and the Genetic Drift term, which represents the stochastic variations or errors that effectively diffuse the population density across the strategy space. 

\subsection{Mixed strategies}
\label{sub:mixed}
The starting point is a Moran process, which is a stochastic model introduced to describe mutations and genetic drift in a finite population \cite{Nowak}, for which any agents in a population of $N$-elements adopt a mixed strategy. Mathematically, the evolution consists of a Markov chain, where, at each discrete time step, one of the $N$ agents is selected to replicate their strategy with a probability proportional to the fitness of their strategy, while one of the $N$ agents is chosen uniformly to abandon their strategy.

We fix $(U, \mathrm{d}_U)$ a compact metric space of pure strategies and $\mathcal{P}(U)$ the associated space of probability measures over $U$ which represents the set of mixed strategies. We endow $\mathcal{P}(U)$ with the metric $\di_{\mathrm{BL}}$ induced by the space $\left(\mathcal{F}(U), \|\cdot\|_{\mathrm{BL}}\right)$ (see Section \ref{S:pre} for the definition of the Arens-Eells space $\mathcal{F}(U)$). This setting makes $\mathcal{P}(U)$ a compact and convex subset of $\mathcal{F}(U)$ (see Section 2.1 of \cite{AFMS}). In comparison with Sections~\ref{s:discrete-process}--\ref{s:compactness}, we remark that $\Pp(U)$ plays here the role of the compact metric space $\Vv$.

For $\lambda\in \mathcal{P}\left(\mathcal{P}(U)\right)$, we define the {\em payoff} function $\tilde{\pi}_\lambda \colon U \to \mathbb{R}$ as follows
\begin{equation}
    \label{e:J-payoff}
\tilde{\pi}_\lambda(u) := \int_{\mathcal{P}(U)}\int_U J(u,u') \di\tilde{\sigma}(u')\di \lambda(\tilde\sigma),
\end{equation}
where $J \colon U\times U \to \mathbb{R}$ is a continuous function. The associated {\em fitness} function $f_\lambda: U \to \mathbb{R}$ is defined by
$$
f_\lambda(u)=1+ w(\tilde{\pi}_\lambda(u)-1).
$$
It follows that the {\em fitness} function defined in \eqref{e:Flambda} and the related {\em payoff} function are given by
\begin{equation}
    \label{e:fit-payoff-mixed}
    F_\lambda(\sigma):=\int_U f_\lambda(u)\di\sigma(u) \qquad \text{and}\qquad \pi_\lambda(\sigma):= \int_U \tilde{\pi}_\lambda(u)\di\sigma(u).
\end{equation}
Note that, since $J$ is continuous, since $\mathcal{P}(U) \times \mathcal{P}(\mathcal{P}(U))$ is compact, and since the $\mathrm{BL}$ norm induces the weak-star topology in $\mathcal{P}(U)$, then $\pi \colon \mathcal{P}(U) \times \mathcal{P}(\mathcal{P}(U)) \to \mathbb{R}$ defined by $\pi(\sigma, \lambda):=\pi_\lambda(\sigma)$ is continuous, whence $(\pi.1)$ holds. Thus, we can apply Theorems \ref{t:compactness-1} and~\ref{t:compactness-3} to study the limit as $N \to +\infty$ of the number of agents.

\subsection{Finite, pure strategies: Fleming--Viot operator and Kimura equation}
\label{s:pure}
To have a better comparison with the existing literature and with the analysis of~\cite{MoOr}, we present the equation obtained in Theorem~\ref{t:compactness-1} and Definition~\ref{d:weak-replicator-mutator} in the particular case where the set of pure strategies $U =\{ u_{1}, \ldots, u_{M}\}$ is finite and each agent can only have a pure strategy $u_{i} \in U$ at disposal. For simplicity, we assume $\gamma =\nu= 1$.
We endow $U$ with the metric $\di_{U} \colon U \times U \to \{0, 1\}$ defined by
\begin{align*}
    \di_{U} (u_{i}, u_{j}) = \delta_{i}^{j} \qquad \text{for every $i, j \in \{1, \ldots, M\}$,}
\end{align*}
where $\delta_{i}^{j}$ denotes the Kronecker delta. This setting is simpler compared to that described in Section~\ref{sub:mixed}: the compact metric space $\Vv$ identifies now with $U$ itself, rather than with~$\Pp(U)$, as only pure strategies are allowed, which would corresponds with Dirac deltas in $\Pp(U)$ centered in the pure strategies $u_{i} \in U$. The distribution of strategies inside the population is represented by $\lambda \in \Pp(\Vv) = \Pp(U)$, which we may as well identify with the simplex
\begin{align*}
    \Delta^{M-1} := \bigg\{ \lambda = (\lambda_{1}, \ldots, \lambda_{M}) \in [0, 1]^{M}: \, \sum_{i=1}^{M} \lambda_{i} = 1 \bigg\}\,.
\end{align*}
Notice that the components $\lambda_{i}$ of $\lambda \in \Delta^{M-1}$ represent the frequency of the pure strategy $u_{i}$. We further remark that if $\psi \in C^{1} (\Delta^{M-1})$ is the restriction of a map $\psi \in C^{1} (\R^{M})$ to the simplex~$\Delta^{M-1}$, then the $C$-differential $D_{\lambda}\psi$ writes as
\begin{align}
\label{e:example-Cdiff}
    D_{\lambda} \psi (\lambda) = P \,\nabla \psi (\lambda) \qquad \lambda \in \Delta^{M-1}\,,
\end{align}
where $\nabla \psi$ stands for the usual gradient of~$\psi$ in $\R^{M}$ and $P \in \mathbb{M}^{M}$ denotes the projection over the hyperplane parallel to $\Delta^{M-1}$ and passing through the origin. Namely,
\begin{align*}
    P = I - \frac{1}{M} \xi \otimes \xi\,, \qquad \text{with } \xi := \left(\begin{array}{c}
    1 \\
    \vdots\\
    1
    \end{array}\right).
\end{align*}
Similarly, if $\psi \in C^{2} (\R^{M})$ then
\begin{align}
\label{e:example-Chess}
    D^{2}_{\lambda} \psi (\lambda) = P \, \nabla^{2} \psi(\lambda)\, P \qquad \lambda \in \Delta^{M-1}\,,
\end{align}
where $\nabla^{2}\psi$ is the Hessian of~$\psi$. 

As done in~\cite{MoOr}, we represent the payoff function~$J$ introduced in~\eqref{e:J-payoff} with a matrix $A \in \mathbb{M}^{M}$. Hence, for $\lambda \in \Delta^{M-1}$ and $u_{i} \in U$ the payoff $\pi_{\lambda} (u_{i})$ takes the form
\begin{align*}
    \pi_{\lambda} (u_{i})  = \sum_{j=1}^{M} A_{ij} \lambda_{j} = (A\lambda)_{i}\,.
\end{align*}
The fact that we only allow for pure strategies spares us two integrations with respect to $\sigma$, as it has been done in~\eqref{e:J-payoff} and~\eqref{e:fit-payoff-mixed}. We further notice that for $\lambda \in \Delta^{M-1}$ we have that
\begin{align*}
    \overline{\pi}_{\lambda} = \sum_{i, j=1}^{M} A_{ij} \lambda_{i} \lambda_{j} = A\lambda \cdot \lambda\,,
\end{align*}
which stands for the average payoff associated to~$\lambda$. For simplicity of notation and in agreement with~\cite{MoOr} we write $b(\lambda) \in \R^{M}$ as the vector with components
\begin{align}
\label{e:def-b}
    b_{i}(\lambda) := \big( ( A \lambda)_{i} - A\lambda \cdot \lambda \big) \lambda_{i} \qquad \text{for every $\lambda \in \Delta^{M-1}$.}
\end{align}
With this notation at hand, we rewrite the last two terms on the right-hand side of~\eqref{e:fleming-viot} for $\Lambda \in C([0, T]; (\Pp(\Delta^{M-1}); W_{1}))$. Precisely, for $\psi \in C^{2}([0, T] \times \Delta^{M-1})$ we have that
\begin{align}
\label{e:blambda0}
\int_{0}^{T}\int_{\Delta^{M-1}} &  D_{\lambda}\psi(t, \lambda) [( \pi_{\lambda} - \overline{\pi}_{\lambda} ) \lambda] \, \di \Lambda_{t} (\lambda) \, \di t  = \int_{0}^{T}\int_{\Delta^{M-1}} D_{\lambda}\psi(t, \lambda) \cdot b(\lambda) \, \di \Lambda_{t} (\lambda) \, \di t
\\
&
= \int_{0}^{T} \int_{\Delta^{M-1}} \nabla \psi (t, \lambda) \cdot P \, b(\lambda)\, \di \Lambda_{t} (\lambda) \, \di t = \int_{0}^{T} \int_{\Delta^{M-1}} \nabla \psi (t, \lambda) \cdot  b(\lambda)\, \di \Lambda_{t} (\lambda) \, \di t\,, \nonumber
\end{align}
where in the last equality we have used that $b(\lambda) \cdot \xi = 0$ for $\lambda \in \Delta^{M-1}$.

As for the last term on the right-hand side of~\eqref{e:fleming-viot}, we once again recall that $\sigma$ can only take values in the finite set~$U$. In view of the identification $\Pp(U) = \Delta^{M-1}$, we notice that for every $i\in \{1, \ldots, M\}$ we are identifying $\delta_{u_{i}}$ with the $i$-th vector~$e_{i}$ of the canonical basis of~$\R^{M}$. Hence, we have that
\begin{align*}
     \int_{0}^{T} \int_{\Pp(\Vv)} &  \int_{\Vv} \int_{\Vv} D^{2}_{\lambda} \psi(t, \lambda)[\delta_{\sigma} - \delta_{\overline{\sigma}} \ , \ \delta_{\sigma} - \delta_{\overline{\sigma}}] \, \di \lambda(\sigma) \, \di \lambda(\overline{\sigma}) \, \di \Lambda_{t} (\lambda) \, \di t
\end{align*}
is identified with
\begin{align}
\label{e:formal}
     \int_{0}^{T} \int_{\Delta^{M-1}}  \sum_{i, j=1}^{M} \lambda_{i} \lambda_{j}D_{\lambda}^{2} \psi(t, \lambda) [e_{i} - e_{j} , e_{i} - e_{j}]\, \di \Lambda_{t} (\lambda) \, \di t \,.
\end{align}
Rearranging the terms on the right-hand side of~\eqref{e:formal}, using the fact that $\lambda \in \Delta^{M-1}$ and the relation~\eqref{e:example-Chess} we deduce that
\begin{align}
    \label{e:formal2}
    \int_{0}^{T}  \int_{\Delta^{M-1}} &   \sum_{i, j=1}^{M} \lambda_{i} \lambda_{j}D_{\lambda}^{2} \psi(t, \lambda) [e_{i} - e_{j} , e_{i} - e_{j}]\, \di \Lambda_{t} (\lambda) \, \di t
    \\
    &
    \quad = \int_{0}^{T} \int_{\Delta^{M-1}} 2 \sum_{i=1}^{M} \lambda_{i} \nabla^{2} \psi (t, \lambda) [P e_{i}, P e_{i}] \, \di \Lambda_{t} (\lambda) \, \di t \nonumber
    \\
    &
    \qquad - \int_{0}^{T} \int_{\Delta^{M-1}} 2\sum_{i, j=1}^{M} \lambda_{i} \lambda_{j} \nabla^{2} \psi (t, \lambda) [P e_{i}, P e_{j}]\, \di \Lambda_{t}(\lambda) \, \di t  \nonumber
    \\
    &
    \quad = 2 \int_{0}^{T} \int_{\Delta^{M-1}} \sum_{i, j=1}^{M} ( \lambda_{i} \delta_{i}^{j} -  \lambda_{i} \lambda_{j} ) \nabla^{2} \psi (t, \lambda) [P e_{i}, P e_{j}]  \, \di \Lambda_{t}(\lambda) \, \di t\,. \nonumber
\end{align}
 We define $B \colon \Delta^{M-1} \to \mathbb{M}^{M}$ by
\begin{align*}
    B_{ij}(\lambda) := ( \lambda_{i} \delta_{i}^{j} -  \lambda_{i} \lambda_{j} ) \qquad \text{for $i, j \in \{1, \ldots, M\}$ and $\lambda \in \Delta^{M-1}$.}
\end{align*}
Then, we rewrite~\eqref{e:formal2} as
\begin{align}
    \label{e:formal30}
    \int_{0}^{T} &  \int_{\Delta^{M-1}}     \sum_{i, j=1}^{M} \lambda_{i} \lambda_{j}D_{\lambda}^{2} \psi(t, \lambda) [e_{i} - e_{j} , e_{i} - e_{j}]\, \di \Lambda_{t} (\lambda) \, \di t
    \\
    &
     = 2 \int_{0}^{T} \int_{\Delta^{M-1}} P B(\lambda)P : \nabla^{2} \psi(t, \lambda) \, \di \Lambda_{t} (\lambda)\, \di t =  2 \int_{0}^{T} \int_{\Delta^{M-1}}  B(\lambda) : \nabla^{2} \psi(t, \lambda) \, \di \Lambda_{t} (\lambda)\, \di t\,, \nonumber
\end{align}
where, in the last equality, we use that $P B(\lambda) P = B(\lambda)$, as a direct computation shows. We define 
\begin{align}
\label{e:blambda}
\mathcal{L}(\phi)(\lambda) & \coloneqq D_{\lambda}\phi(\lambda) \cdot b(\lambda) +  D^{2}_{\lambda} \phi (\lambda) : B(\lambda)  \\
& = \sum_{i=1}^{M}  b_i(\lambda) D_{\lambda}\phi(\lambda)[e_i] + \sum_{i, j=1}^{M} ( \lambda_{i} \delta_{i}^{j} -  \lambda_{i} \lambda_{j} ) D^{2}_{\lambda} \phi (\lambda) [e_{i}, e_{j}]  \nonumber
\end{align}
and notice that if $\phi \in C^{2} (\R^{M})$, then
\begin{align*}
    \mathcal{L}(\phi)(\lambda) & =  D_{\lambda}\phi(\lambda) \cdot b(\lambda) +  D^{2}_{\lambda} \phi (\lambda) : B(\lambda) =  \nabla \phi(\lambda) \cdot b(\lambda) +  \nabla^{2} \phi (\lambda) : B(\lambda)
    \\
    &
    = \sum_{i=1}^{M}  b_i(\lambda) \partial_{i} \phi(\lambda) + \sum_{i, j=1}^{M} ( \lambda_{i} \delta_{i}^{j} -  \lambda_{i} \lambda_{j} ) \partial_{i} \partial_{j} \phi (\lambda)\,, 
\end{align*}
so that $\mathcal{L}(\phi)$ defines the Fleming--Viot operator in the form of~\cite[Eq.~(2.1)]{Fle-Vio}.

Then, combining~\eqref{e:fleming-viot},~\eqref{e:blambda0}, \eqref{e:formal} and~\eqref{e:formal30}, we infer that in the interior of~$\Delta^{M-1}$ the curve $t \mapsto \Lambda_{t}$ solves (in the sense of Definition~\ref{d:weak-replicator-mutator}) the equation
\begin{align}
\label{e:strong-equation}
    \partial_{t} \Lambda_{t} + {\rm div}_{\lambda} \big[ b(\lambda) \Lambda_{t} \big] -  {\rm div}_{\lambda} \big[ {\rm div}_{\lambda}\big ( B(\lambda) \Lambda_{t} \big) \big] = 0\,,
\end{align}
obtained as the adjoint of the Fleming--Viot operator~\eqref{e:blambda}. Here, ${\rm div}_{\lambda}$ denotes the tangential divergence on $\Delta^{M-1}$, defined as the adjoint operator to the $C$-differential (or tangential gradient) $D_{\lambda}$. We further notice that ${\rm div}_{\lambda} {\rm div}_{\lambda}$ is the adjoint to $D^{2}_{\lambda}$ and corresponds to applying twice the tangential divergence.

The particular case $M=2$ leads to the well-known Kimura equation~\cite{Kimura1, Kimura2}, which we can directly deduce from~\eqref{e:blambda} and~\eqref{e:fleming-viot} after identification of the $C$-differential $D_{\lambda} \phi$ given in~\eqref{e:example-Cdiff}. Assume that the test function $\phi$ is the restriction to $\Delta^{1}$ of a $C^{2}$-function of the wholel space~$\R^{2}$. Then, the tangential gradient $D_{\lambda} \phi$ writes for $\lambda = (\lambda_{1} ,  \lambda_{2}) \in \Delta^{1}$ as
\begin{align*}
    D_{\lambda} \phi (\lambda) = \frac{1}{2}
    \left( 
    \begin{array}{c}
    \partial_{1} \phi (\lambda) - \partial_{2} \phi (\lambda) \\
    \partial_{2} \phi (\lambda) - \partial_{1} \phi (\lambda)
    \end{array}
    \right),
\end{align*}
where $\partial_{i}$ denotes the usual partial derivative with respect to $\lambda_{i}$. Similarly, we have that
\begin{align*}
    D^{2}_{\lambda} \phi (\lambda) = \frac{1}{4}\left(
\begin{array}{cc}
\partial^{2}_{1} \phi (\lambda) - \partial_{1} \partial_{2} \phi(\lambda) & \partial_{1} \partial_{2} \phi(\lambda) - \partial^{2}_{2} \phi(\lambda) \\
 \partial_{1} \partial_{2} \phi(\lambda) - \partial^{2}_{1} \phi (\lambda) & \partial^{2}_{2} \phi(\lambda) - \partial_{1} \partial_{2} \phi(\lambda) 
\end{array}
    \right).
\end{align*}
Further notice that defining $\varphi (\lambda_{1}) := \phi (\lambda_{1}, 1 - \lambda_{1})$ for $\lambda_{1} \in [0, 1]$, then, using the constraint $\lambda_{2} = 1 - \lambda_{1}$ in $\Delta^{1}$, we get that
\begin{align*}
     D_{\lambda} \phi (\lambda_{1}, 1 - \lambda_{1}) = 
      \frac{1}{2}
    \left( 
    \begin{array}{c}
    \varphi' (\lambda_{1})  \\
    - \varphi'(\lambda_{1})
    \end{array}
    \right) \qquad 
     D^{2}_{\lambda} \phi (\lambda_{1}, 1 - \lambda_{1}) = \frac{1}{4}\left(
\begin{array}{cc}
\varphi''(\lambda_{1})  & - \varphi''(\lambda_{1}) \\
 - \varphi''(\lambda_{1}) &   \varphi''(\lambda_{1})
\end{array}
    \right).
\end{align*}
Hence, recalling \eqref{e:def-b}, we write \eqref{e:blambda} as
\begin{align*}
    & \mathcal{L}   (\phi) (\lambda_{1}, 1 - \lambda_{1}) \\
    & =  \frac{1}{2}\lambda_{1} \varphi'(\lambda_{1}) \big( a_{11} \lambda_{1} + a_{12} (1 - \lambda_{1}) - a_{11} \lambda_{1}^{2} - a_{12} \lambda_{1} (1 - \lambda_{1}) - a_{21} \lambda_{1} (1 - \lambda_{1}) - a_{22} (1 - \lambda_{1}^{2}) \big) 
    \\
    &
    \quad - \frac{1}{2} ( 1 - \lambda_{1})  \varphi'(\lambda_{1}) \big( a_{21} \lambda_{1} + a_{22} (1 - \lambda_{1}) - a_{11} \lambda_{1}^{2} - a_{12} \lambda_{1} (1 - \lambda_{1}) - a_{21} \lambda_{1} (1 - \lambda_{1}) - a_{22} (1 - \lambda_{1}^{2}) \big)
    \\
    &
    \quad \vphantom{\frac{1}{2}} + \lambda_{1} ( 1 - \lambda_{1}) \varphi''(\lambda_{1}) 
    \\
    &
    =  \lambda_{1} ( 1 - \lambda_{1})  \varphi'(\lambda_{1}) \big(a_{11} \lambda_{1} + a_{12} ( 1  -  \lambda_{1} ) - a_{21} \lambda_{1} - a_{22} (1 - \lambda_{1}) \big) +  \lambda_{1} ( 1 - \lambda_{1}) \varphi''(\lambda_{1})\,.
\end{align*}
Taking $H\colon \R \to \R$ such that $H'(r) = a_{11} r + a_{12} ( 1  -  r ) - a_{21} r - a_{22} (1 - r)$, equation~\eqref{e:strong-equation} reduces to the Kimura equation
\begin{align}
\label{e:kimura-example}
    \partial_{t} \Lambda_{t} + \partial_{\lambda_{1}}  \big[ \lambda_{1} ( 1 - \lambda_{1}) H'(\lambda_{1}) \Lambda_t \big] -  \partial^{2}_{\lambda_{1}}\big( \lambda_{1} (1 -\lambda_{1}) \Lambda_{t} \big) = 0\,,
\end{align}
where $\Lambda_{t}$ can be identified with a curve of probability measures over $\mathcal{P} ([0, 1])$. 

Equation~\eqref{e:kimura-example} has been studied, for instance, in~\cite{Carrilloetal-2022, Casteras-Monsaingeon}. As noted in Remark~\ref{r:kimura-fleming-viot}, we are considering here the same notion of very weak solution given in~\cite{Casteras-Monsaingeon}. Hence,~\eqref{e:kimura-example} comes with no additional boundary conditions and concentration effects in $\lambda_{1} = 0$ and~$\lambda_{1} = 1$ may happen, corresponding to extinction and fixation effects. The characterization of such effects has indeed been considered in~\cite{Casteras-Monsaingeon} as a consequence of mass conservation over the interval $[0, 1]$. 

\subsection{Spin systems}
\label{s:spin}
In this section, we present an example of a discrete-time process with a genuinely infinite set of strategies, borrowing ideas from the theory of spin systems in statistical mechanics.

We consider a Moran process in population of $N$ agents with a set of pure strategies given by $\Vv = \mathbb{S}^1$, where $\mathbb{S}^1$ is the unit circle in $\mathbb{R}^2$. The set of strategies $\Vv$ is endowed with the geodesic distance $\di_{\mathbb{S}^1}$, which makes $(\mathbb{S}^1, \di_{\mathbb{S}^1})$ a compact metric space.


In the discrete Moran process, at time $t^{k}_{m}$, each agent $n \in \{1, \ldots, N\}$ adopts a (random) strategy $\sigma^{N}_{n}(t^{k}_{m}) \in \mathbb{S}^1$. The distribution of strategies within the finite population is represented by the empirical measure $\lambda^N_{t^{k}_{m}} \in \Pp(\mathbb{S}^1)$ defined by
\begin{equation*}
    \lambda^N_{t^{k}_{m}} = \frac{1}{N} \sum_{n=1}^N \delta_{\sigma^{N}_{n}(t^{k}_{m})} \in \Pp(\mathbb{S}^1) \, . 
\end{equation*}
As described in the general setting, the probability of replication of a strategy $\sigma \in \mathbb{S}^1$ in a distribution of strategies $\lambda$ is mediated by the fitness $F_\lambda(\sigma)$ defined in terms of the payoff function as in~\eqref{e:Flambda}. We consider a payoff function given by, for every $\lambda \in \Pp(\mathbb{S}^1)$ and $\sigma \in \mathbb{S}^1$, 
\begin{equation} \label{eq:payoff-spins}
    \pi_{\lambda}(\sigma) = \int_{\mathbb{S}^1} \sigma \cdot \overline{\sigma} \, \di \lambda(\overline{\sigma}) = \sigma \cdot m_\lambda\,, 
\end{equation}
where 
\begin{equation*}
    m_\lambda = \int_{\mathbb{S}^1} \overline{\sigma} \, \di \lambda{(\overline{\sigma})} \in \R^2
\end{equation*}
is the magnetization.
This payoff function favors the replication of strategies that are aligned with the magnetization, mimicking the behavior of spin systems in statistical mechanics, like the classical planar XY model, where the local interaction energy between spins is proportional to their alignment.
Here, however, the lattice where agents sit is fully connected, \emph{i.e.}, each spin interacts with all other spins.
The choice of the specific payoff in~\eqref{eq:payoff-spins} is made for illustrative purposes. Other choices relevant in statistical mechanics models can be made (\emph{e.g.}, antiferromagnetic interactions). 

A comment about the mathematical relevance of the limit as $N \to +\infty$ is in order. 

\begin{remark}
    Let us consider the case where $N$ is fixed. It is very natural to assume that the initial distribution of strategies $\lambda^N_{0}$ is concentrated on a finite number of deterministic strategies $U = \{\sigma_1, \ldots, \sigma_M\} \subset \mathbb{S}^1$ (\emph{e.g.}, when one knows that each agent has adopted one strategy). Since the Moran process does not allow for the nucleation of new strategies, the evolution is confined to the finite set of strategies $U$. This means that, for a fixed $N$, the dynamics of the Moran process can be described by a finite-dimensional system, falling into the structure of the discrete process in Subsection~\ref{s:pure}, even though the underlying strategy space $\mathbb{S}^1$ is infinite. Nonetheless, in the limit as $N \to +\infty$, the number of strategies can explode, and the dynamics can explore the entire infinite strategy set $\mathbb{S}^1$. This highlights the importance of understanding the limit as $N \to +\infty$, as it allows for a richer and more complex behavior that cannot be captured by a finite-dimensional system. 
\end{remark}

Let us now discuss the limit model, described in terms of the limit of the curves $t \mapsto \Lambda^{N}_t \in \Pp(\Pp(\mathbb{S}^1))$.

\begin{remark} \label{rmk:pi-lipschitz}
First of all, notice that $(\pi.1)$ holds, \emph{i.e.}, the function $\pi \colon \mathbb{S}^1 \times \Pp(\mathbb{S}^1) \to \mathbb{R}$ defined by $\pi(\sigma, \lambda):=\pi_\lambda(\sigma)$ is Lipschitz continuous, where $\Pp(\mathbb{S}^1)$ is endowed with the metric induced by the BL-norm~$\| \cdot\|_{\rm BL}$. 
Indeed, if $\sigma, \sigma' \in \mathbb{S}^1$ and $\lambda, \lambda' \in \Pp(\mathbb{S}^1)$, then
\begin{align*} 
    | \pi_{\lambda}(\sigma) - \pi_{\lambda'}(\sigma')| & \leq | \pi_{\lambda}(\sigma) - \pi_{\lambda}(\sigma') | + | \pi_{\lambda}(\sigma') - \pi_{\lambda'}(\sigma') | \\ 
                                                       & = | \sigma \cdot m_{\lambda} - \sigma' \cdot m_{\lambda}| + | \sigma' \cdot m_{\lambda} - \sigma' \cdot m_{\lambda'} | \\ 
                                                       & \leq | \sigma - \sigma' | | m_{\lambda} | + | \sigma' | |m_{\lambda} - m_{\lambda'}| \\ 
                                                       & \leq \di_{\mathbb{S}^1}(\sigma,\sigma') + 2\| \lambda - \lambda' \|_{\mathrm{BL}}\,.
\end{align*}
In addition, considering the average payoff
\begin{equation*}
    \overline{\pi}_{\lambda} = \int_{\mathbb{S}^1} \pi_\lambda(\sigma) \, \di \lambda(\sigma) = \int_{\mathbb{S}^1} \sigma \cdot m_\lambda \, \di \lambda(\sigma) = |m_\lambda|^2 ,
\end{equation*}
we have that $\lambda \mapsto |m_\lambda|^2$ is Lipschitz continuous, since 
\begin{equation*}
    \big| |m_\lambda|^2 - |m_{\lambda'}|^2 \big| \leq |m_\lambda - m_{\lambda'}||m_{\lambda} + m_{\lambda'}| = \Big| \int_{\mathbb{S}^1} \sigma \, \di (\lambda - \lambda') \Big| \Big( |m_{\lambda}| + |m_{\lambda'}| \Big) \leq 4 \| \lambda - \lambda' \|_{\mathrm{BL}} \, . 
\end{equation*}
\end{remark}

We are in a position to apply Theorem~\ref{t:compactness-3} to study the limit under the rescaling of the parameters $N = N_k$, $\tau = \tau_k$ (the discrete time step), and $w = w_k$ prescribed by Theorem~\ref{t:compactness-3}. Thus, there exists a Lipschitz continuous limit curve $t \mapsto \Lambda_t \in \Pp(\Pp(\mathbb{S}^1))$ which is a weak solution to the Eulerian replicator equation~\eqref{e:weak-replicator}, \emph{i.e.}, for every $\psi \in C^1([0,T] \times \Pp(\mathbb{S}^1))$ it holds that 
\begin{align}
    \int_{\Pp(\mathbb{S}^1)} \psi (T, \lambda) \, \di \Lambda_{T} (\lambda) - \int_{\Pp(\mathbb{S}^1)} \psi(0, \lambda) \, \di \Lambda_{0} (\lambda) = & \int_{0}^{T} \int_{\Pp(\mathbb{S}^1)} \partial_{t} \psi(t, \lambda) \, \di \Lambda_{t} (\lambda) \, \di t \label{eq:eulerian-replicator-spin}
    \\
    &
    \nonumber + \gamma\int_{0}^{T} \int_{\Pp(\mathbb{S}^1)} D_{\lambda} \psi(t, \lambda) [(\pi_{\lambda} - \overline{\pi}_{\lambda}) \lambda]\, \di \Lambda_{t} (\lambda) \, \di t\,.
\end{align}
As discussed in Section~\ref{sub:comments}, the weak formulation~\eqref{eq:eulerian-replicator-spin} is the Eulerian specification of an ODE on the metric space $\Pp(\mathbb{S}^1)$, endowed with the distance $\di_{\mathrm{BL}}$ induced by the BL norm, \emph{i.e.}, regarding $\Pp(\mathbb{S}^1)$ as a convex compact subset of the separable Banach space $\mathcal{F}(\mathbb{S}^1)$ introduced in Section~\ref{S:pre}. 

In this example, 
the replicator vector field is the measure $b(\lambda) \in E_{\Pp(\mathbb{S}^1)}$ defined by  
\begin{equation*}
    \int_{\mathbb{S}^1} \varphi(\sigma) \, \di b(\lambda)(\sigma) = \int_{\mathbb{S}^1} (\pi_{\lambda}(\sigma) - |m_{\lambda}|^2) \varphi(\sigma) \, \di \lambda(\sigma) \, , \quad \text{for every } \varphi \in \Lip(\mathbb{S}^1) \, .
\end{equation*}

\begin{remark} \label{rmk:b-lipschitz}
    The map $\lambda \mapsto b(\lambda)$ is Lipschitz continuous in $E_{\Pp(\mathbb{S}^1)}$, since, for every $\varphi \in \Lip(\mathbb{S}^1)$ with $\|\varphi\|_{\Lip} \leq 1$, using Remark~\ref{rmk:pi-lipschitz}, for some positive constant $C$,
\begin{align*}
    \int_{\mathbb{S}^1} \varphi(\sigma) \, \di (b(\lambda) - b(\lambda'))(\sigma) & = \int_{\mathbb{S}^1}  \varphi(\sigma) (\pi_{\lambda}(\sigma) - |m_{\lambda}|^2) \, \di \lambda(\sigma) - \int_{\mathbb{S}^1} \varphi(\sigma) (\pi_{\lambda'}(\sigma) - |m_{\lambda'}|^2) \, \di \lambda'(\sigma) \\ 
                                                                                  & = \int_{\mathbb{S}^1} \varphi(\sigma) (\pi_{\lambda}(\sigma) - \pi_{\lambda'}(\sigma)) \, \di \lambda(\sigma) + \int_{\mathbb{S}^1} \varphi(\sigma) \pi_{\lambda'}(\sigma) \, \di (\lambda - \lambda')(\sigma) \\ 
                                                                                  & \quad + \int_{\mathbb{S}^1} \varphi(\sigma) \big(|m_{\lambda}|^2 - |m_{\lambda'}|^2\big) \, \di \lambda(\sigma) + \int_{\mathbb{S}^1} \varphi(\sigma) |m_{\lambda'}|^2 \, \di (\lambda - \lambda')(\sigma) \\ 
                                                                                  & \leq C \| \lambda - \lambda' \|_{\mathrm{BL}} \, .
\end{align*}
\end{remark}

By Remark~\ref{rmk:b-lipschitz} and~\cite{Brezis}, 
the flow map~$\Psi (t, \cdot)$ related to the ODE in $E_{\Pp(\mathbb{S}^1)}$
\begin{equation} \label{eq:replicator-spin}
       \dot \lambda_t = \gamma \, b(\lambda_t) \, . 
\end{equation}
is well-defined. The ODE~\eqref{eq:replicator-spin} is precisely the replicator equation for distribution of strategies $\lambda$ in the case of a continuum of strategies. As a consequence, there exists a unique solution $\Lambda$ to~\eqref{eq:eulerian-replicator-spin} given by~$\Lambda_{t}  = (\Psi (t, \cdot) )_{\#} \Lambda_{0}$ for every initial condition $\Lambda_0 \in \Pp(\Pp(\mathbb{S}^1))$, which coincides with the solution given by Theorem \ref{t:compactness-3}.

\section{Approximate equation}
\label{s:approximate equation}

The goal of this section is to prove the approximate equation contained in Theorem \ref{t:almost-equation} on a single interval of the form $(t^{k}_{m-1}, t^{k}_{m})$ for $m \in \{1, \ldots, k\}$. As a direct consequence, we will get in Corollary~\ref{c:almost-equation} an approximate equation on the whole interval $[0, T]$ and a Lipschitz-type estimate in Corollary~\ref{c:almost-equation-2}. The estimates of Theorem~\ref{t:almost-equation} also highlight the scalings of the parameters $N$, $\tau_{k}$ and $w$ involved in the time-discrete process. We use from now on the symbol $O(\vartheta)$ to denote a generic function of~$\vartheta \in [0, +\infty)$ continuous in~$0$ and satisfying
\begin{align*}
0 \leq O(\vartheta) \leq \alpha \vartheta \qquad \text{for $\vartheta \in [0, +\infty)$,}
\end{align*}
for some $\alpha \in (0, +\infty)$. 

\begin{theorem}
    \label{t:almost-equation}
    Assume that $(\pi.1)$ holds. For every $k \in \mathbb{N} \setminus \{0\}$, every $m \in \{1, \ldots, k\}$, and every $\psi \in C^{2} ([0, T] \times \Pp(\Vv))$ we have that
\begin{itemize}
\item [(i)] if $w \in \left[0, \frac{1}{2 (C_{\pi} + 1)}\right]$, then
    \begin{align}
    \label{e:approx_eq}
        \int_{\Pp(\Vv)} & \psi (t^{k}_{m}, \lambda) \, \di \Lambda^{N}_{t^{k}_{m}} (\lambda) -     \int_{\Pp(\Vv)} \psi (t^{k}_{m-1}, \lambda) \, \di \Lambda^{N}_{t^{k}_{m-1}} (\lambda) = \int_{t^{k}_{m-1}}^{t^{k}_{m}} \int_{\Pp(\Vv)} \partial_{t} \psi(t, \lambda) \, \di \Lambda^{N}_{t} (\lambda) \, \di t 
        \\
        &
        +  \frac{w}{\tau_{k} N}  \int_{t^{k}_{m-1}}^{t^{k}_{m}} \int_{\Pp(\Vv)} D_{\lambda} \psi (t, \lambda) [(\pi_{\lambda} - \overline{\pi}_{\lambda}) \lambda] \, \di \overline{\Lambda}^{N}_{t} (\lambda) \, \di t  \nonumber
        \\
        &
        +\frac{1}{2 \tau_{k} N^{2}} \int_{t^{k}_{m-1}}^{t^{k}_{m}} \int_{\Pp(\Vv)} \int_{\Vv} \int_{\Vv} D^{2}_{\lambda} \psi (t, \lambda) [\delta_{\sigma} - \delta_{\overline{\sigma}} \ , \ \delta_{\sigma} - \delta_{\overline{\sigma}} ]  \, \di \lambda(\sigma) \, \di \lambda(\overline{\sigma}) \, \di \overline{\Lambda}^{N}_{t} (\lambda) \, \di t \nonumber
        \\
        &
        + \vphantom{\int_{t^{k}_{m-1}}^{t^{k}_{m}} }\rho_{m}(\psi, \tau_{k}, N, w)\,,\nonumber
    \end{align}
    where
    \begin{align}
    \label{e:rhok}
       | \rho_{m} (\psi, \tau_{k}, N, w)|  \leq & \ O \bigg(\frac{  w^{2}  }{ N} \bigg) \| D \psi\|_{L^{\infty} ([0, T]\times \Pp(\Vv))}  + O \bigg( \frac{  w }{ N^{2}} \bigg) \| D^{2} \psi\|_{L^{\infty} ([0, T]\times \Pp(\Vv))}
        \\
        &
         + O \bigg(\frac{1}{N^{2}} \bigg) \bigg( \mathfrak{w}_{D^{2}\psi} (\tau_{k}) +\mathfrak{w}_{D^{2}\psi} \bigg(\frac{1}{N} \bigg) \bigg)  \nonumber
    \end{align}
    and $\mathfrak{w}_{D^{2}\psi}$ denotes a modulus of continuity of~$D^{2}\psi$ in $[0, T]\times \Pp(\Vv)$;
\item [(ii)] if \eqref{strong-sel-ass} holds, then
 \begin{align}
    \label{e:approx_eq-strong}
        \int_{\Pp(\Vv)} & \psi (t^{k}_{m}, \lambda) \, \di \Lambda^{N}_{t^{k}_{m}} (\lambda) -     \int_{\Pp(\Vv)} \psi (t^{k}_{m-1}, \lambda) \, \di \Lambda^{N}_{t^{k}_{m-1}} (\lambda) = \int_{t^{k}_{m-1}}^{t^{k}_{m}} \int_{\Pp(\Vv)} \partial_{t} \psi(t, \lambda) \, \di \Lambda^{N}_{t} (\lambda) \, \di t 
        \\
        &
        +  \frac{1}{\tau_{k} N}  \int_{t^{k}_{m-1}}^{t^{k}_{m}} \int_{\Pp(\Vv)} D_{\lambda} \psi (t, \lambda) \left[\frac{w( \pi_{\lambda} - \overline{\pi}_{\lambda})\lambda}{1 + w( \overline{\pi}_{\lambda} - 1) } \right]  \, \di \overline{\Lambda}^{N}_{t} (\lambda) \, \di t  \nonumber
        + \vphantom{\int_{t^{k}_{m-1}}^{t^{k}_{m}} }\zeta_{m}(\psi, N)\,,\nonumber
    \end{align}
 where
    \begin{align}
    \label{e:rhok-strong}
       | \zeta_{m} (\psi, N)|  \leq &  O \bigg( \frac{ 1 }{ N^{2}} \bigg) \| D^{2} \psi\|_{L^{\infty} ([0, T]\times \Pp(\Vv))}\,.
    \end{align}
\end{itemize}
\end{theorem}

We divide the proof of Theorem~\ref{t:almost-equation} in several lemmas.

\begin{lemma}
    \label{l:gradient-test}
    Assume that $(\pi.1)$ holds. For every $k \in \mathbb{N} \setminus \{0\}$, every $m \in \{1, \ldots, k\}$, and every $\psi \in C^{1} ([0, T] \times \Pp(\Vv))$ we have that
    \begin{itemize}
        \item[(i)] if $w \in \left[0, \frac{1}{2(C_\pi + 1)} \right]$, then
   \begin{align}
    \label{e:gradient-test}
         \frac{1}{\tau_{k}} & \int_{t^{k}_{m-1}}^{t^{k}_{m}} \mathbb{E} \bigg(  D_{\lambda} \psi(t, \lambda^{N}_{t^{k}_{m-1}}  ) \left[ \lambda^{N}_{t^{k}_{m}}   - \lambda^{N}_{t^{k}_{m-1}}    \right] \bigg)  \di t 
         \\
         &
         =  \frac{w}{\tau_{k} N}  \int_{t^{k}_{m-1}}^{t^{k}_{m}} \int_{\Pp(\Vv)} D_{\lambda} \psi (t, \lambda) \left[ ( \pi_{\lambda} - \overline{\pi}_{\lambda}) \lambda \right] \, \di \overline{\Lambda}^{N}_{t} (\lambda) \, \di t + R^{m}_{1}(\psi, \tau_{k}, N, w)\,, \nonumber
    \end{align}
    where 
    \begin{align}
    \label{e:rest-1}
        | R^{m}_{1} ( \psi, \tau_{k}, N, w) | \leq  O \bigg(\frac{  w^{2} }{N} \bigg)  \| D_{\lambda} \psi\|_{L^{\infty} ([0, T]\times \Pp(\Vv))}\,;
    \end{align}
\item[(ii)] if \eqref{strong-sel-ass} holds, then
 \begin{align}
    \label{e:gradient-test-strong-sel}
         \frac{1}{\tau_{k}} & \int_{t^{k}_{m-1}}^{t^{k}_{m}} \mathbb{E} \bigg(  D_{\lambda} \psi(t, \lambda^{N}_{t^{k}_{m-1}}  ) \left[ \lambda^{N}_{t^{k}_{m}}   - \lambda^{N}_{t^{k}_{m-1}}    \right] \bigg)  \di t 
         \\
         &
         =  \frac{1}{\tau_{k} N}  \int_{t^{k}_{m-1}}^{t^{k}_{m}} \int_{\Pp(\Vv)} D_{\lambda} \psi (t, \lambda) \left[\frac{w( \pi_{\lambda}  - \overline{\pi}_{\lambda}) \lambda}{1 + w( \overline{\pi}_{\lambda} - 1) } \right] \, \di \overline{\Lambda}^{N}_{t} (\lambda) \, \di t \,.\nonumber
    \end{align}
\end{itemize}
\end{lemma}

\begin{proof}
Let $\psi$, $N$, $k$, and $m$ be as in the statement of the lemma. In the following, given ${\bm{\sigma}} = (\sigma_1, \ldots, \sigma_N) \in \Vv^N$, we set $\lambda^{\bm{\sigma}} := \frac{1}{N} \sum_{n=1}^N \delta_{\sigma_n}$. By the law of total probability it holds that
\begin{align}
\label{e:103}
& \mathbb{E} \left(  D_{\lambda} \psi(t, \lambda^{N}_{t^{k}_{m-1}}  ) \left[ \lambda^{N}_{t^{k}_{m}}  - \lambda^{N}_{t^{k}_{m-1}}    \right] \right) \nonumber
\\ & \quad = \int_{\Vv^N} \mathbb{E} \left(  D_{\lambda} \psi(t, \lambda^{\bm\sigma}  ) \left[ \lambda^{N}_{t^{k}_{m}}  - \lambda^{\bm\sigma}  \right] \bigg | \,  {\bm{\sigma}}^{N}(t^{k}_{m-1}) = {\bm{\sigma}} \right)\, \di {\bm{\sigma}}^{N}(t^{k}_{m-1})_\# \mathbb{P} ({\bm{\sigma}})
\\
& \quad
= \int_{\Vv^N}  D_{\lambda} \psi(t, \lambda^{\bm{\sigma}}  ) \mathbb{E} \left( \left[ \lambda^{N}_{t^{k}_{m}}  - \lambda^{\bm{\sigma}}  \right] \bigg | \,  {\bm{\sigma}}^{N}(t^{k}_{m-1}) = {\bm{\sigma}} \right)\, \di {\bm{\sigma}}^{N}(t^{k}_{m-1})_\# \mathbb{P} ({\bm{\sigma}}) \,, \nonumber
\end{align}
where in the second equality we have used the linearity of the mean. By construction of $\lambda^{N}_{t^{k}_{m}}$ we have that
\begin{align}
\label{e:104}
   \mathbb{E}  \left( \left[ \lambda^{N}_{t^{k}_{m}}  - \lambda^{\bm{\sigma}}  \right] \bigg | \, {\bm{\sigma}}^{N}(t^{k}_{m-1}) = {\bm{\sigma}}  \right) &  = \frac{1}{N} \sum_{n=1}^{N}  \mathbb{E} \left( \left[ \delta_{\sigma^{N}_{n} (t^{k}_{m})} - \delta_{\sigma_{n} }  \right] \Big| \, {\bm{\sigma}}^{N}(t^{k}_{m-1}) = {\bm{\sigma}}  \right) 
   \\ \nonumber
   &
   = \frac{1}{N} \sum_{n, n'=1}^{N}   \left( \delta_{\sigma_{n'}} - \delta_{\sigma_{n}} \right) \mathbb{P} \left(\sigma^{N}_{n} (t^{k}_{m}) = \sigma_{n'} \Big| \, {\bm{\sigma}}^{N}(t^{k}_{m-1}) = {\bm{\sigma}} \right) \nonumber\, \\ \nonumber
   &  =  \frac{1}{N} \sum_{n, n'=1}^{N}  \left( \delta_{\sigma_{n'}} - \delta_{\sigma_{n}} \right) \frac{ \frac{1}{N^{2}} F_{\lambda^{\bm{\sigma}}} ( \sigma_{n'}) } { \overline{F}_{\lambda^{\bm{\sigma}}}}
        \\ \nonumber
        &
        =  \frac{1}{N \, \overline{F}_{\lambda^{\bm{\sigma}}}} \sum_{n, n' =1}^{N} \left( \delta_{\sigma_{n'}} - \delta_{\sigma_{n}} \right)  \frac{1}{N^{2}} F_{\lambda^{\bm{\sigma}}} (\sigma_{n'}) \nonumber
         \\ \nonumber
         &
       = \frac{1}{N \, \overline{F}_{\lambda^{\bm{\sigma}}}} \sum_{n'=1}^{N} \frac{1}{N}\big(  F_{\lambda^{\bm{\sigma}}} (\sigma_{n'}) - \overline{F}_{\lambda^{\bm{\sigma}}} \big) \delta_{\sigma_{n'}} \\ \nonumber
       & = \frac{1}{N \, \overline{F}_{\lambda^{\bm{\sigma}}}} \big( F_{\lambda^{\bm{\sigma}}} (\cdot) - \overline{F}_{\lambda^{\bm{\sigma}}}) \lambda^{\bm{\sigma}} \nonumber\,.
\end{align}
 Using the expansion~\eqref{e:Flambda} for the fitness $F_{\lambda^{\bm{\sigma}}}$, we get that $F_{\lambda^{\bm{\sigma}}} (\cdot) - \overline{F}_{\lambda^{\bm{\sigma}}} = w(\pi_{\lambda^{\bm{\sigma}}} (\cdot) - \overline{\pi}_{\lambda^{\bm{\sigma}}})$. Then,~\eqref{e:104} becomes
   \begin{align}
   \label{e:105}
   & \mathbb{E}  \left( \left[ \lambda^{N}_{t^{k}_{m}}  - \lambda^{\bm{\sigma}}  \right] \bigg | \, {\bm{\sigma}}^{N}(t^{k}_{m-1}) = {\bm{\sigma}}  \right) = \frac{w}{N ( 1 + w( \overline{\pi}_{\lambda^{\bm{\sigma}}} - 1) ) } \big( ( \pi_{\lambda^{\bm{\sigma}}} (\cdot) - \overline{\pi}_{\lambda^{\bm{\sigma}}}) \lambda^{\bm{\sigma}} \big)
   \end{align}
In particular, we notice that $( \pi_{\lambda^{\bm{\sigma}}} (\cdot) - \overline{\pi}_{\lambda^{\bm{\sigma}}}) \lambda^{\bm{\sigma}} \in E_{\Pp(\Vv)}$ since it is a measure with $0$-mean. 

Combining ~\eqref{e:103} and~\eqref{e:105} and using that $\lambda^{{\bm{\sigma}}^{N}(t^{k}_{m-1})} = \lambda^{N}_{t^{k}_{m-1}}$, we infer that
\begin{align} \label{eq:2609091859}
    & \mathbb{E} \left(  D_{\lambda} \psi(t, \lambda^{N}_{t^{k}_{m-1}}  ) \left[ \lambda^{N}_{t^{k}_{m}}  - \lambda^{N}_{t^{k}_{m-1}}    \right] \right) \\ \nonumber
    & \quad = \int_{\Vv^N}  D_{\lambda} \psi(t, \lambda^{\bm{\sigma}}  ) \frac{w}{N ( 1 + w( \overline{\pi}_{\lambda^{\bm{\sigma}}} - 1) ) } \big( ( \pi_{\lambda^{\bm{\sigma}}} (\cdot) - \overline{\pi}_{\lambda^{\bm{\sigma}}}) \lambda^{\bm{\sigma}} \big) \, \di {\bm{\sigma}}^{N}(t^{k}_{m-1})_\# \mathbb{P} ({\bm{\sigma}}) \\ \nonumber
    & \quad = \int_{\Pp(\Vv)}  D_{\lambda} \psi(t, \lambda) \frac{w}{N ( 1 + w( \overline{\pi}_{\lambda} - 1) ) } \big( ( \pi_{\lambda} (\cdot) - \overline{\pi}_{\lambda}) \lambda \big) \, \di (\lambda^{N}_{t^{k}_{m-1}})_\# \mathbb{P} (\lambda) \\ \nonumber
    & \quad = \int_{\Pp(\Vv)}  D_{\lambda} \psi(t, \lambda) \frac{w}{N ( 1 + w( \overline{\pi}_{\lambda} - 1) ) } \big( ( \pi_{\lambda} (\cdot) - \overline{\pi}_{\lambda}) \lambda \big) \, \di \Lambda^N_{t^{k}_{m-1}} (\lambda)\,.
\end{align}
Then, we observe that 
\begin{align*}
\frac{w}{N ( 1 + w( \overline{\pi}_{\lambda} - 1) ) } \big( ( \pi_{\lambda} (\cdot) - \overline{\pi}_{\lambda}) \lambda \big) = \frac{w}{N}  \big( ( \pi_{\lambda} (\cdot) - \overline{\pi}_{\lambda}) \lambda \big) - \frac{w^{2} ( \overline{\pi}_{\lambda} - 1)}{N( 1 + w( \overline{\pi}_{\lambda} - 1) )}  \big( ( \pi_{\lambda} (\cdot) - \overline{\pi}_{\lambda}) \lambda \big)\,. \nonumber
   \end{align*}
Now, if \eqref{strong-sel-ass} holds, by~\eqref{eq:2609091859} we obtain \eqref{e:gradient-test-strong-sel}. Otherwise, if $w\in \left[0, \frac{1}{2(C_\pi+1)}\right]$ it follows from \eqref{eq:2609091859} and \eqref{e:105} that
\begin{align}
\label{e:150}
\frac{1}{\tau_{k}} \int_{t^{k}_{m-1}}^{t^{k}_{m} } & \mathbb{E} \left(  D_{\lambda} \psi(t, \lambda^{N}_{t^{k}_{m-1}}  ) \left[ \lambda^{N}_{t^{k}_{m}}  - \lambda^{N}_{t^{k}_{m-1}}    \right] \right) \, \di t   =
   \\
   &
   = \frac{w}{\tau_{k} N} \int_{t^{k}_{m-1}}^{t^{k}_{m}} \int_{\Pp(\Vv)} D_{\lambda} \psi(t, \lambda) \left[ ( \pi_{\lambda} (\cdot) - \overline{\pi}_{\lambda}) \lambda \right] \, \di \overline{\Lambda}^{N}_{t} (\lambda) \,\di t \nonumber
   \\
   &
   \qquad - \frac{w^{2}}{\tau_{k} N}  \int_{t^{k}_{m-1}}^{t^{k}_{m}} \int_{\Pp(\Vv)} \frac{\overline{\pi}_{\lambda} - 1}{1 + w( \overline{\pi}_{\lambda} - 1) } D_{\lambda} \psi(t, \lambda) \left[  ( \pi_{\lambda} (\cdot) - \overline{\pi}_{\lambda}) \lambda   \right]\, \di \overline{\Lambda}^{N}_{t} (\lambda) \,\di t \,. \nonumber
\end{align}
  We set
  \begin{align*}
       R^{m}_{1} ( \psi, \tau_{k}, N, w) := - \frac{w^{2}}{\tau_{k} N}  \int_{t^{k}_{m-1}}^{t^{k}_{m}} \int_{\Pp(\Vv)} \frac{\overline{\pi}_{\lambda} - 1}{1 + w( \overline{\pi}_{\lambda} - 1)} D_{\lambda} \psi(t, \lambda) \left[  ( \pi_{\lambda} (\cdot) - \overline{\pi}_{\lambda}) \lambda   \right]\, \di \overline{\Lambda}^{N}_{t} (\lambda) \,\di t\,.
  \end{align*}
  By assumption~$(\pi.1)$, by~\eqref{e:pi-bdd} and since $w\in \left[0,\frac{1}{2(C_\pi+1)}\right]$, we have that for every $\lambda \in \Pp(\Vv)$
\begin{align*}
    \left| \frac{\overline{\pi}_{\lambda} - 1}{1 + w( \overline{\pi}_{\lambda} - 1) }  D_{\lambda} \psi(t, \lambda)  \big[ ( \pi_{\lambda} (\cdot) - \overline{\pi}_{\lambda}) \lambda \big] \right| & \leq \| D_{\lambda}\psi\|_{L^{\infty} ([0, T] \times \Pp(\Vv))} \left| \frac{\overline{\pi}_{\lambda} - 1}{1 + w( \overline{\pi}_{\lambda} - 1) }\right| \, \| (\pi_{\lambda} - \overline{\pi}_{\lambda}) \lambda\|_{\rm BL} 
    \\
    &
    \leq 2C_{\pi} ( 1 + C_{\pi}) \| D_{\lambda}\psi\|_{L^{\infty} ([0, T] \times \Pp(\Vv))}\,.
\end{align*}
Hence~\eqref{e:gradient-test} and~\eqref{e:rest-1} hold. This concludes the proof of the lemma.
\end{proof}

In what follows, for $\varphi \in C([0, T]\times \Pp(\Vv))$ we denote by $\mathfrak{w}_{\varphi}$ the modulus of (uniform) continuity of~$\varphi$ in $[0, T] \times \Pp(\Vv)$. 

\begin{lemma}
    \label{l:second-gradient-test}
     Assume that $(\pi.1)$ holds. For every $k \in \mathbb{N} \setminus \{0\}$, every $m \in \{1, \ldots, k\}$, and every $\psi \in C^{2} ([0, T] \times \Pp(\Vv))$ we have that
     \begin{itemize}
         \item[(i)] if $w\in\left[0,\frac{1}{2(C_\pi+1)}\right]$, then
    \begin{align}
     \label{e:second-gradient-test}
        & \frac{1}{\tau_{k}} \int_{t^{k}_{m-1}}^{t^{k}_{m}} \int_{0}^{1} \mathbb{E} \bigg( D^{2}_{\lambda} \psi \left( t, \lambda^{N}_{t^{k}_{m-1}} + r (\lambda^{N}_{t} - \lambda^{N}_{t^{k}_{m-1}}) \right)  \left[ \lambda^{N}_{t}   - \lambda^{N}_{t^{k}_{m-1}}  , \lambda^{N}_{t^{k}_{m}}   - \lambda^{N}_{t^{k}_{m-1}}    \right]\bigg) \, \di r \, \di t 
         \\
         &
         \, = \frac{1}{2 \tau_{k} N^{2}} \int_{t^{k}_{m-1}}^{t^{k}_{m}} \int_{\Pp(\Vv)} \int_{\Vv} \int_{\Vv} D^{2}_{\lambda} \psi (t, \lambda) [\delta_{\sigma} - \delta_{\overline{\sigma}} \,, \delta_{\sigma} - \delta_{\overline{\sigma}} ]  \, \di \lambda(\sigma) \, \di \lambda(\overline{\sigma}) \, \di \overline{\Lambda}^{N}_{t} (\lambda) \, \di t + R^{m}_{2} (\psi, \tau_{k}, N, w), \nonumber
     \end{align}
    where
    \begin{align}
        \label{e:rest-2}
        | R^{m}_{2} (\psi, \tau_{k}, N, w) | \leq & \ O\bigg(\frac{w}{N^{2}} \bigg) \| D^{2}\psi\|_{L^{\infty} ([0, T] \times \Pp(\Vv))} 
        \\
        &
        \nonumber + O \bigg(\frac{1}{N^{2}} \bigg) \bigg( \mathfrak{w}_{D^{2}\psi} (\tau_{k}) +\mathfrak{w}_{D^{2}\psi} \bigg(\frac{1}{N} \bigg) \bigg) \,;
    \end{align}
\item[(ii)] if \eqref{strong-sel-ass} holds, then
\begin{align}
     \label{e:second-gradient-test-strong}
        & \bigg|\frac{1}{\tau_{k}} \int_{t^{k}_{m-1}}^{t^{k}_{m}} \int_{0}^{1} \mathbb{E} \bigg( D^{2}_{\lambda} \psi \left( t, \lambda^{N}_{t^{k}_{m-1}} + r (\lambda^{N}_{t} - \lambda^{N}_{t^{k}_{m-1}}) \right)  \left[ \lambda^{N}_{t}   - \lambda^{N}_{t^{k}_{m-1}}  , \lambda^{N}_{t^{k}_{m}}   - \lambda^{N}_{t^{k}_{m-1}}    \right]\bigg) \, \di r \, \di t \bigg| 
        \\
        &
         \, \leq \ O\bigg(\frac{1}{N^{2}} \bigg) \| D^{2}\psi\|_{L^{\infty} ([0, T] \times \Pp(\Vv))}. \nonumber
\end{align}
\end{itemize}
\end{lemma}

\begin{proof}
     Let $\psi$, $N$, $k$, and $m$ be as in the statement of the lemma. By construction of $\lambda^{N}_{t}$, for $t \in (t^{k}_{m-1}, t^{k}_{m})$ and $r \in [0, 1]$ we have that
  \begin{align}
        \label{e:107}
       \mathbb{E}&  \bigg( D^{2}_{\lambda} \psi \left( t, \lambda^{N}_{t^{k}_{m-1}} + r (\lambda^{N}_{t} - \lambda^{N}_{t^{k}_{m-1}}) \right)  \left[ \lambda^{N}_{t}   - \lambda^{N}_{t^{k}_{m-1}}  , \lambda^{N}_{t^{k}_{m}}   - \lambda^{N}_{t^{k}_{m-1}}    \right]\bigg)
       \\
       &
       = \frac{t - t^{k}_{m-1}}{\tau_{k}} \, \mathbb{E} \bigg( D^{2}_{\lambda} \psi \left(t, \lambda^{N}_{t^{k}_{m-1}} + r (\lambda^{N}_{t} - \lambda^{N}_{t^{k}_{m-1}} ) \right)  \left[ \lambda^{N}_{t^{k}_{m}} - \lambda^{N}_{t^{k}_{m-1}} ,  \lambda^{N}_{t^{k}_{m}} - \lambda^{N}_{t^{k}_{m-1}} \right] \bigg)\nonumber
       \\
       &
        =  \frac{t - t^{k}_{m-1}}{\tau_{k}} \,  \mathbb{E} \bigg( D^{2}_{\lambda} \psi \big(t, \lambda^{N}_{t^{k}_{m-1}}\big)   \left[ \lambda^{N}_{t^{k}_{m}} - \lambda^{N}_{t^{k}_{m-1}} ,  \lambda^{N}_{t^{k}_{m}} - \lambda^{N}_{t^{k}_{m-1}} \right] \bigg)  \nonumber
        \\
        &
        \quad + \frac{t - t^{k}_{m-1}}{\tau_{k}} \, \mathbb{E} \bigg( \left( D^{2}_{\lambda} \psi \left( t, \lambda^{N}_{t^{k}_{m-1}} + r (\lambda^{N}_{t} - \lambda^{N}_{t^{k}_{m-1}}) \right) -  D^{2}_{\lambda} \psi \big(t, \lambda^{N}_{t^{k}_{m-1}} \big) \right)   \left[ \lambda^{N}_{t^{k}_{m}} - \lambda^{N}_{t^{k}_{m-1}} ,  \lambda^{N}_{t^{k}_{m}} - \lambda^{N}_{t^{k}_{m-1}} \right] \bigg) \,.\nonumber
   \end{align}  
   We now estimate the two terms on the right-hand side of~\eqref{e:107}. As in the proof of Lemma~\ref{l:gradient-test}, given ${\bm{\sigma}} = (\sigma_1, \ldots, \sigma_n) \in \Vv^N$, we set $\lambda^{\bm{\sigma}} := \frac{1}{N} \sum_{i=1}^N \delta_{\sigma_n}$. By the law of total probability we rewrite the first term on the right-hand side of~\eqref{e:107} as
   \begin{align}
   \label{e:151}
        &\frac{t - t^{k}_{m-1}}{\tau_{k}} \,  \mathbb{E} \bigg( D^{2}_{\lambda} \psi \big(t, \lambda^{N}_{t^{k}_{m-1}}\big)   \left[ \lambda^{N}_{t^{k}_{m}} - \lambda^{N}_{t^{k}_{m-1}} ,  \lambda^{N}_{t^{k}_{m}} - \lambda^{N}_{t^{k}_{m-1}} \right] \bigg) 
        \\
        &
        \nonumber \quad =  \frac{t - t^{k}_{m-1}}{\tau_{k}} \int_{\Vv^N} \mathbb{E}  \bigg( D^{2}_{\lambda} \psi( t, \lambda^{\bm{\sigma}})  \left[ \lambda^{N}_{t^{k}_{m}}   - \lambda^{\bm{\sigma}} \  , \ \lambda^{N}_{t^{k}_{m}}   - \lambda^{\bm{\sigma}}    \right] \ \bigg|\ {\bm{\sigma}}^{N}(t^{k}_{m-1}) = {\bm{\sigma}} \bigg) \, \di {\bm{\sigma}}^{N}(t^{k}_{m-1})_\# \mathbb{P} ({\bm{\sigma}}) \,. \nonumber
   \end{align}
   We compute 
   \begin{align*}
       & \mathbb{E}  \bigg( D^{2}_{\lambda} \psi( t, \lambda^{\bm{\sigma}})  \left[ \lambda^{N}_{t^{k}_{m}}   - \lambda^{\bm{\sigma}} \  , \ \lambda^{N}_{t^{k}_{m}}   - \lambda^{\bm{\sigma}}    \right] \ \bigg|\ {\bm{\sigma}}^{N}(t^{k}_{m-1}) = {\bm{\sigma}} \bigg) \\ 
       & \quad = \frac{1}{N^2} \sum_{n=1}^{N}\sum_{\ell=1}^{N}  \mathbb{E} \left(  D^{2}_{\lambda} \psi (t, \lambda^{\bm{\sigma}})  \left[ \delta_{\sigma_{n}^{N} (t^{k}_{m})}  - \delta_{\sigma_{n}} \ ,  \ \delta_{\sigma_{\ell}^{N} (t^{k}_{m})}  - \delta_{\sigma_{\ell}} \right] \ \Big| \ {\bm{\sigma}}^{N}(t^{k}_{m-1}) = {\bm{\sigma}} \right) 
    \end{align*}
    and we observe that, since only one agent may replace its strategy, if the event $\sigma^{N}_n(t^{k}_{m}) = \sigma_{n'}$ occurs, then $\delta_{\sigma_{\ell}^{N} (t^{k}_{m})}  - \delta_{\sigma_{\ell}} \neq 0$ only if $\ell = n$. From this and from the fact that $\lambda^{\bm{\sigma}} = \frac{1}{N} \sum_{n=1}^N \delta_{\sigma_n}$, it follows that
    \begin{align*}
       & \mathbb{E}  \bigg( D^{2}_{\lambda} \psi( t, \lambda^{\bm{\sigma}})  \left[ \lambda^{N}_{t^{k}_{m}}   - \lambda^{\bm{\sigma}} \  , \ \lambda^{N}_{t^{k}_{m}}   - \lambda^{\bm{\sigma}}    \right] \ \bigg|\ {\bm{\sigma}}^{N}(t^{k}_{m-1}) = {\bm{\sigma}} \bigg) \\ 
       & \quad = \frac{1}{N^2} \sum_{n=1}^{N} \mathbb{E} \left(  D^{2}_{\lambda} \psi (t, \lambda^{\bm{\sigma}})  \left[ \delta_{\sigma_{n}^{N} (t^{k}_{m})}  - \delta_{\sigma_{n}} \ ,  \ \delta_{\sigma_{n}^{N} (t^{k}_{m})}  - \delta_{\sigma_{n}} \right] \ \Big| \ {\bm{\sigma}}^{N}(t^{k}_{m-1}) = {\bm{\sigma}} \right) \\ 
       & \quad = \frac{1}{N^2} \sum_{n=1}^{N} \sum_{n'=1}^{N}   D^{2}_{\lambda} \psi (t, \lambda^{\bm{\sigma}})  \left[ \delta_{\sigma_{n'}} - \delta_{\sigma_{n}} \ ,  \ \delta_{\sigma_{n'}}  - \delta_{\sigma_{n}} \right] \mathbb{P} \Big( \sigma^{N}_n(t^{k}_{m}) = \sigma_{n'} \ \Big| \ {\bm{\sigma}}^{N}(t^{k}_{m-1}) = {\bm{\sigma}} \Big) \\ 
       & \quad = \frac{1}{N^2} \sum_{n=1}^{N} \sum_{n'=1}^{N}   D^{2}_{\lambda} \psi (t, \lambda^{\bm{\sigma}})  \left[ \delta_{\sigma_{n'}} - \delta_{\sigma_{n}} \ ,  \ \delta_{\sigma_{n'}}  - \delta_{\sigma_{n}} \right] \frac{\frac{1}{N^{2}} F_{\lambda^{\bm{\sigma}}} (\sigma_{n'}) }{ \overline{F}_{\lambda^{\bm{\sigma}}}} \\ 
       & \quad = \frac{1}{N^2} \int_{\Vv} \int_{\Vv}  D^{2}_{\lambda} \psi (t, \lambda^{\bm{\sigma}})  \left[ \delta_{\sigma'} - \delta_{\sigma''} \ ,  \ \delta_{\sigma'}  - \delta_{\sigma''} \right] \frac{F_{\lambda^{\bm{\sigma}}}(\sigma')}{\overline{F}_{\lambda^{\bm{\sigma}}}} \, \di \lambda^{\bm{\sigma}}(\sigma') \di \lambda^{\bm{\sigma}}(\sigma'') \,. 
    \end{align*}
    Inserting this into~\eqref{e:151}, using the fact that $\lambda^{{\bm{\sigma}}^{N}(t^{k}_{m-1})} = \lambda^{N}_{t^{k}_{m-1}}$, and using the explicit expression~\eqref{e:Flambda} of $F_{\lambda}$ and $\overline{F}_{\lambda}$ we get that 
    \begin{align} \label{eq:2609092111}
        & \frac{t - t^{k}_{m-1}}{\tau_{k}} \,  \mathbb{E} \bigg( D^{2}_{\lambda} \psi \big(t, \lambda^{N}_{t^{k}_{m-1}}\big)   \left[ \lambda^{N}_{t^{k}_{m}} - \lambda^{N}_{t^{k}_{m-1}} ,  \lambda^{N}_{t^{k}_{m}} - \lambda^{N}_{t^{k}_{m-1}} \right] \bigg)
         \\ \nonumber
        & = \frac{t - t^{k}_{m-1}}{\tau_{k} N^2} \int_{\Vv^N} \int_{\Vv} \int_{\Vv}  D^{2}_{\lambda} \psi (t, \lambda^{\bm{\sigma}})  \left[ \delta_{\sigma'} - \delta_{\sigma''} \ ,  \ \delta_{\sigma'}  - \delta_{\sigma''} \right] \frac{ F_{\lambda^{\bm{\sigma}}}(\sigma')}{\overline{F}_{\lambda^{\bm{\sigma}}}}  \, \di \lambda^{\bm{\sigma}}(\sigma') \di \lambda^{\bm{\sigma}}(\sigma'') \, \di {\bm{\sigma}}^{N}(t^{k}_{m-1})_\# \mathbb{P} ({\bm{\sigma}}) \\ \nonumber
        & = \frac{t - t^{k}_{m-1}}{\tau_{k} N^2} \int_{\Pp(\Vv)} \int_{\Vv} \int_{\Vv}  D^{2}_{\lambda} \psi (t, \lambda)  \left[ \delta_{\sigma'} - \delta_{\sigma''} \ ,  \ \delta_{\sigma'}  - \delta_{\sigma''} \right] \frac{ F_{\lambda}(\sigma')}{\overline{F}_{\lambda}}  \, \di \lambda(\sigma') \di \lambda(\sigma'') \, \di (\lambda^{N}_{t^{k}_{m-1}})_\# \mathbb{P} (\lambda) \\ \nonumber
        & = \frac{t - t^{k}_{m-1}}{\tau_{k} N^2} \int_{\Pp(\Vv)} \frac{1}{\overline{F}_{\lambda}} \int_{\Vv} \int_{\Vv}  D^{2}_{\lambda} \psi (t, \lambda)  \left[ \delta_{\sigma'} - \delta_{\sigma''} \ ,  \ \delta_{\sigma'}  - \delta_{\sigma''} \right] \frac{ F_{\lambda}(\sigma')}{\overline{F}_{\lambda}}  \, \di \lambda(\sigma') \di \lambda(\sigma'') \, \di \overline{\Lambda}^{N}_{t}(\lambda) \\ \nonumber
        & = \frac{t - t^{k}_{m-1}}{\tau_{k} N^2} \int_{\Pp(\Vv)}  \int_{\Vv} \int_{\Vv}  D^{2}_{\lambda} \psi (t, \lambda)  \left[ \delta_{\sigma'} - \delta_{\sigma''} \ ,  \ \delta_{\sigma'}  - \delta_{\sigma''} \right]  \frac{( 1 + w ( \pi_{\lambda} (\sigma') - 1))}{ ( 1 + w( \overline{\pi}_{\lambda} - 1) )} \, \di \lambda(\sigma') \di \lambda(\sigma'') \, \di \overline{\Lambda}^{N}_{t}(\lambda) \\ \nonumber
        & = \frac{(t - t^{k}_{m-1})}{\tau_{k} N^{2}} \int_{\Pp(\Vv)} \int_{\Vv} \int_{\Vv} D^{2}_{\lambda} \psi (t, \lambda) [\delta_{\sigma'} - \delta_{\sigma''}\ , \  \delta_{\sigma'} - \delta_{\sigma''} ]  \, \di \lambda(\sigma') \, \di \lambda(\sigma'') \, \di \overline{\Lambda}^{N}_{t} (\lambda)
         \\ \nonumber
        & \quad + \frac{w ( t - t^{k}_{m-1}) }{\tau_{k} N^{2}} \int_{\Pp(\Vv)} \int_{\Vv} \int_{\Vv} D^{2}_{\lambda} \psi (t, \lambda) [\delta_{\sigma'} - \delta_{\sigma''} \ , \ \delta_{\sigma'} - \delta_{\sigma''} ]  \frac{(\pi_{\lambda} (\sigma') - \overline{\pi}_{\lambda})}{(1+w(\overline{\pi}_{\lambda}-1))} \, \di \lambda(\sigma') \, \di \lambda(\sigma'') \, \di \overline{\Lambda}^{N}_{t} (\lambda)
\\ \nonumber
         &
         = : \frac{( t - t^{k}_{m-1})}{\tau_{k} N^{2}} \int_{\Pp(\Vv)} \int_{\Vv} \int_{\Vv} D^{2}_{\lambda} \psi (t, \lambda) [\delta_{\sigma'} - \delta_{\sigma''}\ , \  \delta_{\sigma'} - \delta_{\sigma''} ]  \, \di \lambda(\sigma') \, \di \lambda(\sigma'')\, \di \overline{\Lambda}^{N}_{t} (\lambda) + R^{m}_{2, 1} ( t,  \psi, \tau_{k}, N, w)\,.
    \end{align}
    Setting
    \begin{align*}
        R^{m}_{2, 1} &  (\psi, \tau_{k}, N, w)  :=  \frac{1}{\tau_{k}}\int_{t^{k}_{m-1}}^{t^{k}_{m}}   R^{m}_{2, 1} ( t,  \psi, \tau_{k}, N, w) \, \di t 
    \end{align*}
    and combining~\eqref{e:151} and \eqref{eq:2609092111}, we write
    \begin{align}
        \label{e:154}
        \int_{t^{k}_{m-1}}^{t^{k}_{m}} &  \frac{t - t^{k}_{m-1}}{\tau^{2}_{k}} \, \mathbb{E}  \left( D^{2}_{\lambda} \psi( t, \lambda^{N}_{t^{k}_{m-1}})  \left[ \lambda^{N}_{t^{k}_{m}}   - \lambda^{N}_{t^{k}_{m-1}} \ , \ \lambda^{N}_{t^{k}_{m}}   - \lambda^{N}_{t^{k}_{m-1}}    \right]\right) \, \di t 
        \\
        &
        = \frac{1}{N^{2}} \int_{t^{k}_{m-1}}^{t^{k}_{m}} \int_{\Pp(\Vv)}  \int_{\Vv} \int_{\Vv} \frac{t - t^{k}_{m-1}}{\tau^{2}_{k}} \, D^{2}_{\lambda} \psi (t, \lambda) [\delta_{\sigma'} - \delta_{\sigma''}\ , \  \delta_{\sigma'} - \delta_{\sigma''} ]  \, \di \lambda(\sigma') \, \di \lambda(\sigma'')\, \di \overline{\Lambda}^{N}_{t} (\lambda) \, \di t \nonumber
        \\
        &
        \qquad \vphantom{\int_{\Om}^{\Om}} + R^{m}_{2, 1}   (\psi, \tau_{k}, N, w) \,. \nonumber
    \end{align}
    We notice that by regularity of $\psi$ and by $(\pi.1)$ (see also~\eqref{e:pi-bdd}), if $w \in \left[0,\frac{1}{2 ( C_{\pi} + 1) }\right]$ it holds
    \begin{align*}
       | R^{m}_{2, 1} (\psi, \tau_{k}, N, w) | \leq \frac{w\,   \| D^{2}_{\lambda} \psi\|_{L^{\infty} ([0, T]\times \Pp(\Vv))}}{N^{2}} \,   4(C_{\pi} + 1) \,. 
    \end{align*}
    Therefore, we infer that
    \begin{align}
        \label{e:rest-2.1}
        | R^{m}_{2, 1} (\psi, \tau_{k}, N, w) | \leq O \bigg(\frac{w}{N^{2}} \bigg) \| D^{2}_{\lambda} \psi\|_{L^{\infty} ([0, T]\times \Pp(\Vv))}\,.
    \end{align}
With a similar argument, if follows from \eqref{e:154} that, if \eqref{strong-sel-ass} holds, then 
 \begin{align}
        \label{e:154-strong}
        \bigg|\int_{t^{k}_{m-1}}^{t^{k}_{m}} &  \frac{t - t^{k}_{m-1}}{\tau^{2}_{k}} \, \mathbb{E}  \left( D^{2}_{\lambda} \psi( t, \lambda^{N}_{t^{k}_{m-1}})  \left[ \lambda^{N}_{t^{k}_{m}}   - \lambda^{N}_{t^{k}_{m-1}} \ , \ \lambda^{N}_{t^{k}_{m}}   - \lambda^{N}_{t^{k}_{m-1}}    \right]\right) \, \di t  \bigg|
        \\
        &
        \leq O \bigg(\frac{1}{N^{2}} \bigg) \| D^{2}_{\lambda} \psi\|_{L^{\infty} ([0, T]\times \Pp(\Vv))}\,. \nonumber
\end{align}
We further estimate the first term on the right-hand side of~\eqref{e:154} in order to obtain, up to an additional remainder, the integral term in~\eqref{e:second-gradient-test}. Indeed, we have that
    \begin{align}
    \label{e:155}
        \frac{1}{N^{2}}&  \int_{t^{k}_{m-1}}^{t^{k}_{m}} \int_{\Pp(\Vv)}  \int_{\Vv} \int_{\Vv} \frac{t - t^{k}_{m-1}}{\tau^{2}_{k}} \, D^{2}_{\lambda} \psi (t, \lambda) [\delta_{\sigma'} - \delta_{\sigma''}\ , \  \delta_{\sigma'} - \delta_{\sigma''} ]  \, \di \lambda(\sigma') \, \di \lambda(\sigma'')\, \di \overline{\Lambda}^{N}_{t} (\lambda) \, \di t
        \\
        &
        = \frac{1}{N^{2}} \int_{t^{k}_{m-1}}^{t^{k}_{m}}  \int_{\Pp(\Vv)} \int_{\Vv} \int_{\Vv}\frac{t - t^{k}_{m-1}}{\tau^{2}_{k}} \, D^{2}_{\lambda} \psi (t^{k}_{m-1}, \lambda) [\delta_{\sigma'} - \delta_{\sigma''} \  ,   \ \delta_{\sigma'} - \delta_{\sigma''} ]  \, \di \lambda(\sigma') \, \di \lambda(\sigma'') \, \di \overline{\Lambda}^{N}_{t} (\lambda) \, \di t \nonumber 
        \\
        &
        \qquad + \vphantom{\int_{t^{k}_{m-1}}^{t^{k}_{m}}} R^{m}_{2,2} (\psi, \tau_{k}, N, w)  \nonumber
       \\
       &
       = \frac{1}{2 \tau_{k} N^{2}} \int_{t^{k}_{m-1}}^{t^{k}_{m}} \int_{\Pp(\Vv)} \int_{\Vv} \int_{\Vv} D^{2}_{\lambda} \psi (t, \lambda) [\delta_{\sigma'} - \delta_{\sigma''} \ , \ \delta_{\sigma'} - \delta_{\sigma''} ]  \, \di \lambda(\sigma') \, \di \lambda(\sigma'') \, \di \overline{\Lambda}^{N}_{t} (\lambda) \, \di t  \nonumber
       \\
       &
       \qquad \vphantom{\int_{t^{k}_{m-1}}^{t^{k}_{m}}} + R^{m}_{2,2} (\psi, \tau_{k}, N, w) + R^{m}_{2,3} (\psi, \tau_{k}, N, w)\,, \nonumber
    \end{align}
   where
   \begin{align*}
        R^{m}_{2,2} (\psi, \tau_{k}, N, w) & := \frac{1}{N^{2}} \int_{t^{k}_{m-1}}^{t^{k}_{m}} \int_{\Pp(\Vv)} \int_{\Vv} \int_{\Vv} \frac{t - t^{k}_{m-1}}{\tau^{2}_{k}} \Big[D^{2}_{\lambda} \psi (t, \lambda) [\delta_{\sigma'} - \delta_{\sigma''} \ , \ \delta_{\sigma'} - \delta_{\sigma''} ] 
        \\
        &
        \qquad\qquad\qquad \vphantom{\int_{\Pp(\Vv)}} - D^{2}_{\lambda} \psi (t^{k}_{m-1}, \lambda) [\delta_{\sigma'} - \delta_{\sigma''} \ , \ \delta_{\sigma'} - \delta_{\sigma''} ] \Big]   \, \di \lambda(\sigma') \, \di \lambda(\sigma'') \, \di \overline{\Lambda}^{N}_{t} (\lambda) \, \di t  \,, \nonumber
        \\
         R^{m}_{2,3} (\psi, \tau_{k}, N, w)  & :=  \frac{1}{2\tau_{k} N^{2}} \int_{t^{k}_{m-1}}^{t^{k}_{m}} \int_{\Pp(\Vv)} \int_{\Vv} \int_{\Vv} \Big[D^{2}_{\lambda} \psi (t^{k}_{m-1}, \lambda) [\delta_{\sigma'} - \delta_{\sigma''} \ , \  \delta_{\sigma'} - \delta_{\sigma''} ] 
         \\
         &
         \qquad\qquad\qquad \vphantom{\int_{\Pp(\Vv)}} - D^{2}_{\lambda} \psi (t, \lambda)  [\delta_{\sigma'} - \delta_{\sigma''} \ , \  \delta_{\sigma'} - \delta_{\sigma''} ]  \Big] \, \di \lambda(\sigma') \, \di \lambda(\sigma'') \, \di \overline{\Lambda}^{N}_{t} (\lambda) \, \di t\,.
   \end{align*}
   Let us denote by~$\mathfrak{w}_{D^{2} \psi}$ the modulus of continuity of~$D^{2}\psi$  in $[0, T] \times \Pp(\Vv)$. Without loss of generality, we assume $\mathfrak{w}_{D^{2} \psi}$ to be increasing and concave. Hence, we estimate
   \begin{align}
   \label{e:R34-1}
       | R^{m}_{2,2} (\psi, \tau_{k}, N, w) |  \leq   O\bigg( \frac{1 } {N^{2}} \bigg) \mathfrak{w}_{D^{2}\psi} (\tau_{k}) \,, 
       \\
       \label{e:R34-2}| R^{m}_{2,3}(\psi, \tau_{k}, N, w) | \leq  O\bigg( \frac{1 } {N^{2}} \bigg) \mathfrak{w}_{D^{2}\psi} (\tau_{k}) \,.
   \end{align}
   To conclude for~\eqref{e:second-gradient-test}, we are left to estimate the last term in~\eqref{e:107}. By the law of total probability, and by the definition of $\lambda_t^N$ given in \eqref{e:def-interpolata-2}, for every $r \in [0, 1]$ and every $t \in (t^{k}_{m-1}, t^{k}_{m})$ we may write
   \begin{align}
       \label{e:10000}
       & \bigg| \mathbb{E} \bigg( \left( D^{2}_{\lambda} \psi \left( t, \lambda^{N}_{t^{k}_{m-1}} + r (\lambda^{N}_{t} - \lambda^{N}_{t^{k}_{m-1}}) \right) -  D^{2}_{\lambda} \psi \big(t, \lambda^{N}_{t^{k}_{m-1}} \big) \right)   \left[ \lambda^{N}_{t^{k}_{m}} - \lambda^{N}_{t^{k}_{m-1}} ,  \lambda^{N}_{t^{k}_{m}} - \lambda^{N}_{t^{k}_{m-1}} \right] \bigg) \bigg|
       \\
       &
       \quad \leq \mathbb{E} \bigg( \mathfrak{w}_{D^{2} \psi} \left( r \left\| \lambda^{N}_{t} - \lambda^{N}_{t^{k}_{m-1}}\right\|_{\rm BL} \right) \left\| \lambda^{N}_{t^{k}_{m}} - \lambda^{N}_{t^{k}_{m-1}} \right\|_{\rm BL}^{2} \bigg) \nonumber
       \\
       &
       \quad \leq  \int_{\Vv^N}\mathbb{E} \bigg(  \mathfrak{w}_{D^{2} \psi} \left(  \big\| \lambda^{N}_{t^{k}_{m}} - \lambda \big\|_{\rm BL} \right) \big\| \lambda^{N}_{t^{k}_{m}} - \lambda^{\bm{\sigma}} \big\|_{\rm BL}^{2} \, \Big| \, {\bm{\sigma}}^{N}(t^{k}_{m-1}) = {\bm{\sigma}} \bigg) \, \di {\bm{\sigma}}^{N}(t^{k}_{m-1})_\# \mathbb{P} ({\bm{\sigma}}) \nonumber\,.
   \end{align}
   By construction of $\lambda^{N}_{t^{k}_{m}}$ and by the concavity of~$\mathfrak{w}_{D^{2} \psi}$ we continue in~\eqref{e:10000} with
      \begin{align}
       \label{e:10001}
       & \bigg| \mathbb{E} \bigg( \left( D^{2}_{\lambda} \psi \left( t, \lambda^{N}_{t^{k}_{m-1}} + r (\lambda^{N}_{t} - \lambda^{N}_{t^{k}_{m-1}}) \right) -  D^{2}_{\lambda} \psi \big(t, \lambda^{N}_{t^{k}_{m-1}} \big) \right)   \left[ \lambda^{N}_{t^{k}_{m}} - \lambda^{N}_{t^{k}_{m-1}} ,  \lambda^{N}_{t^{k}_{m}} - \lambda^{N}_{t^{k}_{m-1}} \right] \bigg)\bigg|
       \\
       &
       \leq  \int_{\Vv^N}\mathbb{E} \bigg(  \mathfrak{w}_{D^{2} \psi} \left( \frac{1}{N} \sum_{n=1}^{N} \big\| \delta_{\sigma^{N}_{n} ( t^{k}_{m})} - \delta_{\sigma_{n}} \big\|_{\rm BL} \right) \times \nonumber \\ 
       & \hspace{2em} \times \left\| \frac{1}{N} \sum_{n=1}^{N}  \delta_{\sigma^{N}_{n} ( t^{k}_{m})} - \delta_{\sigma_{n}} \right\|_{\rm BL}^{2} \, \Big| \, {\bm{\sigma}}^{N}(t^{k}_{m-1}) = {\bm{\sigma}} \bigg) \, \di {\bm{\sigma}}^{N}(t^{k}_{m-1})_\# \mathbb{P} ({\bm{\sigma}}) \nonumber
       \\
       &
       = \frac{1}{N^{2} }  \int_{\Vv^N} \sum_{n, n'=1}^{N} \mathfrak{w}_{D^{2} \psi} \left( \frac{1}{N} \| \delta_{\sigma_{n'}} - \delta_{\sigma_{n}} \|_{\rm BL} \right) \| \delta_{\sigma_{n'}} - \delta_{\sigma_{n}} \|_{\rm BL}^{2} \times \nonumber \\ 
      & \hspace{2em} \times \mathbb{P} \big( \sigma^{N}_{n} (t^{k}_{m}) = \sigma_{n'} \big| \, {\bm{\sigma}}^{N}(t^{k}_{m-1}) = {\bm{\sigma}} \big) \, \di {\bm{\sigma}}^{N}(t^{k}_{m-1})_\# \mathbb{P}({\bm{\sigma}}) \nonumber 
       \\
       &
       \leq \int_{\Vv^N} \frac{4}{N^{2}}  \mathfrak{w}_{D^{2} \psi} \left( \sum_{n, n'=1}^{N}  \frac{1}{N} \| \delta_{\sigma_{n'}} - \delta_{\sigma_{n}} \|_{\rm BL}  \, \mathbb{P} \big( \sigma^{N}_{n} (t^{k}_{m}) = \sigma_{n'} \big| \, {\bm{\sigma}}^{N}(t^{k}_{m-1}) = {\bm{\sigma}}  \big) \right)  \, \di {\bm{\sigma}}^{N}(t^{k}_{m-1})_\# \mathbb{P}({\bm{\sigma}}) \nonumber
       \\
       & \leq \frac{4 }{N^{2}}  \mathfrak{w}_{D^{2} \psi} \bigg( \frac{2}{N} \bigg) \,, \nonumber
     \end{align}
where in the last inequality we use \eqref{def-tran-prob}. We set
     \begin{align*}
         R^{m}_{2, 4} (\psi, \tau_{k}, N, w) := \int_{t^{k}_{m-1}}^{t^{k}_{m}} \int_{0}^{1} \frac{t - t^{k}_{m-1}}{\tau^{2}_{k}} \, \mathbb{E} \Bigg(& \bigg( D^{2}_{\lambda} \psi \left( t, \lambda^{N}_{t^{k}_{m-1}} + r (\lambda^{N}_{t} - \lambda^{N}_{t^{k}_{m-1}}) \right) 
         \\
         &
         -  D^{2}_{\lambda} \psi \big(t, \lambda^{N}_{t^{k}_{m-1}} \big) \bigg)   \left[ \lambda^{N}_{t^{k}_{m}} - \lambda^{N}_{t^{k}_{m-1}} ,  \lambda^{N}_{t^{k}_{m}} - \lambda^{N}_{t^{k}_{m-1}} \right] \bigg) \, \di r \di t \,.\nonumber
     \end{align*}
     In view of~\eqref{e:10001} we have that
     \begin{align}
         \label{e:10002}
         | R^{m}_{2, 4} (\psi, \tau_{k}, N, w) | \leq O \bigg ( \frac{1}{N^{2}} \bigg) \mathfrak{w}_{D^{2} \psi} \bigg( \frac{2}{N} \bigg)\,.
     \end{align}
     
     We conclude by defining
   \begin{align*}
       R^{m}_{2} (\psi, \tau_{k}, N, w) := \sum_{\ell = 1}^{4}R^{m}_{2, \ell} (\psi, \tau_{k}, N, w) \,. 
   \end{align*}
Thus, if $w\in\left[0, \frac{1}{2(C_\pi +1)}\right]$, combining~\eqref{e:154}, \eqref{e:rest-2.1}, \eqref{e:155}, \eqref{e:R34-1}, \eqref{e:R34-2}, \eqref{e:10001}, and \eqref{e:10002}, we infer~\eqref{e:second-gradient-test}. Otherwise, if \eqref{strong-sel-ass} holds, combining \eqref{e:154-strong}, \eqref{e:10001}, and \eqref{e:10002}, we obtain \eqref{e:second-gradient-test-strong}.
\end{proof}

We are now in a position to conclude the proof of Theorem~\ref{t:almost-equation}.

\begin{proof}[Proof of Theorem~\ref{t:almost-equation}]
For every test $\psi \in C^{2} ([0, T] \times \Pp(\Vv))$ and every $t \in (t^{k}_{m-1}, t^{k}_{m})$ we have that
\begin{align}
\label{e:100}
\frac{\di}{\di t} [ \psi ( t, \lambda^{N}_{t}) ] = \partial_{t} \psi(t, \lambda^{N}_{t}) + D_{\lambda} \psi(t, \lambda^{N}_{t}) \left[ \frac{\lambda^{N}_{t^{k}_{m}} - \lambda^{N}_{t^{k}_{m-1}} }{\tau_{k}} \right]\,.
\end{align}
Integrating~\eqref{e:100} over $\Omega$ and over $(t^{k}_{m-1}, t^{k}_{m})$ we get that
\begin{align}
\label{e:101}
\int_{\Pp (\Vv)} &  \psi ( t^{k}_{m}, \lambda)\, \di \Lambda^{N}_{t^{k}_{m}} (\lambda) - \int_{\Pp (\Vv)} \psi ( t^{k}_{m-1}, \lambda)\, \di \Lambda^{N}_{t^{k}_{m-1}} (\lambda) \\
&
=  \int_{t^{k}_{m-1}}^{t^{k}_{m}} \int_{\Pp(\Vv)}\partial_{t} \psi(t, \lambda) \, \di \Lambda^{N}_{t} (\lambda) \, \di t + \frac{1}{\tau_{k}} \int_{t^{k}_{m-1}}^{t^{k}_{m}} \mathbb{E} \bigg(  D_{\lambda} \psi(t, \lambda^{N}_{t}  ) \left[ \lambda^{N}_{t^{k}_{m}}   - \lambda^{N}_{t^{k}_{m-1}}    \right] \bigg)  \di t \,. \nonumber
\end{align}
Hence, we have to estimate the last term on the right-hand side of~\eqref{e:101}. By Taylor expansion, we  have that
\begin{align}
\label{e:102}
D_{\lambda} & \psi(t, \lambda^{N}_{t}  ) \left[ \lambda^{N}_{t^{k}_{m}}   - \lambda^{N}_{t^{k}_{m-1}}  \right] = D_{\lambda} \psi(t, \lambda^{N}_{t^{k}_{m-1}}  ) \left[ \lambda^{N}_{t^{k}_{m}}   - \lambda^{N}_{t^{k}_{m-1}}    \right] 
\\
&
+ \frac{t - t^{k}_{m-1}}{\tau_{k}} \int_{0}^{1}  D^{2}_{\lambda} \psi \left( t, \lambda^{N}_{t^{k}_{m-1}} + r (\lambda^{N}_{t} - \lambda^{N}_{t^{k}_{m-1}}) \right)  \left[ \lambda^{N}_{t^{k}_{m}}   - \lambda^{N}_{t^{k}_{m-1}}  , \lambda^{N}_{t^{k}_{m}}   - \lambda^{N}_{t^{k}_{m-1}}    \right] \, \di r \,. \nonumber
\end{align}
Averaging equation~\eqref{e:102} over $\Om$ and integrating over the time interval~$[t^{k}_{m-1}, t^{k}_{m}]$ we deduce both~\eqref{e:approx_eq} and \eqref{e:approx_eq-strong} thanks to Lemmas~\ref{l:gradient-test}--\ref{l:second-gradient-test} after setting
\begin{align*}
    & \rho_{m} ( \psi, \tau_{k}, N, w) := R^{m}_{1} ( \psi, \tau_{k}, N, w) + R^{m}_{2} ( \psi, \tau_{k}, N, w) \,.
\end{align*}
\end{proof}

\begin{corollary}
    \label{c:almost-equation}
    Assume that $(\pi.1)$ holds. For every $k \in \mathbb{N} \setminus \{0\}$, and every $\psi \in C^{2} ([0, T] \times \Pp(\Vv))$ we have that
    \begin{itemize}
        \item[(i)] if $w\in \left[0, \frac{1}{2(C_\pi+1)}\right]$, then
    \begin{align}
    \label{e:approx_eq-2}
        \int_{\Pp(\Vv)} & \psi (T, \lambda) \, \di \Lambda^{N}_{T} (\lambda) -     \int_{\Pp(\Vv)} \psi (0, \lambda) \, \di \Lambda^{N}_{0} (\lambda) = \int_{0}^{T} \int_{\Pp(\Vv)} \partial_{t} \psi(t, \lambda) \, \di \Lambda^{N}_{t} (\lambda) \, \di t 
        \\
        &
        +  \frac{w}{\tau_{k} N}  \int_{0}^{T} \int_{\Pp(\Vv)} D_{\lambda} \psi (t, \lambda) [(\pi_{\lambda} - \overline{\pi}_{\lambda}) \lambda] \, \di \overline{\Lambda}^{N}_{t} (\lambda) \, \di t  \nonumber
        \\
        &
        +\frac{1}{2 \tau_{k} N^{2}} \int_{0}^{T} \int_{\Pp(\Vv)} \int_{\Vv} \int_{\Vv} D^{2}_{\lambda} \psi (t, \lambda) [\delta_{\sigma} - \delta_{\overline{\sigma}} , \delta_{\sigma} - \delta_{\overline{\sigma}} ]  \, \di \lambda(\sigma) \, \di \lambda(\overline{\sigma}) \, \di \overline{\Lambda}^{N}_{t} (\lambda) \, \di t + \widetilde{\rho} (\psi, \tau_{k}, N, w)\,,\nonumber
    \end{align}
    where
    \begin{align}
    \label{e:rhok-2}
        | \widetilde{\rho} (\psi, \tau_{k}, N, w)| \leq & \ O \bigg(\frac{  w^{2}  }{\tau_{k} N} \bigg) \| D \psi\|_{L^{\infty} ([0, T]\times \Pp(\Vv))}  + O \bigg( \frac{  w }{\tau_{k} N^{2}} \bigg) \| D^{2} \psi\|_{L^{\infty} ([0, T]\times \Pp(\Vv))}
        \\
        &
        + O \bigg( \frac{ 1}{\tau_{k} N^{2}} \bigg)  \bigg( \mathfrak{w}_{D^{2}\psi} (\tau_{k}) +  \mathfrak{w}_{D^{2}\psi} \bigg( \frac{1}{N}\bigg) \bigg)   \,; \nonumber
    \end{align}
\item[(ii)] if \eqref{strong-sel-ass} holds, then
\begin{align}
    \label{e:approx_eq-2-strong}
        \int_{\Pp(\Vv)} & \psi (T, \lambda) \, \di \Lambda^{N}_{T} (\lambda) -     \int_{\Pp(\Vv)} \psi (0, \lambda) \, \di \Lambda^{N}_{0} (\lambda) = \int_{0}^{T} \int_{\Pp(\Vv)} \partial_{t} \psi(t, \lambda) \, \di \Lambda^{N}_{t} (\lambda) \, \di t 
        \\
        &
        +  \frac{1}{\tau_{k} N}  \int_{0}^{T} \int_{\Pp(\Vv)} D_{\lambda} \psi (t, \lambda)  \left[\frac{w( \pi_{\lambda} - \overline{\pi}_{\lambda}) \lambda}{1 + w( \overline{\pi}_{\lambda} - 1) } \right] \, \di \overline{\Lambda}^{N}_{t} (\lambda) \, \di t  
         + \widetilde{\zeta} (\psi, \tau_{k}, N)\,,\nonumber
    \end{align}
    where
    \begin{align}
    \label{e:rhok-2-strong}
        | \widetilde{\zeta} (\psi, \tau_{k}, N)| \leq & \  O \bigg( \frac{  1 }{\tau_{k} N^{2}} \bigg) \| D^{2} \psi\|_{L^{\infty} ([0, T]\times \Pp(\Vv))} \,. 
    \end{align}
\end{itemize}
\end{corollary}

\begin{proof}
    It is enough to sum up~\eqref{e:approx_eq} and \eqref{e:approx_eq-strong} respectively, for $m \in \{1, \ldots, k\}$, recall that $\tau_k=\frac{T}{k}$, and set
    \begin{align*}
        & \widetilde{\rho} (\psi, \tau_{k}, N, w) := \sum_{m=1}^{k} \rho_{m} (\psi, \tau_{k}, N, w)\,.
    \end{align*}
\end{proof}

For later use, we also write the following variant of Corollary~\ref{c:almost-equation}, where we take test functions independent of~$t$ and we localize~\eqref{e:approx_eq-2} and \eqref{e:approx_eq-2-strong} respectively.

\begin{corollary}
    \label{c:almost-equation-2}
     Assume that $(\pi.1)$ holds. There exists a positive constant $C>0$ such that for every $k \in \mathbb{N} \setminus \{0\}$, every $s < t \in [0, T]$, and every $\psi \in C^{2} ( \Pp(\Vv))$ we have that
     \begin{itemize}
         \item[(i)] if $w \in \left[0, \frac{1}{2(C_{\pi} + 1)}\right] $, then
    \begin{align}
    \label{e:approx_eq-3}
        \int_{\Pp(\Vv)} & \psi (\lambda) \, \di \Lambda^{N}_{t} (\lambda) -     \int_{\Pp(\Vv)} \psi (\lambda) \, \di \Lambda^{N}_{s} (\lambda) 
        \\
        &
        \leq  C \bigg( \frac{w}{\tau_{k} N} \| D\psi\|_{L^{\infty} (\Pp(\Vv))} +  \frac{2}{\tau_{k} N^{2}} \| D^{2}\psi\|_{L^{\infty} (\Pp(\Vv))} \bigg) | t - s| + \mu ( s, t, \psi, \tau_{k}, N, w)\,,\nonumber
    \end{align}
    where
    \begin{align*}
        | \mu ( s, t, \psi, \tau_{k}, N, w) | \leq & \ O \bigg(\frac{  w^{2}  }{\tau_{k} N} \bigg) \| D \psi\|_{L^{\infty} ( \Pp(\Vv))}  + O \bigg( \frac{  w }{\tau_{k} N^{2}} \bigg) \| D^{2} \psi\|_{L^{\infty} ( \Pp(\Vv))}
        \\
        &
        + O \bigg( \frac{ 1}{\tau_{k} N^{2}} \bigg) \mathfrak{w}_{D^{2}\psi} \bigg( \frac{1}{N}\bigg)  \,; \nonumber 
    \end{align*}
    \item[(ii)] if \eqref{strong-sel-ass} holds, then
     \begin{align}
    \label{e:approx_eq-3-strong}
        \int_{\Pp(\Vv)} & \psi (\lambda) \, \di \Lambda^{N}_{t} (\lambda) -     \int_{\Pp(\Vv)} \psi (\lambda) \, \di \Lambda^{N}_{s} (\lambda) 
        \\
        &
        \leq  C\left(\frac{1}{\tau_{k} N} \| D\psi\|_{L^{\infty} (\Pp(\Vv))}\right)| t - s| + \hat{\mu} ( s, t, \psi, \tau_{k}, N)\,,\nonumber
    \end{align}
    where
    \begin{align*}
        | \hat{\mu} ( s, t, \psi, \tau_{k}, N) | \leq & \ O \bigg( \frac{ 1 }{\tau_{k} N^{2}} \bigg) \| D^{2} \psi\|_{L^{\infty} ( \Pp(\Vv))}.
    \end{align*}
    \end{itemize}
\end{corollary}

\begin{proof}
Let $s < t \in [0, T]$ and let $m, h \in \{1, \ldots, k\}$ be such that $s \in [t^{k}_{h-1}, t^{k}_{h})$ and $t \in [t^{k}_{m-1}, t^{k}_{m})$. Let us assume that $m >h+1$, otherwise the argument is even simpler. Assume the case $w\in \left[0,\frac{1}{2(C_\pi+1)}\right]$. Applying Theorem~\ref{t:almost-equation} to the test $\psi \in C^{2} (\Pp(\Vv))$ (independent of time) we have that
\begin{align*}
     \int_{\Pp(\Vv)} & \psi ( \lambda) \, \di \Lambda^{N}_{t^{k}_{m-1}} (\lambda) -     \int_{\Pp(\Vv)} \psi (\lambda) \, \di \Lambda^{N}_{t^{k}_{h}} (\lambda) = \frac{w}{\tau_{k} N}  \int_{t^{k}_{h}}^{t^{k}_{m-1}} \int_{\Pp(\Vv)} D_{\lambda} \psi (\lambda) [(\pi_{\lambda} - \overline{\pi}_{\lambda}) \lambda] \, \di \overline{\Lambda}^{N}_{r} (\lambda) \, \di r  \nonumber
        \\
        &
        +\frac{1}{2 \tau_{k} N^{2}} \int_{t^{k}_{h}}^{t^{k}_{m-1}} \int_{\Pp(\Vv)} \int_{\Vv} \int_{\Vv} D^{2}_{\lambda} \psi ( \lambda) [\delta_{\sigma} - \delta_{\overline{\sigma}} , \delta_{\sigma} - \delta_{\overline{\sigma}} ]  \, \di \lambda(\sigma) \, \di \lambda(\overline{\sigma}) \, \di \overline{\Lambda}^{N}_{r} (\lambda) \, \di r \nonumber
        \\
        &
        + \vphantom{\int_{t^{k}_{m-1}}^{t^{k}_{m}}} \overline{\rho} ( h, m, \psi, \tau_{k}, N, w)\,,\nonumber 
\end{align*}
where we have set 
\begin{align*}
\overline{\rho} ( h, m, \psi, \tau_{k}, N, w) := \sum_{n=h+1}^{m-1} \rho_{n} ( \psi, \tau_{k}, N, w)\,.
\end{align*}
This, together with~$(\pi.1)$ and~\eqref{e:pi-bdd}, implies that
\begin{align}
\label{e:lip-1}
    \int_{\Pp(\Vv)} & \psi ( \lambda) \, \di \Lambda^{N}_{t^{k}_{m-1}} (\lambda) -     \int_{\Pp(\Vv)} \psi (\lambda) \, \di \Lambda^{N}_{t^{k}_{h}} (\lambda)
    \\
    &
    \leq C\bigg( \frac{w}{\tau_{k} N} \| D\psi\|_{L^{\infty} (\Pp(\Vv))} + \frac{2}{\tau_{k} N^2} \| D^{2} \psi\|_{L^{\infty} (\Pp(\Vv))} \bigg) | t^{k}_{m-1} - t^{k}_{h} | + \overline{\rho}\,, \nonumber
\end{align}
for a constant~$C>0$ independent of $s, t, \psi, \tau_k, N, w$. We note that, since $\psi$ is independent of time, by Theorem \ref{t:almost-equation}, we get
\begin{align*}
|\overline{\rho} ( h, m, \psi, \tau_{k}, N, w)|\leq & O \bigg(\frac{  w^{2}  }{\tau_{k} N} \bigg) \| D \psi\|_{L^{\infty} ( \Pp(\Vv))}  + O \bigg( \frac{  w }{\tau_{k} N^{2}} \bigg) \| D^{2} \psi\|_{L^{\infty} ( \Pp(\Vv))}
        \\
        &
        + O \bigg( \frac{ 1}{\tau_{k} N^{2}} \bigg) \mathfrak{w}_{D^{2}\psi} \bigg( \frac{1}{N}\bigg) \,.
\end{align*}
Hence, we have to add to~\eqref{e:lip-1} only the contributions in the intervals $[t^{k}_{m-1}, t]$ and $[s, t^{k}_{h}]$. We deal with the interval $[t^{k}_{m-1}, t]$, as the argument for the second one is very similar. As in the proof of Theorem~\ref{t:almost-equation}, we have that
\begin{align*}
    & \int_{\Pp(\Vv)}  \psi(\lambda) \, \di \Lambda^{N}_{t} (\lambda) - \int_{\Pp(\Vv)} \psi(\lambda) \, \di \Lambda^{N}_{t^{k}_{m-1}} (\lambda) = \frac{1}{\tau_{k}} \int_{t^{k}_{m-1}}^{t} \mathbb{E} \bigg(  D_{\lambda} \psi( \lambda^{N}_{r}  ) \left[ \lambda^{N}_{t^{k}_{m}}   - \lambda^{N}_{t^{k}_{m-1}}    \right] \bigg)  \di r
    \\
    &
    =  \frac{1}{\tau_{k}} \int_{t^{k}_{m-1}}^{t}  \mathbb{E} \left(  D_{\lambda} \psi(\lambda^{N}_{t^{k}_{m-1}}  ) \left[ \lambda^{N}_{t^{k}_{m}}  - \lambda^{N}_{t^{k}_{m-1}}    \right] \right)\, \di r  
    \\
    &
    \quad +  \frac{1}{\tau_{k}} \int_{t^{k}_{m-1}}^{t} \int_{0}^{1} \frac{r - t^{k}_{m-1}}{\tau_{k}}  \, \mathbb{E} \bigg( D^{2}_{\lambda} \psi \big( \lambda^{N}_{t^{k}_{m-1}}  + \theta ( \lambda^{N}_{r}  - \lambda^{N}_{t^{k}_{m-1}} ) \big) \left[ \lambda^{N}_{t^{k}_{m}}   - \lambda^{N}_{t^{k}_{m-1}}  , \lambda^{N}_{t^{k}_{m}}   - \lambda^{N}_{t^{k}_{m-1}}    \right]\bigg) \,\di \theta \, \di r \,.
\end{align*}
By the regularity of the test function and repeating the arguments of Lemmas~\ref{l:gradient-test}--\ref{l:second-gradient-test}, we infer that (notice that now $\psi$ is independent of time)
\begin{align}
\label{e:lip-2}
    \int_{\Pp(\Vv)} &  \psi(\lambda) \, \di \Lambda^{N}_{t} (\lambda) - \int_{\Pp(\Vv)} \psi(\lambda) \, \di \Lambda^{N}_{t^{k}_{m-1}} (\lambda) 
    \\
    &
    \leq C \bigg( \frac{w}{\tau_{k} N} \| D\psi\|_{L^{\infty} (\Pp(\Vv))} + \frac{ 2  }{  \tau_{k} N^{2}} \| D^{2}\psi\|_{L^{\infty} (\Pp(\Vv))} \bigg) | t - t^{k}_{m-1}| + \overline{\rho}_1(m,t,\psi,N,w), \nonumber
\end{align}
for a positive constant $C>0$ independent of $s, t, \psi, \tau_k, N, w$, where
\begin{align*}
|\overline{\rho}_1 (m, t, \psi, N, w)|\leq & O \bigg(\frac{  w^{2}  }{ N} \bigg) \| D \psi\|_{L^{\infty} ( \Pp(\Vv))}  + O \bigg( \frac{  w }{N^{2}} \bigg) \| D^{2} \psi\|_{L^{\infty} ( \Pp(\Vv))}
        + O \bigg( \frac{ 1}{ N^{2}} \bigg) \mathfrak{w}_{D^{2}\psi} \bigg( \frac{1}{N}\bigg) \,.
\end{align*}
A similar argument leads to the estimate
\begin{align}
\label{e:lip-3}
    \int_{\Pp(\Vv)} &  \psi(\lambda) \, \di \Lambda^{N}_{t^{k}_{h}} (\lambda) - \int_{\Pp(\Vv)} \psi(\lambda) \, \di \Lambda^{N}_{s} (\lambda) 
    \\
    &
    \leq C \bigg( \frac{w}{\tau_{k} N} \| D\psi\|_{L^{\infty} (\Pp(\Vv))} + \frac{ 2  }{  \tau_{k} N^{2}} \| D^{2}\psi\|_{L^{\infty} (\Pp(\Vv))} \bigg) | t^{k}_{h} - s |  + \overline{\rho}_2(s,h,\psi, N,w),  \nonumber
\end{align}
where $\overline{\rho}_2$ is estimated as $\overline{\rho}_1$. Combining~\eqref{e:lip-1}--\eqref{e:lip-3} and defining 
$$
\mu(s,t,\psi,\tau_k,N,w):= \overline{\rho}+\overline{\rho}_1+\overline{\rho}_2,
$$
we conclude for~\eqref{e:approx_eq-3}. \\
Finally, the proof of \eqref{e:approx_eq-3-strong} under Assumption \eqref{strong-sel-ass} follows by the exact same argument.
\end{proof}

\section{Proofs of the Theorems~\ref{t:compactness-1}, ~\ref{t:compactness-3}, and \ref{t:compactness-3-strong}}
\label{s:compactness}

In order to prove Theorems \ref{t:compactness-1} and \ref{t:compactness-3}, we fix $N_{k} \in \mathbb{N} \setminus \{0\}$ and $w_{k} \in \left[0, \frac{1}{2( C_{\pi} + 1)}\right]$ two sequences such that $N_{k} \to +\infty$ and~$w_{k} \to 0$ as $k \to \infty$. We recall that $\tau_k=\frac{T}{k}$. The precise scalings are given in~\eqref{e:scalings} and in~\eqref{e:scalings-2}. Both proofs rely on the expansion obtained in Corollary~\ref{c:almost-equation-2}

\subsection{Proof of Theorem~\ref{t:compactness-1}}
\label{sub:thm1}

We fix $w_{k}$, $\tau_{k}$, and $N_{k}$ such that~\eqref{e:scalings} holds. We divide the convergence proof in two steps. The first one concerns compactness of the sequence $\Lambda^{N_{k}}$; the second one shows that the limit curve $\Lambda$ is a very weak solution to the replicator equations with genetic drift. 

\smallskip

\paragraph{\bf Step 1: Compactness of $\Lambda^{N_{k}}$.} We divide the proof into two sub-steps. 

\noindent {\em Step 1.1: Construction of a limit curve.} We notice that in view of~\eqref{e:scalings} and of Corollary~\ref{c:almost-equation-2}, there exists a constant $C>0$ such that for every $k \in \mathbb{N} \setminus\{0\}$, every $s < t \in [0, T]$, and every $\psi \in C^{2} (\Pp(\Vv))$ we have that
\begin{align}
\label{e:something}
     \int_{\Pp(\Vv)} & \psi (\lambda) \, \di \Lambda^{ N_{k}}_{t} (\lambda) -     \int_{\Pp(\Vv)} \psi (\lambda) \, \di \Lambda^{ N_{k}}_{s} (\lambda) 
        \\
        &
        \leq  C \left(  \| D\psi\|_{L^{\infty} (\Pp(\Vv))} + \| D^{2}\psi\|_{L^{\infty} (\Pp(\Vv))} \right) | t - s| + C \mathfrak{w}_{D^{2}\psi} \left(\frac{1}{N_k}\right) \,.  \nonumber
\end{align}
We select a dense and at most countable subset $D$ of $[0, T]$. By Lemma~\ref{l:dZ} and by a diagonal argument, we may find a subsequence independent of $t \in D$ (not relabeled) and $\Lambda_{t} \in \Pp(\Pp(\Vv))$ such that 
\begin{align*}
    \lim_{k\to \infty} \di_{\mathcal{Z}} (\Lambda^{N_{k}}_{t}, \Lambda_{t}) = \lim_{k \to \infty} W_{1} (\Lambda^{N_{k}}_{t}, \Lambda_{t}) = 0 \qquad \text{for $t \in D$}\,.
\end{align*}
Notice that we exploit the compactness of $(\Pp(\Vv), \| \cdot\|_{\rm BL})$ for the convergence in $W_{1}$. For every $s<t \in D$ and every $\psi \in C^{2} (\Pp(\Vv))$ we pass to the limit in~\eqref{e:something} as $k \to \infty$, obtaining
\begin{align}
    \label{e:130}
    \int_{\Pp(\Vv)} \psi(\lambda) \, \di (\Lambda_{t} - \Lambda_{s}) (\lambda) \leq C \big( \| D \psi\|_{L^{\infty} (\Pp(\Vv))} + \| D^{2} \psi\|_{L^{\infty} (\Pp(\Vv))} \big)  |t - s|\,.
\end{align}
In particular,~\eqref{e:130} holds for $\psi \in \mathcal{Z}$. Hence,
\begin{align}
    \label{e:131}
    \di_{\mathcal{Z}} (\Lambda_{t}, \Lambda_{s}) \leq C | t - s| \qquad \text{for $s, t \in D$.}
\end{align}

We now have to extend $(\Lambda_{t})_{t \in D}$ to the whole interval $[0, T]$. For $t \in [0, T] \setminus D$, we define $\Lambda_{t} \in \Pp(\Pp( \Vv))$ as limit in $\di_{\mathcal{Z}}$ of $\Lambda_{t_{j}}$ for a sequence $t_{j} \in D$ such that $t_{j} \to t$. First we show that such a definition is well-posed. Indeed, by Lemma~\ref{l:dZ} we have that, for every $t_{j} \in D$ with $t_{j} \to t$, there exists $\Sigma \in \Pp(\Pp(\Vv))$ such that, up to a not relabeled subsequence, 
\begin{align}
\label{e:190000}
    \lim_{j \to \infty} \di_{\mathcal{Z}} (\Lambda_{t_{j}}, \Sigma) = \lim_{j \to \infty} W_{1}(\Lambda_{t_{j}}, \Sigma) = 0\,.
\end{align}
Let $s_{j} \in D$ and $\overline{\Sigma} \in \Pp(\Pp(\Vv))$ be such that $s_{j} \to t$ and
\begin{align*}
    \lim_{j \to \infty} \di_{\mathcal{Z}} (\Lambda_{s_{j}}, \overline\Sigma) = \lim_{j \to \infty} W_{1}(\Lambda_{s_{j}}, \overline\Sigma) = 0\,.
\end{align*}
By~\eqref{e:131} we have that
\begin{align}
\label{e:190001}
  \di_{\mathcal{Z}} (\Sigma, \overline{\Sigma}) = \lim_{j\to\infty}  \di_{\mathcal{Z}} (\Lambda_{t_{j}}, \Lambda_{s_{j}}) \leq \lim_{j \to \infty} \, C | t_{j} - s_{j}| = 0\,.
\end{align}
Thus, it must be $\Sigma = \overline{\Sigma}$. This implies that $\Lambda_{t} \in \Pp(\Pp(\Vv))$ is well-defined for every $t \in [0, T]$. 

\noindent {\em Step 1.2: Pointwise convergence and continuity.} We now show  that 
\begin{align}
    \label{e:limit}
    \lim_{k \to \infty} \di_{\mathcal{Z} } ( \Lambda^{ N_{k}}_{t} , \Lambda_{t}) = \lim_{k \to \infty} W_{1} ( \Lambda^{ N_{k}}_{t}, \Lambda_{t}) = 0\qquad \text{for every $t \in [0, T]$.}
\end{align}
Notice that~\eqref{e:limit} holds for $t \in D$ by construction. Let us fix $t \in [0, T] \setminus D$ and let $t_{j} \in D$ be such that $t_{j} \to t$ and $\Lambda_{t_{j}} \to \Lambda_{t}$ both in $\di_{\mathcal{Z}}$ and in $W_{1}$. We may further assume that, along a not relabeled subsequence,
\begin{align*}
    \lim_{k \to \infty} \di_{\mathcal{Z}} (\Lambda^{ N_{k}}_{t}, \overline{\Lambda}) = \lim_{k \to \infty} W_{1} (\Lambda^{ N_{k}}_{t}, \overline{\Lambda}) = 0 \qquad \text{for some $\overline{\Lambda} \in \Pp(\Pp(\Vv))$.}
\end{align*}
In view of Corollary~\ref{c:almost-equation-2} and of~\eqref{e:scalings} we have that for every $\psi \in C^{2} (\Pp(\Vv))$ it holds
\begin{align}
    \label{e:134}
    \!\!\!\!\! \int_{\Pp(\Vv)} \!\!\!\! \psi(\lambda)  \di (\Lambda^{ N_{k}}_{t} - \Lambda^{ N_{k}}_{t_{j}}) (\lambda) \leq & \ C \big( \| D \psi\|_{L^{\infty} (\Pp(\Vv))} + \| D^{2}\psi\|_{L^{\infty} (\Pp(\Vv))} \big) | t - t_{j} |   +C \mathfrak{w}_{D^{2}\psi} \left(\frac{1}{N_k}\right).
\end{align}
Passing to the limit as $k \to \infty$ in~\eqref{e:134} we deduce that for every $\psi \in C^{2} (\Pp(\Vv))$
\begin{align}
    \label{e:135}
    \int_{\Pp(\Vv)} \psi(\lambda) \, \di (\overline{\Lambda} - \Lambda_{t_{j}}) (\lambda) \leq C \big( \| D \psi\|_{L^{\infty} (\Pp(\Vv))} + \| D^{2}\psi\|_{L^{\infty} (\Pp(\Vv))} \big) | t - t_{j} |  \,.
\end{align}
Thus, we infer that
\begin{align}
    \label{e:136}
   \di_{\mathcal{Z}} ( \overline{\Lambda} , \Lambda_{t_{j}})  \leq C  | t - t_{j} |  \,.
\end{align}
Passing to the limit as $j \to \infty$ in~\eqref{e:136} we conclude that $\overline{\Lambda} = \Lambda_{t}$. Hence, \eqref{e:limit} holds.

We further notice that~\eqref{e:131} can be extended to every $s, t \in [0, T]$, which implies that $t\mapsto \Lambda_{t}$ is Lipschitz continuous from $[0, T]$ with values in $(\Pp(\Pp(\Vv)), \di_{\mathcal{Z}})$. This, together with the compactness of~$(\Pp(\Pp(\Vv)), W_{1})$ implies that $t\mapsto \Lambda_{t}$ is continuous in the $W_{1}$-metric as well. This concludes the proof of the compactness step.

\smallskip
\paragraph{\bf Step 2: The limit equation.}
     For $\psi \in C^{2}([0, T] \times \Pp(\Vv))$ we pass to the limit term by term in~\eqref{e:approx_eq-2}. In view of the $W_{1}$-convergence of $\Lambda^{N_{k}}_{t}$ to $\Lambda_{t}$ for $t \in [0, T]$, by the Dominated Convergence Theorem we immediately deduce that
\begin{align}
\label{e:time-der-conv}
    \lim_{k \to \infty} \int_{0}^{T} \int_{\Pp(\Vv)} \partial_{t} \psi(t, \lambda) \, \di \Lambda^{N_{k}}_{t} (\lambda) \, \di t = \int_{0}^{T} \int_{\Pp(\Vv)} \partial_{t} \psi(t, \lambda) \, \di \Lambda_{t} (\lambda) \, \di t  \,.
\end{align}
We notice that, by Lemma~\ref{l:distance},
\begin{align*}
   \lim_{k \to \infty}  W_{1} ( \overline{\Lambda}^{N_{k}}_{t}, \Lambda_{t})  = 0 \qquad \text{for every $t \in [0, T]$.} 
\end{align*}
Moreover, $\lambda \mapsto ( \pi_{\lambda} - \overline{\pi}_{\lambda}) \lambda$ is continuous from $(\Pp(\Vv), \| \cdot\|_{\rm BL})$ into $(E_{\Pp(\Vv)}, \| \cdot\|_{\rm BL})$ in view of $(\pi.1)$ and of~\eqref{e:pi-bdd}. Hence, the Dominated Convergence Theorem yields
\begin{align}
\label{e:ineeditafter}
    \lim_{k \to \infty}  \frac{w_{k}}{\tau_{k} N_{k}} \int_{0}^{T} \int_{\Pp(\Vv)} &  D_{\lambda} \psi(t, \lambda) [(\pi_{\lambda} - \overline{\pi}_{\lambda} ) \lambda] \, \di \overline{\Lambda}^{ N_{k}}_{t} (\lambda) \, \di t 
    \\
    &
    = \gamma\int_{0}^{T}\int_{\Pp(\Vv)}  D_{\lambda} \psi(t, \lambda) [(\pi_{\lambda} - \overline{\pi}_{\lambda} ) \lambda] \, \di \Lambda_{t} (\lambda) \, \di t  \,. \nonumber
\end{align}

We prove that for every $t \in [0, T]$ the map
\begin{align}
\label{e:145}
    \lambda \mapsto \int_{\Vv} \int_{\Vv} D^{2}_{\lambda} \psi(t, \lambda) [\delta_{\sigma} - \delta_{\overline{\sigma}} , \delta_{\sigma} - \delta_{\overline{\sigma}} ]\, \di \lambda(\sigma) \, \di \lambda(\overline{\sigma})
\end{align}
is continuous from $(\Pp(\Vv) , \| \cdot\|_{\rm BL})$ into $\R$. If this is the case, the convergence of $\overline{\Lambda}^{ N_{k}}_{t}$ to $\Lambda_{t}$ in the $1$-Wasserstein distance and the Dominated Convergence Theorem yield
\begin{align}
\label{e:146}
    \lim_{k \to \infty}  &  \, \frac{1}{2 \tau_{k}  N_{k}^{2}} \int_{0}^{T} \int_{\Pp(\Vv)} \int_{\Vv} \int_{\Vv} D^{2}_{\lambda} \psi(t, \lambda) [\delta_{\sigma} - \delta_{\overline{\sigma}} , \delta_{\sigma} - \delta_{\overline{\sigma}} ]\, \di \lambda(\sigma) \, \di \lambda(\overline{\sigma})\, \di \overline{\Lambda}^{ N_{k}}_{t} (\lambda) \, \di t 
    \\
    &
    = \frac{\nu}{2 } \int_{0}^{T} \int_{\Pp(\Vv)} \int_{\Vv} \int_{\Vv} D^{2}_{\lambda} \psi(t, \lambda) [\delta_{\sigma} - \delta_{\overline{\sigma}} , \delta_{\sigma} - \delta_{\overline{\sigma}} ]\, \di \lambda(\sigma) \, \di \lambda(\overline{\sigma})\, \di \Lambda_{t} (\lambda) \, \di t \,. \nonumber
\end{align}
For $\lambda_{j}, \lambda \in \Pp(\Vv)$ such that $\| \lambda_{j} - \lambda\|_{\rm BL}\to 0$ we estimate
\begin{align}
\label{e:140}
    \bigg| \int_{\Vv}  \int_{\Vv} & D^{2}_{\lambda} \psi(t, \lambda) [\delta_{\sigma} - \delta_{\overline{\sigma}} , \delta_{\sigma} - \delta_{\overline{\sigma}} ]\, \di \lambda( \sigma ) \, \di \lambda(\overline{\sigma}) - \int_{\Vv} \int_{\Vv} D^{2}_{\lambda} \psi(t, \lambda_{j} ) [\delta_{\sigma} - \delta_{\overline{\sigma}} , \delta_{\sigma} - \delta_{\overline{\sigma}} ]\, \di \lambda_{j}(\sigma) \, \di \lambda_{j}(\overline{\sigma}) \bigg| 
    \\
    &
    \leq \bigg| \int_{\Vv} \int_{\Vv} \Big( D^{2}_{\lambda} \psi(t, \lambda_{j}) - D^{2}_{\lambda} \psi(t, \lambda) \Big) [\delta_{\sigma} - \delta_{\overline{\sigma}} , \delta_{\sigma} - \delta_{\overline{\sigma}} ]\, \di \lambda_{j}(\sigma) \, \di \lambda_{j}(\overline{\sigma}) \bigg| \nonumber
    \\
    &
    \quad + \bigg| \int_{\Vv} \int_{\Vv} D^{2}_{\lambda} \psi(t, \lambda) [\delta_{\sigma} - \delta_{\overline{\sigma}} , \delta_{\sigma} - \delta_{\overline{\sigma}} ]\, \di ( \lambda_{j} - \lambda) (\sigma) \, \di \lambda_{j}(\overline{\sigma})  \bigg| \nonumber
    \\
    &
    \quad + \bigg| \int_{\Vv} \int_{\Vv} D^{2}_{\lambda} \psi(t, \lambda) [\delta_{\sigma} - \delta_{\overline{\sigma}} , \delta_{\sigma} - \delta_{\overline{\sigma}} ]\, \di \lambda (\sigma) \, \di ( \lambda_{j} - \lambda )(\overline{\sigma})  \bigg| \,. \nonumber
\end{align}
The first term on the right-hand side of~\eqref{e:140} can be estimated by
\begin{align}
\label{e:143}
\!\!    \bigg| \int_{\Vv} \int_{\Vv} & \Big( D^{2}_{\lambda} \psi(t, \lambda_{j}) - D^{2}_{\lambda} \psi(t, \lambda) \Big) [\delta_{\sigma} - \delta_{\overline{\sigma}} , \delta_{\sigma} - \delta_{\overline{\sigma}} ]\, \di \lambda_{j}(\sigma) \, \di \lambda_{j}(\overline{\sigma}) \bigg| 
    \leq 4 \mathfrak{w}_{D^{2}\psi} ( \| \lambda_{j} - \lambda\|_{\rm BL})\,. 
\end{align}
As for the second term on the right-hand side of~\eqref{e:140}, we notice that the map
\begin{align}
\label{e:141}
    \sigma \mapsto D^{2}_{\lambda} \psi(t, \lambda) [\delta_{\sigma} - \delta_{\overline{\sigma}} , \delta_{\sigma} - \delta_{\overline{\sigma}} ]
\end{align}
is Lipschitz continuous from $(\Vv, \di)$ into $\R$, with uniform Lipschitz constant with respect to~$\overline{\sigma} \in \Vv$. Indeed, for every $\sigma_{1}, \sigma_{2}, \overline{\sigma} \in \Vv$ we have that
\begin{align*}
    & \Big| D^{2}_{\lambda} \psi(t, \lambda) [\delta_{\sigma_{1}} - \delta_{\overline{\sigma}} , \delta_{\sigma_{1}} - \delta_{\overline{\sigma}} ] -  D^{2}_{\lambda} \psi(t, \lambda) [\delta_{\sigma_{2}} - \delta_{\overline{\sigma}} , \delta_{\sigma_{2}} - \delta_{\overline{\sigma}} ] \Big|
    \\
    &
    \qquad \leq \Big| D^{2}_{\lambda} \psi(t, \lambda) [\delta_{\sigma_{1}} - \delta_{\overline{\sigma}} , \delta_{\sigma_{1}} - \delta_{\overline{\sigma}} ] -  D^{2}_{\lambda} \psi(t, \lambda) [\delta_{\sigma_{2}} - \delta_{\overline{\sigma}} , \delta_{\sigma_{1}} - \delta_{\overline{\sigma}} ] \Big|
    \\
    &
    \qquad \quad + \Big| D^{2}_{\lambda} \psi(t, \lambda) [\delta_{\sigma_{2}} - \delta_{\overline{\sigma}} , \delta_{\sigma_{1}} - \delta_{\overline{\sigma}} ] -  D^{2}_{\lambda} \psi(t, \lambda) [\delta_{\sigma_{2}} - \delta_{\overline{\sigma}} , \delta_{\sigma_{2}} - \delta_{\overline{\sigma}} ] \Big|
    \\
    &
    \qquad \leq \| D^{2}_{\lambda} \psi\|_{L^{\infty} ([0, T] \times \Pp(\Vv))} \| \delta_{\sigma_{1}} - \delta_{\sigma_{2}}\|_{\rm BL}  \Big( \| \delta_{\sigma_{1}} - \delta_{\overline{\sigma}}\|_{\rm BL} +  \| \delta_{\sigma_{2}} - \delta_{\overline{\sigma}}\|_{\rm BL}\Big)
    \\
    &
    \qquad \leq 4 \| D^{2}_{\lambda} \psi\|_{L^{\infty} ([0, T] \times \Pp(\Vv))} \, \di (\sigma_{1}, \sigma_{2})\,.
\end{align*}
In particular, the Lipschitz constant of the map defined in~\eqref{e:141} is controlled from above by $4 \| D^{2}_{\lambda} \psi\|_{L^{\infty} ([0, T] \times \Pp(\Vv))} $. This implies that
\begin{align}
\label{e:142}
    \bigg| \int_{\Vv} \int_{\Vv} & D^{2}_{\lambda} \psi(t, \lambda) [\delta_{\sigma} - \delta_{\overline{\sigma}} , \delta_{\sigma} - \delta_{\overline{\sigma}} ]\, \di ( \lambda_{j} - \lambda) (\sigma) \, \di \lambda_{j}(\overline{\sigma})  \bigg|
    \\
    &
    \leq \vphantom{\int_{\Vv}} 8 \| D^{2}_{\lambda} \psi\|_{L^{\infty} ([0, T]\times \Pp(\Vv))} \| \lambda_{j} - \lambda\|_{\rm BL}\,. \nonumber
\end{align}
In the same way we deduce that 
\begin{align}
    \label{e:144}
     \bigg| \int_{\Vv} \int_{\Vv} & D^{2}_{\lambda} \psi(t, \lambda) [\delta_{\sigma} - \delta_{\overline{\sigma}} , \delta_{\sigma} - \delta_{\overline{\sigma}} ]\, \di \lambda (\sigma) \, \di ( \lambda_{j} - \lambda) (\overline{\sigma})  \bigg|
     \\
     &
     \leq \vphantom{\int_{\Vv}} 8 \| D^{2}_{\lambda} \psi\|_{L^{\infty} ([0, T]\times \Pp(\Vv))} \| \lambda_{j} - \lambda\|_{\rm BL}\,. \nonumber
\end{align}
Combining~\eqref{e:143}, \eqref{e:142}, and~\eqref{e:144} we conclude the map in~\eqref{e:145} is continuous and~\eqref{e:146} holds.

Finally, passing to the limit in~\eqref{e:approx_eq-2} we infer~\eqref{e:fleming-viot}. This concludes the proof of Theorem~\ref{t:compactness-1}.

\subsection{Proof of Theorem~\ref{t:compactness-3}}
\label{sub:thm2}
Under the scalings~\eqref{e:scalings-2}, we now prove the convergence of $\Lambda^{N_{k}}$ to a solution to the replicator equation~\eqref{e:weak-replicator}. The proof is similar to that of Theorem~\ref{t:compactness-1}. Below we remark the main differences.

The compactness argument follows the line of the proof of Theorem~\ref{t:compactness-1}. We fix $D \subseteq [0, T]$ dense and at most countable. We select a not relabeled subsequence and, for every $t \in D$, a probability measure $\Lambda_{t} \in \Pp(\Pp(\Vv))$ such that $W_{1}(\Lambda^{N_{k}}_{t}, \Lambda_{t}) \to 0$. For every $s < t \in D$ we write~\eqref{e:approx_eq-3} and pass to the limit as $k \to \infty$. In view of~\eqref{e:scalings-2} we obtain that there exists $C >0$ independent of $s, t\in D$ such that for every $\psi \in C^{2} (\Pp(\Vv))$
\begin{align}
    \label{e:130-2}
    \int_{\Pp(\Vv)} \psi(\lambda) \, \di (\Lambda_{t} - \Lambda_{s}) (\lambda) \leq C \| D \psi\|_{L^{\infty} (\Pp(\Vv))}   |t - s|\,.
\end{align}
By Lemma \ref{l:density} it is possible to extend~\eqref{e:130-2} to functions $\psi \in {\rm Lip} (\Pp(\Vv))$ in the form
\begin{align}
    \label{e:130-3}
    \int_{\Pp(\Vv)} \psi(\lambda) \, \di (\Lambda_{t} - \Lambda_{s}) (\lambda) \leq C {\rm Lip} (\psi)   |t - s| \qquad \text{for $s < t \in D$}.
\end{align}
Then,~\eqref{e:130-3} implies that
\begin{align}
    \label{e:130-4}
    W_{1} (\Lambda_{s}, \Lambda_{t}) \leq C | t - s| \qquad \text{for every $s, t \in D$.}
\end{align}
With~\eqref{e:130-4} at hand, it is enough to repeat the argument of~\eqref{e:190000}--\eqref{e:190001} to construct a limit curve $t \mapsto \Lambda_{t}$ with values in $\Pp(\Pp(\Vv))$, simply by replacing $\di_{\mathcal{Z}}$ with the $W_{1}$-distance. 

The pointwise convergence of $\Lambda^{N_{k}}$ to $\Lambda$ is obtained reasoning as in~\eqref{e:134}--\eqref{e:136}. Precisely, for $t \in D$ we have that
$W_{1} (\Lambda^{N_{k}}_{t}, \Lambda_{t}) \to 0$ by construction. For $t \in [0, T] \setminus D$, let $t_{j} \in D$ be such that $t_{j} \to t$ and assume that $W_{1} (\Lambda_{t}^{N_{k}} , \overline{\Lambda}) \to 0$ for some $\overline{\Lambda} \in \Pp(\Pp(\Vv))$ (which holds up to a subsequence by compactness). We write~\eqref{e:approx_eq-3} for $t_{j}, t$ and $\psi \in C^{2} (\Pp(\Vv))$. Pass to the limit as $k \to \infty$, by~\eqref{e:scalings-2} we infer, similarly to~\eqref{e:130-2},
\begin{align}
    \label{e:180000}
    \int_{\Pp(\Vv)} \psi(\lambda) \, \di (\overline{\Lambda} - \Lambda_{t_{j}}) (\lambda) \leq C \| D \psi\|_{L^{\infty} (\Pp(\Vv))}   |t - t_{j}|\,.
\end{align}
for some constant~$C>0$ independent of~$\psi$, $t$, and~$j$. As before, we can extend~\eqref{e:180000} to tests $\psi \in {\rm Lip} (\Pp(\Vv))$ and thus conclude that
\begin{align*}
    W_{1} (\overline{\Lambda}, \Lambda_{t_{j}}) \leq C|t - t_{j}|\,.
\end{align*}
By construction of~$\Lambda_{t}$, this implies that $\overline{\Lambda} = \Lambda_{t}$ and $W_{1} ( \Lambda^{N_{k}}_{t} , \Lambda_{t}) \to 0$ for every $t \in [0, T]$. The Lipschitz continuity of $t \mapsto \Lambda_{t}$ with respect to the $W_{1}$-distance follows from~\eqref{e:130-4}.

Repeating the arguments of~\eqref{e:time-der-conv} and~\eqref{e:ineeditafter} we infer that $\Lambda$ is a weak solution of the replicator equation~\eqref{e:weak-replicator}. This concludes the proof of Theorem~\ref{t:compactness-3}.

\subsection{Proof of Theorem~\ref{t:compactness-3-strong}}
\label{sub:thm3-strong}

Let $N_{k} \in \mathbb{N} \setminus \{0\}$ such that $N_{k} \to +\infty$ as $k \to \infty$. Assume the scaling \eqref{e:scaling-strong}. Then the proof for the non-vanishing selection parameter case (Assumption \eqref{strong-sel-ass}) follows along the same lines as the proof of Theorem \ref{t:compactness-3} noting that \eqref{e:scaling-strong} implies 
$$
\lim_{k\to +\infty} \frac{1}{\tau_k N_k^2} = 0.
$$
It is worth emphasizing that the sole modification consists in checking the continuity from $(\Pp(\Vv), \| \cdot\|_{\rm BL})$ into $(E_{\Pp(\Vv)}, \| \cdot\|_{\rm BL})$ of the following function given by Corollary \ref{c:almost-equation}
$$
\lambda \mapsto \frac{w(\pi_\lambda-\overline{\pi}_\lambda)\lambda}{1+w(\overline{\pi}_\lambda-1)}.
$$
This property is a direct consequence of $(\pi.1)$ and of~\eqref{e:pi-bdd}. Recalling that $(\Pp(\Vv), \|\cdot\|_{\mathrm{BL}})$ is compact in $\mathcal{F}(\Vv)$, we further infer that
$$
\min_{\lambda\in\Pp(\Vv)} \left(1+w(\overline{\pi}_\lambda-1)\right) >0.
$$

\section*{Acknowledgment}

The authors are members of the Gruppo Nazionale per l'Analisi Matematica, la Probabilit\`a e le loro Applicazioni (INdAM-GNAMPA) and acknowledge the support of the INdAM-GNAMPA 2026 Project ``Sistemi multi-agente e replicatore: derivazione particellare e ottimizzazione'' CUP E53C25002010001. The work of SA was further supported by the FRA2022 Project ``ReSinAPAS'' and the STAR2024 project ``Effective Asymptotics of Microscopic and Nonlocal Interactions'' (CUP E63C22004420003). GO was partially supported by the Italian Ministry of University and Research under the Programme ``Department of Excellence'' Legge 232/2016 (Grant No.\ CUP - D93C23000100001).

\end{document}